\documentclass{amsart}
\usepackage{latexsym}
\usepackage{amsmath}
\usepackage{amssymb}
\usepackage[all,cmtip,2cell]{xy}
\UseTwocells
\usepackage{bracket}
\usepackage{autobreak}
\usepackage{enumitem}
\usepackage[hypertexnames=false]{hyperref} 
\usepackage[capitalize]{cleveref}
\crefname{equation}{}{}
\crefname{enumi}{}{}

\usepackage[\ifdefined\enabletodonotes\else disable\fi]{todonotes}

\usepackage{autonum} 

\usepackage{tikz}
\usetikzlibrary{positioning}
\usetikzlibrary{arrows.meta}

\numberwithin{equation}{section}
\newtheorem{thm}{Theorem}[section]
\newtheorem{prop}[thm]{Proposition}
\newtheorem{cor}[thm]{Corollary}
\newtheorem{lem}[thm]{Lemma}

\newtheorem{assertion}[thm]{Assertion}
\theoremstyle{definition}
\newtheorem{defn}[thm]{Definition}

\theoremstyle{remark}
\newtheorem{rem}[thm]{Remark}
\newtheorem{ex}[thm]{Example}

\newcommand{\K}{{\mathbb K}}
\newcommand{\Q}{{\mathbb Q}}
\newcommand{\F}{{\mathcal F}}
\newcommand{\C}{{\mathcal C}}
\newcommand{\D}{{\mathcal D}}
\newcommand{\e}{\varepsilon}
\newcommand{\R}{{\mathbb R}}
\newcommand{\mapright}[1]{%
 \smash{\mathop{%
  \hbox to 1cm{\rightarrowfill}}\limits_{#1}}}
\newcommand{\maprightd}[2]{%
 \smash{\mathop{%
  \hbox to 1.2cm{\rightarrowfill}}\limits^{#1}\limits_{#2}}}
\newcommand{\mapleft}[1]{%
 \smash{\mathop{%
  \hbox to 1cm{\leftarrowfill}}\limits_{#1}}}
\newcommand{\mapleftu}[1]{%
 \smash{\mathop{%
  \hbox to 0.8cm{\leftarrowfill}}\limits^{#1}}}
\newcommand{\maprightu}[1]{%
 \smash{\mathop{%
  \hbox to 1cm{\rightarrowfill}}\limits^{#1}}}
\newcommand{\maprightud}[2]{%
 \smash{\mathop{%
  \hbox to 0.5cm{\rightarrowfill}}\limits^{#1}_{#2}}}
\newcommand{\mapleftud}[2]{%
 \smash{\mathop{%
  \hbox to 1cm{\leftarrowfill}}\limits^{#1}_{#2}}}

\newcounter{eqn}[section]

\begin{document}

\title[Local systems in diffeology]{Local systems in diffeology
}

\footnote[0]{{\it 2020 Mathematics Subject Classification}: 
18F15, 58A35, 58A40, 55U10, 55P62,  55T99.
\\ 
{\it Key words and phrases.} Diffeology, simplicial set, local system, stratifold, fibrewise rationalization, de Rham complex, differential graded algebra, minimal model


Department of Mathematical Sciences, 
Faculty of Science,  
Shinshu University,   
Matsumoto, Nagano 390-8621, Japan   
e-mail:{\tt kuri@math.shinshu-u.ac.jp}
}

\author{Katsuhiko KURIBAYASHI}
\date{}
   
\maketitle

\begin{abstract} By making use of Halperin's local systems over simplicial sets and the model structure of the category of diffeological spaces due to Kihara, 
we introduce a framework of rational homotopy theory for such smooth spaces with arbitrary fundamental groups. As a consequence, we have an equivalence between the homotopy categories of fibrewise rational diffeological spaces and an algebraic category of minimal local systems elaborated by G\'omez-Tato, Halperin and Tanr\'e. 
In the latter half of this article, a spectral sequence converging to the singular de Rham cohomology of a diffeological adjunction space is constructed with the pullback
of relevant local systems. In case of a stratifold obtained by attaching manifolds, 
the spectral sequence converges to the Souriau--de Rham cohomology algebra of the diffeological space. By using the pullback construction, 
we also discuss a local system model for a topological homotopy pushout. 
\end{abstract}

\tableofcontents

\section{Introduction}

Diffeology invented by Souriau \cite{So} provides plenty of useful objects which are generalizations of manifolds. In fact, the category 
$\mathsf{Diff}$ of {\it diffeological spaces} is complete, cocomplete and cartesian closed. Therefore, pullbacks, adjunction spaces and mapping spaces are defined naturally as diffeological spaces with smooth structures. A notion of smooth homotopy groups of diffeological spaces is established in \cite{PI}; see also \cite{C-W} and \cite[Chapter 5]{IZ}. Smooth classifying spaces,  which play an important role in classifying diffeological bundles, are considered in \cite{M-W, C-W2020}. It is worth mentioning that a diffeological space is also regarded as a concrete sheaf; see \cite{B-H, W-W}. 
Thus, diffeology makes a fruitful field involving, moreover, manifolds with corners, orbifolds, stratifolds in the sense of Kreck \cite{Kreck}, infinite-dimensional manifolds and the notion of Chen's iterated integral; see \cite{ A-K, G-Z, IZ-K-Z, Kihara, K2020} for details and also  Appendix \ref{app:appA} for stratifolds.

Throughout this article, manifolds are assumed to be smooth, Hausdorff, second countable and without boundary unless otherwise specified. 
A significant categorical point of view is that the forgetful functor from the category $\mathsf{Mfd}$ of manifolds to the category $\mathsf{Top}$ 
of topological spaces factors through an embedding functor $m :\mathsf{Mfd} \to \mathsf{Diff}$. 
Therefore, if  homotopy theory is furthermore developed in 
$\mathsf{Diff}$ without forgetting smooth structures, we may have new ideas for the study of smooth homotopy invariants of manifolds and singular spaces.

Rational homotopy theory initiated by Quillen \cite{Quillen} and Sullivan \cite{Sullivan} presents remarkable advantages in computing homotopy invariants of topological spaces and manifolds. In the theory, geometric information is encoded algebraically in differential graded Lie algebras and 
commutative differential graded algebras (CDGA's for short); see, for example, \cite{FHT, FOT} for applications of such algebraic models. Thus, one might expect an advance of rational homotopy theory for diffeological spaces. 
For example, the calculus of differential forms (the de Rham calculus) for manifolds expands in $\mathsf{Diff}$ thanks to the Souriau--de Rham complex \cite{So}; see also  \cite[Chapters 6, 7, 8 and 9]{IZ} and \cite{P}. 
Thus, the calculus will be more accessible via algebraic models if rational homotopy theory becomes to be applicable to $\mathsf{Diff}$. 

One of two aims of this manuscript is to introduce a framework of rational (real) homotopy theory for diffeological spaces in which we may develop the diffeological de Rham calculus. In order to support the theory,  it is important to offer computational tools for invariants of diffeological spaces, for example, de Rham and singular cohomology algebras.  
Another aim is to construct a new spectral sequence converging to the singular de Rham cohomology \cite{K2020} of a diffeological adjunction space. 
To those ends, {\it local systems} over simplical sets with values in CDGA's introduced and developed by Halperin \cite{Halperin} are 
effectively used in diffeology.

For the reader, brief expositions of basics on  diffeology and a review of local systems \cite{Halperin, G-H-T} are contained in this manuscript; see 
Sections \ref{sect:diffspaces} and \ref{sect:local_systems}.

\subsection{An overview of rational homotopy theory for diffeological spaces} 
Very recently, Kihara \cite{Kihara, Kihara1} has succeeded in 
introducing a {\it Quillen model structure} on the category $\mathsf{Diff}$.  
The result \cite[Theorem 1.5]{Kihara} enables us to obtain a pair of Quillen equivalences  
\begin{equation}\label{eq:equivalence}
\xymatrix@C45pt@R10pt{
\mathsf{Set}^{\Delta^{\text{op}}}  
\ar@<1ex>[r]^-{ | \ |_D}
& \mathsf{Diff}  \ar@<1.2ex>[l]^-{S^D(\ \! )}_-{\bot}   \\
}
\end{equation}
between the category $\mathsf{Set}^{\Delta^{\text{op}}}$ of simplicial sets with the Kan--Quillen model structure \cite{Quillen_model}\cite[Chapter I,  Theorem 11.3]{G-J}
and $\mathsf{Diff}$ endowed with the model structure described in 
\cite[Theorem 1.3]{Kihara1}; see also Theorem \ref{thm:Kihara_Model}. Here $| \ |_D$ and $S^D( \ )$ denote  the {\it realization functor} and the {\it singular simplex functor}, respectively; see Section \ref{sect:Model} for the functors. Thus, 
rational homotopy theory for nilpotent diffeological spaces comes immediately via these model structures. 
In fact, by combining the result with the Sullivan--de Rham equivalence theorem \cite[9.4 Theorem]{B-G}, one has the following theorem. 

\begin{thm}\label{thm:nilpotentDiff} Let $\text{\em fN$\Q$-Ho}(\mathsf{Diff})$ be the full subcategory of the homotopy category $\text{\em Ho}(\mathsf{Diff})$ consisting of connected nilpotent rational diffeological spaces which correspond to connected nilpotent rational Kan complexes  
of finite $\Q$-type via the singular simplex functor $S^D( \ )$ in the diagram (\ref{eq:equivalence}). Then, 
there exists an equivalence of categories 
\[
\xymatrix@C15pt@R10pt{
\text{\em Ho}(\mathcal{M}\text{\em -alg}^{\text{op}}) \ar[r]^-{\simeq}&
 \text{\em fN$\Q$-Ho}(\mathsf{Diff}), 
}
\]
where $\mathcal{M}\text{\em -alg}$ denotes the category of minimal Sullivan algebras of finite $\Q$-type. 
\end{thm}


For Sullivan algebras, the reader is refered to \cite[Section 7]{B-G} and \cite[Section 12]{FHT}.  
We observe that for each diffeological space $M$, the simplicial set $S^D(M)$ is a Kan complex; see \cite[Lemma 9.4 (1)]{Kihara1}. 
Moreover, the result \cite[Theorem 1.4]{Kihara1} asserts that the homotopy groups of a pointed diffeological space $M$ in the sense of Iglesias-Zemmour \cite[5.15, 5.17]{IZ} (see also \cite[Theorem 3.2]{C-W}) are isomorphic to those of the simplicial set $S^D(M)$. 

We proceed to the consideration on localizations of more general diffeological spaces. 
In \cite{G-H-T}, G\'omez-Tato, Halperin and Tanr\'e constructed an algebraic rational homotopy theory for topological spaces 
with arbitrary fundamental groups by means of {\it minimal} local systems. 
In this manuscript, by relating the local systems over $K(\pi, 1)$-spaces to the model structure of $\mathsf{Diff}$, 
we propose a framework of rational homotopy theory for diffeological spaces to which we may apply a lot of computational machinery in the  theory for topological spaces. 

In order to describe our main result pertaining to localizations of diffeological spaces, we fix some notation and terminology. 
We say that a pointed connected Kan complex $X$ is {\it fibrewise rational} if the universal cover $\widetilde{X}$ of $X$ is rational; that is, 
$H_i(\widetilde{X} ; {\mathbb Z})$ is a vector space over $\Q$ for $i \geq 2$; see \cite[Definition 6.1]{G-H-T}. A fibrewise rational Kan complex $X$ is of {\it finite type} if 
$H_i(\widetilde{X}; {\mathbb Z})$ is a finite dimensional vector space over $\Q$ for each $i\geq 2$.  
We refer the reader to \cite{Iv,  Mo, R-W-Z} for comparisons between the fibrewise localization (see Section \ref{sect:local_systems}) and other localizations.
In particular,  the result \cite[Theorem 3]{R-W-Z} tells us that a map $f : (X, b)\to (Y, c)$ between pointed path-connected spaces is a fibrewise rational homotopy equivalence if and only if the induced map $C_*(X, b) \to C_*(Y, b)$ between the pointed normalized singular chains induces 
a quasi-isomorphism on the cobar construction. 

We call a pointed connected diffeological space $M$ {\it fibrewise rational} (of {\it finite type}) if so is the Kan complex $S^D(M)$. One of our main results is described as follows.  

\begin{thm}\label{thm:summary}
Let $\text{\em Ho}(\mathsf{Diff}_{*})$ be the homotopy category of pointed diffeological spaces and $\text{\em fib$\Q$-Ho}(\mathsf{Diff}_*)$ the full subcategory of $\text{\em Ho}(\mathsf{Diff}_{*})$ consisting of fibrewise rational connected diffeological spaces of finite type. Then, there exists an equivalence of categories 
\[
\xymatrix@C15pt@R10pt{
\text{\em Ho}({\mathcal M}_\Q) \ar[r]^-{\simeq}&
\text{\em fib$\Q$-Ho}(\mathsf{Diff}_{*}), 
 }
\]
where $\text{\em Ho}({\mathcal M}_\Q)$ is the homotopy category of {\em minimal local systems} introduced by G\'omez-Tato, Halperin and Tanr\'e in \cite{G-H-T}; see Section \ref{sect:local_systems} for the local system. 
\end{thm}



Unlike Theorem \ref{thm:nilpotentDiff}, the proof of the result above is not immediate.  
In fact, we prove Theorem \ref{thm:summary} by extracting skillfully the simplicial argument from the proof of \cite[Theorem 6.2]{G-H-T}; see 
Theorem \ref{thm:main1}. As a consequence, we see that the fibrewise rational homotopy theory for topological spaces factors through that for diffeological spaces; see Section \ref{sect:perspectives} for a precise discussion on this topic. 

Thanks to Theorem  \ref{thm:summary}, the homotopy classification problem of {\it cofibrant} diffeological spaces is equivalent to that of minimal local systems; see, for example, \cite{Hovey} for the general theory of model categories. We observe that all diffeological spaces are fibrant in $\mathsf{Diff}$ and hence in $\mathsf{Diff}_*$ ; see Theorem \ref{thm:Kihara_Model} and \cite[Proposition 1.1.8]{Hovey}.  

One might be interested in an example of an explicit minimal local system model for a diffeological space; see Section \ref{section:FR} for such a model.
While the result \cite[Theorem 3.12]{G-H-T} (Theorem \ref{thm:miminal_models}) indeed allows us to obtain a minimal local system model for a given diffeological space  in Theorem \ref{thm:summary}, 
Theorem \ref{thm:nilToMinimalLocalSystem} gives a comparably tractable minimal one for a nilpotent diffeological space. 
In fact, a {\it relative} minimal Sullivan algebra (\cite[Section 14]{FHT}) obtained by the Moore-Postnikov tower of a nilpotent diffeological space $M$ gives rise to a minimal local system for $M$  in 
Theorem \ref{thm:summary}; see the commutative diagram (\ref{eq:*-construction}) in Proposition \ref{prop:nilToMinimalLocalSystem}. 

For example, let $L^\infty M:= C^\infty(S^1, M)$ be the diffeological free loop space of a simply connected manifold $M$, which is the diffeological space of smooth maps from $S^1$ to $M$ endowed with the functional diffeology; see Example \ref{ex:manifolds_subsets_products} (1) and (4) for such a diffeology. 
In Example \ref{ex:LM}, we have an explicit minimal local system model  for 
$L^\infty M$ by using a minimal Sullivan model for a topological free loop space $L^cM$ endowed with the compact-open topology.



\subsection{A spectral sequence for a diffeological adjunction space}
Local systems over $K(\pi, 1)$-spaces are certainly key tools to construct the fibrewise localizations in \cite{G-H-T} and then in this manuscript. Moreover, since more general local systems are defined over arbitrary simplicial sets, we expect  that they will provide further applications. By using machinery in \cite[Chapters 12 and 14]{Halperin}, we also discuss applications of such local systems in diffeology and topology. 

Consider a commutative diagram 
\begin{equation}
\xymatrix@C30pt@R15pt{
X_1 \ar[dr]_-{p_1} & X_0 \ar[d]_-{p_0} \ar[l] \ar[r] & X_2\ar[dl]^-{p_2} \\
 & K &
}
\end{equation}
in $\mathsf{Set}^{\Delta^{\text{op}}}$ consisting of Kan fibrations $p_i$ for $i = 0, 1$ and $2$. Then, under a suitable condition on the diagram, 
the pullback ${\mathcal Q}$ of local systems obtained by $p_i$ enables us to construct a first quadrant spectral sequence of algebras 
of the form 
\[
E_2^{*,*} \cong H^*(K, {\mathcal H}_{\mathcal Q}^*) \Longrightarrow H^*(A_{DR}(X_1\cup_{X_0}X_2)). 
\]
Here ${\mathcal H}_{\mathcal Q}^*$ is the system of local coefficients associated with ${\mathcal Q}$ and $A_{DR}(X_1\cup_{X_0}X_2)$ denotes the singular de Rham complex of the adjunction simplicial set $X_1\cup_{X_0}X_2$; see Section \ref{sect:deRham} for the complex. 
The details on the spectral sequence are described in Theorem \ref{thm:SS}. Applying the result in the category $\mathsf{Diff}$, in particular, we have a spectral sequence converging to the Souriau--de Rham cohomology of 
a diffeological adjunction space such as a {\it smooth CW complex} \cite{I} and a {\it parametrized stratifold} \cite{Kreck}; see Examples \ref{ex:an_example2} and \ref{ex:an_example}. 
A small computation with Theorem  \ref{thm:SS} is made in Example \ref{ex:compare}. 

The spectral sequence in Theorem \ref{thm:SS} is equipped with a differential graded algebra structure, so that we may recover the target cohomology as an algebra 
unlike a Mayer--Vietoris type exact sequence. In this sense,  
the tool for computing the de Rham cohomology  seem to be a novel type of spectral sequences. For the Leray--Serre spectral sequence and the Eilenberg--Moore spectral sequence for fibrations in $\mathsf{Diff}$, 
we refer the reader to the results \cite[Theorems 5.4 and 5.5]{K2020}. 

More applications of the pullback of local systems in $\mathsf{Top}$ are described in Section \ref{sect:Top}. In particular, Theorem \ref{thm:NewAssertions} tells us 
that an appropriate local system for the homotopy colimit of $\xymatrix@C10pt@R10pt{\! X& C \ar[l]_i \ar[r]^j & Y}$ over a space in $\mathsf{Top}$ is obtained by the pullback construction of local systems for $X$, $C$ and $Y$. Thus we have a {\it local system model} for the homotopy colimit, which determines the fibrewise rational homotopy type of the space; see the end of Section \ref{sect:RHT_for_Diff} for the terminology `local system model'. 
As a corollary of Theorem \ref{thm:NewAssertions}, we have a spectral sequence converging to the singular cohomology of the homotopy colimit 
with coefficients in $\Q$ as an algebra; see Theorem \ref{thm:Top_SS}.  
Moreover, 
combining Theorem \ref{thm:NewAssertions} with the result on an explicit minimal local system for $L^c{\mathbb C}P^n$ considered in 
Section \ref{sect:LocalSystems}, we obtain a comparatively tractable local system for 
$\texttt{hocolim}(\xymatrix@C10pt@R10pt{\! L^c{\mathbb C}P^n& \Omega^c{\mathbb C}P^n \ar[l]_i \ar[r]^i & L^c{\mathbb C}P^n})$ over the singular simplex of 
$K(\pi_1(\Omega^c{\mathbb C}P^n), 1)$,  
where $i$ is the inclusion of the based loop space; see Example \ref{ex:TopLoopSp}. 

One might be interested in the same results in the category $\mathsf{Diff}$ as in Theorems \ref{thm:NewAssertions} and  \ref{thm:Top_SS}. 
While the notion of $h$-{\it cofibration} in the sense of Kihara \cite[Section 9.2]{Kihara} would be applicable in considering such a spectral sequence, we do not pursue the topic in this manuscript.

\subsection{Remarks and future perspectives}\label{sect:perspectives}
What is worth noting is that we can use  every {\it admissible} simplicial CDGA in developing non-simply connected rational (real) homotopy theory due to G\'omez-Tato, Halperin and Tanr\'e. Thanks to the highly versatility of the framework, the simplicial polynomial de Rham complex $(A_{PL}^*)_\bullet$ (see, for example \cite[Section 10 (c)]{FHT}) 
and also the simplicial singular de Rham complex $(A_{DR}^*)_\bullet$ work well in the study of rational (real) homotopy theory for diffeological spaces; see Sections \ref{sect:deRham} and \ref{sect:RHT_for_Diff} for the details. 

We describe an advantage of the use of such simplicial CDGA's.
For a simplicial CDGA $A^*_\bullet$ and a simplicial set $K$, we define a CDGA $A^*_\bullet(K)$ by  
\[
A^*_\bullet(K) := \mathsf{Set}^{\Delta^{\text{op}}}(K, A^*_\bullet), 
\]
where the CDGA structure is given by that of the target $A^*_n$ for each $n$. 
In what follows, we may write $A(K)$ for $A^*_\bullet(K)$.
While there is a quasi-isomorphism between 
$(A_{PL})(K)\otimes {\mathbb R}$ and 
$(A_{DR})(K)$ for each simplicial set $K$; see \cite[Corollary 3.5]{K2020}, a direct use of the simplicial de Rham complex for a
diffeological space brings {\it smooth structures} to real homotopy theory. 
In fact, if a given diffeological space $X$ is 
a smooth CW complex in the sense of Iwase \cite{I} or a parametrized stratifold (a {\it p-stratifold} for short) in the sense of Kreck \cite{Kreck}, then the CDGA 
$A_{DR}(S^D(X))$  
is quasi-isomorphic to the de Rham complex $\Omega^*(X)$ due to Souriau \cite{So} via the {\it factor map}, where $S^D (\ )$ denotes the singular simplex functor 
in (\ref{eq:equivalence}); see Section \ref{sect:deRham} and 
\cite[Theorem 2.4]{K2020} for the factor map and the assertion. 
Thus, when we consider a commutative algebraic model  over $\R$ for such a diffeological space $X$, its elements may be described 
in terms of {\it differential forms} in $\Omega^*(X)$ by which key features of $X$ are captured. 

To illustrate the point, we consider the unreduced suspension $\Sigma M$ of a closed manifold $M$, which is regarded as a p-stratifold. 
Propositions \ref{prop:suspensions} and \ref{prop:formsOnSigma M} yield commutative models for $\Sigma M$ described in terms of 
the de Rham complex $\Omega^*(M)$ and the differential $1$-form $dt$ in $\Omega^*(\R^1)$.  As a consequence, we see that all products of closed forms except the unit vanish in the de Rham cohomology of $\Sigma M$ because $dt\wedge dt = 0$.  While the vanishing property is well know for the topological suspension space, it is here stressed that the propositions above clarify the cause of the vanishing with the differential $1$-form. 

We observe that if a  diffeological space $X$ comes from a manifold via the embedding $m :\mathsf{Mfd} \to \mathsf{Diff}$ mentioned above, then the Souriau--de Rham complex of $X$ is isomorphic to the usual de Rham complex as a CDGA; see Remark \ref{rem:tautological_map} below. 
We refer the reader to 
\cite{K2023} for the topic on describing generators in the de Rham cohomology of a diffeological free loop space with differential forms via Chen's iterated integral map.

We relate Theorems \ref{thm:nilpotentDiff} and \ref{thm:summary} to rational homotopy theory for topological spaces.   
First, we consider a diagram  
\begin{equation}\label{eq:Diagram}
\xymatrix@C60pt@R15pt{
& & \mathsf{Sets}^{\Delta^{\text{op}}}   \ar@<1ex>[d]^-{| \ |_D}  \ar@<1.7ex>[rd]^-(0.6){| \ |}&  \\
\mathsf{Mfd} \ar[r]_{\text{\tiny fully faithful}}^j \ar@/^2.0pc/[rr]^{m : \text{fully faithful}} & \mathsf{Stfd} \ar[r]^-k &\mathsf{Diff} 
\ar@<0.7ex>[r]^-{D} \ar@<1.5ex>[u]^-{S^D}_{\vdash}  
& \mathsf{Top}. \ar@<1.0ex>[l]^-{C}_-{\bot} \ar@<-0ex>[ul]^-(0.7){S}_{\vdash} 
}
\end{equation}
This explains relationships among categories and functors which we deal with in this manuscript. Here $| \ |$ and $S$ denote the topological realization functor and the standard singular simplex functor between $\mathsf{Top}$ and $\mathsf{Sets}^{\Delta^{\text{op}}}$, respectively. The functor $j$ is recalled in Appendix \ref{app:appA} together with the fact that the composite $k\circ j$ coincides with the embedding $m$; see Example \ref{ex:manifolds_subsets_products} (1) for the functor $m$.  A brief review of the functors $C$ and $D$ is given in 
Section \ref{sect:diffspaces}; see \cite{C-W, Kihara1, Kihara, S-Y-H} for more details. 
As noted in the beginning of the introduction, the composite $D\circ m$ is nothing but the forgetful functor. 

The result \cite[Lemma 9.5]{Kihara} asserts that the adjoints $| \ |$ and $S$ between $\mathsf{Sets}^{\Delta^{op}}$ and 
$\mathsf{Top}$ factor through $\mathsf{Diff}$ with the functors 
described in the diagram (\ref{eq:Diagram}), respectively. Hence,  
in view of Theorem \ref{thm:nilpotentDiff}, we see that the minimal Sullivan model  for a nilpotent connected topological space $X$ whose homotopy type is of a CW-complex coincides with such a model  for the diffeological space $C(X)$. 

Moreover, pointed connected topological spaces, which are in the image of the realization functor $| \ |$ from the full subcategory 
$\text{fib$\Q$-Ho}(\mathsf{Set}^{\Delta^{\text{op}}}_{\ *})$, are rational spaces in the sense of G\'omez-Tato, Halperin and Tanr\'e \cite{G-H-T}. Thus, the equivalence in \cite[Theorem 6.2 (1)]{G-H-T} factors through that in Theorem \ref{thm:summary}; see also Remark \ref{rem:nil_top}. 
More precisely, the composite $| \ |_D \circ i \circ \circ\langle \  \rangle$ with the functors $i$ and $\langle  \ \rangle$ in the sequence (\ref{eq:embedding}) below gives rise to the equivalence in \cite[Theorem 6.2 (1)]{G-H-T}; see (\ref{eq:realization}) for the definition of the functor $\langle  \ \rangle$. 

In rational homotopy theory for topological spaces, the result \cite[Proposition 17.9]{FHT} asserts that the composite $| \ |\circ (A_{PL}^*)_\bullet$ assigns 
a locally trivial bundle to a relative Sullivan algebra. 
Thus, we can investigate topological fibrations algebraically; see also \cite[20.3 Theorem]{Halperin}.
However, the realization functor 
$| \ |_D$ does not preserve finite products in general, unlike the realization functor $| \ |$ to $\mathsf{Top}$; see \cite[Remark 2.21]{Kihara}. Then, it does not seem that the realization of a relative Sullivan algebra with functors $| \ |_D$ and $(A_{PL}^*)_\bullet$ gives a diffeological bundle in $\mathsf{Diff}$ similar to the topological case in general. We have to take care the difference of behavior when converting results in $\mathsf{Top}$ to those in $\mathsf{Diff}$ via 
$\mathsf{Sets}^{\Delta^{\text{op}}}$. 

We expect that in future work, the de Rham calculus and 
rational (real) homotopy theory in diffeology are developed with commutative algebraic models for diffeological spaces, which will be obtained by modifying such models for topological spaces; see Examples \ref{ex:Alg_models}, \ref{ex:LM} and Theorem \ref{thm:nilToMinimalLocalSystem} which illustrate the expectation. 

\subsection{The organization of the paper}
Section \ref{sect:Diff} gives a brief overview of concepts and materials that we use throughout this article.  
More precisely, 
after recalling the definition of a diffeological space, we summarize the Quillen equivalence between the category $\mathsf{Set}^{\Delta^{\text{op}}}$ of simplicial sets and 
$\mathsf{Diff}$ which is proved by Kihara in \cite{Kihara}. Moreover, the original de Rham complex due to Souriau, the singular one and the de Rham theorem for  diffeological spaces are reviewed briefly.  

In Section \ref{sect:RHT_for_Diff}, we recall carefully crucial results  on local systems in \cite{G-H-T},  
with which rational homotopy theory for non-simply connected spaces is developed. By combining those ingredients with Kihara's model category structure on $\mathsf{Diff}$,  
we make a framework of rational and real homotopy theory for diffeological spaces. Theorem \ref{thm:summary} is proved in this section. 

Section \ref{sect:LocalSystems} addresses examples of local systems associated with relative Sullivan algebras. 
In particular, we see that a relative Sullivan model for a fibration in $\mathsf{Set}^{\Delta^{\text{op}}}$ gives rise to the fibrewise localization of 
a diffeological space via the realization functor $| \ |_D$; see Proposition \ref{prop:FromKS-ex} and Example \ref{ex:Alg_models}.

In Section \ref{sect:SS}, we construct the spectral sequence mentioned above in Theorem \ref{thm:SS}.  
We also treat toy examples concerning computations of the spectral sequence. 
Examples \ref{ex:an_example} and \ref{ex:an_example2} explain 
that the spectral sequence is applicable to a smooth CW-complex and a p-stratifold.  
Moreover, we consider a local system 
in $\mathsf{Top}$. Then Theorems \ref{thm:NewAssertions} and  \ref{thm:Top_SS} are proved. 

In Appendix \ref{app:appA}, 
we recall a {\it stratifold} introduced in \cite{Kreck} and a functor $k$ from the category $\mathsf{Stfd}$ 
of stratifolds to $\mathsf{Diff}$ considered in \cite{A-K}. The main result of this section (Theorem \ref{thm:p-stfd}) asserts that the functor $k$ assigns an adjunction space in $\mathsf{Diff}$ to a p-stratifold up to smooth homotopy. 

Appendix \ref{app:appB} is devoted to constructing a commutative algebraic model for 
the unreduced suspension of a connected closed manifold. 
This model is used in Example \ref{ex:compare} when comparing two spectral sequences obtained by Theorem \ref{thm:SS}.

\section{The category $\mathsf{Diff}$ of diffeological spaces and de Rham complexes}\label{sect:Diff}

In this section, we summarize some background on diffeological spaces. A good reference for diffeology is the book \cite{IZ}. Moreover, we review the Souriau--de Rham complex, the singular de Rham complex in diffeology and the model structure on the category of 
diffeological spaces mentioned in Introduction. 

\subsection{Diffeological spaces} \label{sect:diffspaces} 
We begin by recalling the definition of a diffeological space. Afterward, some examples of diffeological spaces are given.  

\begin{defn}
Let $X$ be a set. A set  $\D$ of functions $U \to X$ for each open subset $U$ in ${\mathbb R}^n$ and for each $n \in {\mathbb N}$ 
is a {\it diffeology} of $X$ if the following three conditions hold: 
\begin{enumerate}
\item (Covering) Every constant map $U \to X$ for all open subset $U \subset {\mathbb R}^n$ is in $\D$;
\item (Compatibility) If $U \to X$ is in $\D$, then for any smooth map $V \to U$ from an open subset $V$ of ${\mathbb R}^m$, the composite 
$V \to U \to X$ is also in $\D$; 
\item
(Locality) If $U = \cup_i U_i$ is an open cover and $U \to X$ is a map such that each restriction $U_i \to X$ is in $\D$, 
then the map $U \to X$ is in $\D$. 
\end{enumerate}
\end{defn}

In what follows, we may call an open subset of $\R^n$ a {\it domain}. 
A {\it diffeological space} $(X, \D)$ is comprised of a set $X$ and a diffeology $\D$ of $X$. 
A map from a domain to $X$ and an element of a diffeology $\D$ are called a {\it parametrization} and a {\it plot} of $X$, respectively. 
Let $(X, \D^X)$ and $(Y, \D^Y)$ be diffeological spaces. A map $X \to Y$ is {\it smooth} if for any plot $p \in \D^X$, the composite 
$f\circ p$ is in $\D^Y$. 
All diffeological spaces and smooth maps form a category $\mathsf{Diff}$. 

\begin{ex}\label{ex:manifolds_subsets_products}  (1) 
Let $M$ be a manifold. Then, the underlying set and the {\it standard diffeology} $\D^M$ give rise to a diffeological space $(M, \D^M)$, where 
$\D^M$ is defined to be the set of all smooth maps $U \to M$ from domains to $M$ in the usual sense. We have a functor $m : \mathsf{Mfd} \to \mathsf{Diff}$ from the category of manifolds defined by $m(M) = (M, \D^M)$. This functor is fully faithful; see, for example, \cite[2.1 Example]{B-H}. 

(2) For a diffeological space $(X, \D^X)$ and a subset $A$ of $X$, we define $\D^X(A)$ by 
$\D^X(A):= \{ p : U \to A \mid U \ \text{is a domain and} \ i\circ p \in \D^X\}$, where $i : A \to X$ is the inclusion. Then, the set $\D^X(A)$ is a diffeology of $A$, which is called the 
{\it sub-diffeology}. We call $(A, D^X(A))$ a 
{\it diffeological subspace} of $X$. 

(3) 
Let $\{(X_i, \D_i) \}_{i \in I}$ be a family of diffeological spaces. Then, the product $\Pi_{i\in I}X_i$ has a diffeology $\D$, called the {\it product diffeology}, defined to be the set of all parameterizations $p :  U \to \Pi_{i\in I}X_i$ such that $\pi_i\circ p$ are plots of $X_i$ for each $i \in I$, where $\pi_i : \Pi_{i\in I}X_i \to X_i$ denotes the canonical projection.

(4) Let $(X, \D^X)$ and $(Y, \D^Y)$ be diffeological spaces. We consider the set $C^\infty(X, Y)$ of all smooth maps from $X$ to $Y$. The {\it functional diffeology} is defined to be 
the set of parametrizations $p : U \to C^\infty(X, Y)$ whose adjoints $ad(p) : U \times X \to Y$ are smooth.  
\end{ex}

A p-stratifold in the sense of Kreck \cite{Kreck} is regarded as an adjunction space endowed with the quotient diffeology described below; see Theorem \ref{thm:p-stfd} for more details. 

\begin{ex}\label{ex:finaldiff}
(1) Let ${\mathcal F} : =\{f_i : Y_i \to X\}_{i\in I}$ be a set of maps from diffeological spaces $(Y_i, \D^{Y_i})$ ($i \in I$) to a set $X$. 
Then a diffeology $\D^X$ of $X$ is defined to be the set of 
parametrizations $p : U \to X$ such that, for each $r \in U$, there exists an open neighborhood $V_r$ of $r$ in $U$ and a plot $p_i \in \D^{Y_i}$ with $f_i\circ p_i = p|_{V_r}$. 
We call $\D^X$ the {\it final diffeology} of $X$ with respect to ${\mathcal F}$. 

(2) For a family of diffeological spaces  $\{(X_i, \D_i) \}_{i \in I}$, the coproduct $\coprod_{i\in I}X_i$ has a diffeology $\D$, called the {\it sum diffeology}, defined by 
the final diffeology with respect to the set of canonical inclusions. 

(3) Let $(X, \D)$ be a diffeological space with an equivalence relation $\sim$. Then a difffeology of $X/\sim$, called the {\it quotient diffeology}, is defined by the final diffeology with respect to the quotient map $X \to X/\sim$. 
\end{ex}

With the constructions in Examples \ref{ex:manifolds_subsets_products}  and \ref{ex:finaldiff}, we see that the category $\mathsf{Diff}$ is complete, cocomplete and cartesian closed; 
see \cite[Theorem 5.25]{B-H}.

The category $\mathsf{Diff}$ is related to $\mathsf{Top}$ with adjoint functors. Let $X$ be a topological space. Then the {\it continuous diffeology} is defined by the family of continuous parametrizations $U \to X$. This yields a functor 
$C : \mathsf{Top} \to \mathsf{Diff}$. 

For a diffeological space $(M, \D_M)$, we say that a subset $A$ of $M$ is {\it D-open} if for every plot 
$p \in \D_M$, the inverse image $p^{-1}(A)$ is an open subset of the domain of $p$ equipped with the standard topology. The family of D-open subsets of $M$ defines a topology of $M$. Thus, by giving the topology to each diffeological space, we have a functor 
$D: \mathsf{Diff} \to \mathsf{Top}$ which is the left adjoint to $C$; see \cite{S-Y-H} for more details. The topology for a diffeological space $M$ is called the {\it D-topology} of $M$. 
\subsection{A model structure on $\mathsf{Diff}$}\label{sect:Model} 
Let ${\mathbb A}^{n}:=\left\{(x_0, ..., x_n) \in {\mathbb R}^{n+1} \middle| \ \sum_{i=0}^n x_i = 1 \right\}$
be the affine space 
equipped with the sub-diffeology of the manifold ${\mathbb R}^{n+1}$. 
Observe that ${\mathbb A}^{n}$ is diffeomorphic to the manifold ${\mathbb R}^n$ 
with the projection $p$ defined by $p(x_0, x_1, ...,x_n)  = (x_1, ..., x_n)$. We denote by $\Delta^n_{\text{sub}}$ 
the diffeological space, whose underlying set is the $n$-simplex $\Delta^n$, 
equipped with the sub-diffeology of the affine space ${\mathbb A}^{n}$. 

In order to explain a model category structure on $\mathsf{Diff}$, we recall the {\it standard simplices}. Let $\Delta^n$ be the $n$-simplex and $d^i : \Delta^{n-1} \to 
\Delta^n$ the affine map defined by skipping the $i$th vertex $v_i$ in the codomain. We define inductively the standard $n$-simplex $\Delta^n$ by the $n$-simplex endowed with the final diffeology with respect to maps $\{\varphi_i\}_{i=0, ..., n}$, where $\varphi_i : \Delta^{n-1}\times [0, 1) \to \Delta^n$ is a map defined by $$\varphi_i(x, t) = (1-t)v_i + td^i(x)$$ for each $i$; see Example \ref{ex:finaldiff} (1).  Observe that 
$\Delta^n = \Delta_{\text{sub}}^n$ for $n \leq 1$ by definition. 
We refer the reader to \cite[Section 1.2]{Kihara1} for details of the standard simplices and their axiomatic treatment. The $k$th {\it horn} $\Lambda_k^n$ is a subspace of $\Delta^n$ with the sub-diffeology 
defined by 
$\Lambda_k^n :=\{ (x_0, ..,x_n) \in \Delta^n \mid x_i = 0 \ \text{for some} \ i\neq k \}$. 

Moreover, the singular simplex functor $S^D( \ \! )$ and the realization functor $| \ |_D$ in \cite[1.2]{Kihara} are defined by 
\[
S^D_n(X): = {\mathsf{Diff}}(\Delta^n, X) \ \ \ \text{and} \  \ \ 
|K|_D := \text{colim}_{\Delta \downarrow K}\Delta^n
\]
for a diffeological space $X$ and a simplicial set $K$, 
respectively.
We are ready to introduce the model structure on  $\mathsf{Diff}$ due to Kihara. 

\begin{thm}\label{thm:Kihara_Model} {\em (\cite[Theorem 1.3]{Kihara1})}
Define  a map $f : X \to Y$ in $\mathsf{Diff}$ to be 
\begin{itemize}
\item[{\em (1)}]
a weak equivalence if $S^D(f) : S^D(X) \to  S^D(Y)$ is a weak equivalence in $\mathsf{Set}^{\Delta^{\text{\em op}}}$, 
\item[{\em (2)}]
a fibration if the map $f$ has the right lifting property with respect to the inclusions 
$\Lambda_k^n \to \Delta^n$ for all $n > 0$ and $0\leq k\leq n$, and 
\item[{\em (3)}]
a cofibration if the map $f$ has the left lifting property with respect to all maps that are both fibrations and weak equivalences.
\end{itemize}

\medskip
\noindent
With these choices, $\mathsf{Diff}$ is a compactly generated model category whose object is always fibrant. 
\end{thm}

While the construction of the standard simplex is somewhat complicated, 
each $k$th horn $\Lambda_k^n$ is a smooth deformation retract of the simplex $\Delta^n$; see \cite[Section 8]{Kihara}.
This is a significant part in proving Theorem \ref{thm:Kihara_Model}. 
On the other hand, the result \cite[Proposition A.2]{Kihara} yields that 
the $n$-simplex with the sub-diffeology of the affine space 
${\mathbb A}^n$ does {\it not} contain the $k$th horn as a retract. We observe that a weak equivalence, a fibration and a cofibration are different from those 
in \cite{C-W} defined by using the {\it affine} smooth singular simplex functor; see Remark \ref{rem:S^D} below.

As mentioned in the introduction, 
the result \cite[Theorem 1.5]{Kihara} asserts that the adjoint pair $(| \ |_D, S^D(\ ))$ in (\ref{eq:equivalence}) gives a Quillen equivalence between 
$\mathsf{Set}^{\Delta^{\text{op}}}$ and $\mathsf{Diff}$. Then, the result \cite[Proposition 1.3.17]{Hovey} allows us to obtain 
a pair of Quillen equivalences
\begin{equation}\label{eq:QE_based} 
\xymatrix@C45pt@R10pt{
\mathsf{Set}^{\Delta^{\text{op}}}_{\ *}  
\ar@<1ex>[r]^-{ | \ |_D}
& \mathsf{Diff}_*  \ar@<1.2ex>[l]^-{S^D(\ \! )}_-{\bot}   \\
}
\end{equation}
between the category $\mathsf{Set}^{\Delta^{\text{op}}}_{ \ *}$ of pointed simplicial sets with the pointed model structure and $\mathsf{Diff}_*$ the category of pointed diffeological spaces with the pointed model structure. 
Thus, in order to define a {\it fibrewise rational} diffeological space in $\mathsf{Diff}_*$, it suffices to work in the homotopy category of simplicial sets. 
We observe that for each diffeological space $X$, the simplicial set $S^D(X)$ is a Kan complex; 
see \cite[Lemma 9.4 (1)]{Kihara1}. 

\begin{rem}\label{rem:S^D}
The functor $S^D( \ )$ used above is different from the singular simplex functor $S^D( \ )_{\text{aff}}$ in \cite{C-W}.  
In constructing the latter functor, the affine space ${\mathbb A}^p$ is used instead of the standard simplex $\Delta^p$ mentioned above; that is, 
the simplex $S^D(X)_{\text{aff}}:=\{ S^D_p(X)_{\text{aff}}\}$ for a diffeological space $X$ is defined by $S^D_p(X)_{\text{aff}} =\mathsf{Diff}({\mathbb A}^p, X)$. 

However, the inclusion 
$\Delta^p \to {\mathbb A}^p$ induces a quasi-isomorphism between the singular chain complex for $S^D( \ )_{\text{aff}}$ and 
that for $S^D( \ )$; see the paragraph before \cite[Theorem 5.4]{K2020}, the comments in the end of 
\cite[Section 5]{K2020} and also \cite[Remark 3.9]{Kihara}. 
Very recently, Kihara \cite{Kihara2022} shows that the inclusion indeed gives rise to a weak equivalence $S^D(X)_{\text{aff}} \to S^D(X)$ in $\mathsf{Set}^{\Delta^{\text{op}}}$ for every diffeological space $X$. 
\end{rem}

\subsection{De Rham complexes}\label{sect:deRham}
This section serves as an introduction to the original de Rham complex due to Souriau \cite{So} and the simplicial de Rham complex \cite{K2020} for a diffeological space. 

Let $(X, \D^X)$ be a diffeological space. 
Given an open subset $U$ of ${\mathbb R}^n$, we define  $\D^X(U)$ by the set of plots with $U$ as the domain.  
Let $\Lambda^*(U)$ be the usual de Rham complex of $U$; that is, 
$\Lambda^*(U) = \{h : U \longrightarrow \wedge^*(\oplus_{i=1}^{n} {\mathbb R}dx_i ) \mid h \ \text{is smooth}\}$. Let $\mathsf{Open}$ denote the category consisting of open subsets of Euclidian spaces and smooth maps between them.  
We can regard $\D^X( \ )$ and $\Lambda^*( \ )$  as functors from $\mathsf{Open}^{\text{op}}$ to $\mathsf{Sets}$ the category of sets. 

A $p$-{\it form} is a natural transformation from $\D^X( \ )$ to $\Lambda^*( \ )$. Then, the complex $\Omega^*(X)$ called the {\it Souriau--de Rham complex} is defined to be the 
cochain algebra consisting of $p$-forms for $p\geq 0$; that is, $\Omega^*(X)$ is the direct sum of 
\[
\Omega^p(X) := \Set{
\xymatrix@C35pt@R10pt{
\mathsf{Open}^{\text{op}} \rtwocell^{\D^X}_{\Lambda^p}{\hspace*{0.2cm}\omega} & 
\mathsf{Sets} }
| \omega \ \text{is a natural transformation}
}
\]
with the cochain algebra structure induced by that of $\Lambda^*(U)$ pointwisely. 
The Souriau--de Rham complex is certainly a generalization of the usual de Rham complex of a manifold as mentioned below. 

\begin{rem}\label{rem:tautological_map} 
Let $M$ be a manifold and $\Omega_{\text{deRham}}^*(M)$ the usual de Rham complex of $M$. 
We recall the {\it tautological map} $\nu : \Omega_{\text{deRham}}^*(M) \to \Omega^*(M)$ defined by 
$$
\nu(\omega) = \{ p^*\omega\}_{p \in {{\mathcal D}^{M}}},
$$
where $\D^M$ is the standard diffeology of $M$. 
It follows from \cite[Theorem 2.1]{H-V-C} that $\nu$ is an isomorphism of CDGA's. 
\end{rem}

One might expect the de Rham theorem that gives an isomorphism between the de Rham cohomology $H^*(\Omega^*(X))$ and the singular cohomology $H^*(S^D(X); \R)$. 
However, we have to pay attention to the fact that, in general, the de Rham theorem does {\it not} hold for the Souriau--de Rham complex. In fact, for the irrational torus, which is recalled in Example \ref{ex:T}, there is no such an isomorphism; 
see \cite[Exercise 119]{IZ} and \cite[Remark 2.9]{K2020} for details. 

In \cite{I-I}, Iwase and Izumida proved the de Rham theorem for diffeological spaces 
introducing {\it cubic} de Rham complex instead of the Souriau--de Rham complex. We mention that the strategy is by means of a Mayer--Vietoris exact sequence. 
Here, we recall from \cite{K2020} another de Rham complex for our purpose.  

Let $(A_{DR}^*)_\bullet$ be the simplicial de Rham complex 
defined by $(A^*_{DR})_n := \Omega^*({\mathbb A}^{n})$ for each $n\geq 0$. This is used in developing real homotopy theory for diffeological spaces in the next section; see Remark \ref{rem:R-localdiffSp}.
Moreover, for a diffeological space $X$, we define a map 
\[
\alpha : \Omega^*(X) \to A_{DR}^*(S^D(X)_{\text{aff}}):= \mathsf{Set}^{\Delta^{\text{op}}}(S^D(X)_{\text{aff}}, (A_{DR}^*)_\bullet)
\]
of cochain algebras by $\alpha(\omega)(\sigma) = \sigma^*(\omega)$; see Remark \ref{rem:S^D} for the functor 
$S^D( \ )_{\text{aff}}$.  
We call the complex $A_{DR}^*(S^D(X)_{\text{aff}})$ and the map $\alpha$ the {\it singular} de Rham complex and the {\it factor map} for $X$, respectively. 

It is worth mentioning that the de Rham theorem for the singular de Rham complex holds 
in $\mathsf{Diff}$. More precisely, we have

\begin{thm}\cite[Theorem 2.4 and Corollary 2.5]{K2020} 
Let $X$ be a diffeological space $X$. Then there is a natural integration map from $A_{DR}^*(S^D(X)_{\text{\em aff}})$ to the cochain complex of $S^D(X)_{\text{\em aff}}$ with coefficients in $\R$ which induces an isomorphism of algebras on cohomology. 
\end{thm}

In view of Remark \ref{rem:S^D}, we also establish the de Rham theorem for the singular de Rham complex $A_{DR}^*(S^D(X))$ for {\it every} diffeological space $X$. 
Furthermore, the factor map behaves well for particular  diffeological spaces.

\begin{prop}\label{prop:factor_map}\cite[Theorem 2.4]{K2020} 
Let $X$ be a manifold, finite dimensional smooth CW complex or a diffeological space which comes from 
a $p$-stratifolds via the functor $k$ mentioned in Appendix \ref{app:appA}. Then {\em (i)} the factor map for $X$ is a quasi-isomorphism and {\em (ii)} the de Rham theorem holds for the Souriau--de Rham complex of $X$. 
\end{prop}

\section{A candidate for a localized diffeological space}\label{sect:RHT_for_Diff}
In order to develop rational homotopy theory for diffeological spaces, we need notions of fibrewise local objects, fibrewise localizations in $\mathsf{Diff}$ and their  algebraic models. The first one is described in the paragraph before Theorem \ref{thm:summary}. 
The second one is introduced formally in the beginning of the following subsection. Afterward, we verify that the procedure of the fibrewise localizations for  topological spaces due to 
G\'omez-Tato, Halperin and Tanr\'e \cite{G-H-T} works well also for diffeological spaces. 
Then, we investigate more explicitly fibrewise localizations of diffeological spaces together with their algebraic models by means of local systems. 

\subsection{Local systems}\label{sect:local_systems} 
For a pointed connected fibrant simplicial set (Kan complex) $X$, we have a fibration of the form $\widetilde{X} \to X \to K(\pi_1(X), 1)$ in which $\widetilde{X}$ 
is the universal cover of $X$; see, for example, \cite[III, 1 and VI, 3]{G-J}. 
We assume that $\widetilde{X}$ is of finite type with respect to rational cohomology; that is, 
$\text{dim}\ \!H^i(\widetilde{X}; \Q)< \infty$ for each $i \geq 0$.  
Let $X_{\mathbb Q}$ denote the {\it fibrewise rationalization} 
of $X$ in the sense of Bousfield and Kan \cite[Chapter I \S 8]{B-K}. By definition, the rationalization fits into  
the commutative diagram
\[
\xymatrix@C35pt@R12pt{
\widetilde{X}_{\mathbb Q} \ar[r] & X_{\mathbb Q} \ar[r] & K(\pi_1(X), 1) \\ 
\widetilde{X} \ar[r]  \ar[u]^{\ell}  & X \ar[r] \ar[u] & K(\pi_1(X), 1) \ar@{=}[u]
}
\]
whose upper row is also a fibration and $\ell$ is the classical rationalization of the simply-connected simplicial set $\widetilde{X}$. 

Let $M$ be a pointed connected diffeological space. Then, the construction above is applicable to the simplicial set $S^D(M)$. 
We shall call the realization 
$|S^D(M)_{\mathbb Q}|_D$ the {\it fibrewise rationalization} of $M$ and denote it by 
$M_{\mathbb Q}$. 

Let $\pi^D_i(N)$ denote the $i$th smooth homotopy group of a diffeological space $N$; see \cite[Chapter 5]{IZ} and \cite[Section 3.1]{C-W}.
The model structure on $\mathsf{Diff}$ mentioned in Section \ref{sect:Model} implies that each diffeological space is fibrant and hence the counit 
$\varepsilon : |S^D(M)|_D \to M$ is a weak equivalence; see \cite[Proposition 1.3.12]{Hovey}. Therefore, the result \cite[Theorem 1.4]{Kihara1} allows us to deduce that 
$\pi_1^D(M_{\mathbb Q})\cong \pi_1^D(M)$ and 
$\pi_i^D(M_{\mathbb Q})\cong \pi_i^D(M)\otimes {\mathbb Q}$ for $i\geq 2$. 

In order to introduce simplicial CDGA's for explicit localizations of diffeological spaces, more notations and terminology are needed. 

\begin{defn} \label{defn:extendableOb}
A simplicial object $A$ in a category $\C$ is {\it extendable} if for any $n$, every subset ${\mathcal I} \subset \{0, 1, ..., n\}$ and any elements 
$\Phi_i \in A_{n-1}$ for $i  \in {\mathcal I}$ which satisfy the condition that $\partial_i\Phi_j = \partial_{j-1} \Phi_i$ for $i <j$, there exists an element 
$\Phi \in A_n$ such that $\Phi_i = \partial_i \Phi$ for $i \in {\mathcal I}$.  
\end{defn} 

The following result describes an important property of an extendable simplicial object. 

\begin{prop}\text{\em (\cite[Proposition 10.4]{FHT})}\label{prop:extendability} If $A$ is an extendable simplicial object and $L\subset K$ is an inclusion of simplicial sets, the natural map $A(K) \to A(L)$ is surjective.
\end{prop}
 
In order to introduce  fibrewise localizations of topological spaces and diffeological spaces, we use an extendable simplicial CDGA which satisfies more involved conditions.  
Hereafter, let $\K$ denote a field of characteristic zero. A CDGA that we deal with is non-negatively graded. 
 
\begin{defn}\label{defn:admissible}{(\cite[Definition 3.2]{G-H-T})}
A simplicial CDGA  $A_\bullet$ is {\it admissible} if it satisfies the following conditions. 
\begin{enumerate}
\item[(i)] $A_0 = A_0^0 = \K$. 
\item[(ii)] For each $n$, there are elements $t_1, ..., t_n \in A_n^0$ such that 
$A_n = A_n^0\otimes_\K \wedge(dt_1, ..., dt_n)$. 
\item[(iii)] The simplicial algebra $H(A_\bullet)$ satisfies: $H^i(A_\bullet) = 0$ for $i\geq 1$, and $H^0(A_\bullet)$ is the constant simplicial algebra $\K$. 
\item[(iv)] The simplicial algebra $A_\bullet^i$ for $i\geq 0$ is extendable. 
\item[(v)] Let $n$ be a non-negative integer. If $f\in A_n^0$ satisfies $df = fw$ for some $w \in A_n^1$, then either $f=0$ or $f$ is invertible in $A_n^0$.  
\end{enumerate}
\end{defn}

For example, the simplicial polynomial de Rham complex  $(A_{PL}^*)_\bullet$ over ${\mathbb Q}$ is admissible. 
Moreover, 
the simplicial de Rham complex $(A^*_{DR})_\bullet$ mentioned in Section \ref{sect:deRham} is also admissible. These facts follow from 
\cite[Lemma 3.3]{G-H-T} and Remark \ref{rem:tautological_map}. We observe that the condition (v) in Definition \ref{defn:admissible} is used in proving the uniqueness of a {\it minimal $A$-algebra}; see the proof of \cite[Lemma 3.9]{G-H-T}. 
With the simplicial CDGA's $(A_{PL}^*)_\bullet$ and $(A^*_{DR})_\bullet$, we can form more explicitly a fibrewise localization of a diffeological space; see Theorem \ref{thm:A-localization} below. 
Our algebraic model for the fibrewise rationalization is constructed by making use of a {\it local system}. We recall the important notion. 

Let $\Delta[n]$ denote the standard $n$-simplex $\text{Hom}_{\Delta}( \ , [n])$, where $\Delta$ denotes the ordinal number category. 
For a simplicial set $K$, we may regard $K$ as a category whose objects are simplicial maps 
$\sigma : \Delta[n] \to K$ and whose morphisms $\alpha : \sigma \to \tau$ are the simplicial maps $\alpha : \Delta[\dim \sigma] \to \Delta[\dim \tau]$ with 
$\tau \circ \alpha = \sigma$.  

\begin{defn}\label{defn:LocalSystems}(\cite[Definition 1.1]{G-H-T} and \cite[12.15 Definition]{Halperin}) 
(1) A {\it local system} $E$ is a presheaf on a simplicial set $K$ with values in CDGA's.

(2) A morphism $\psi : E \to E'$ of local systems over $K$ is a morphism of presheaves whose image $\psi_\sigma : E_\sigma \to E_\sigma'$ is 
a morphisms of CDGA's for each $\sigma \in K$.  

(3) A morphism $\psi : E \to E'$ of local systems over $K$ is a {\it quasi-isomorphism} if for each $\sigma \in K$, 
the map $\psi_\sigma^* : H^*(E_\sigma) \to H^*(E_\sigma')$ induced by $\psi_\sigma$ is an isomorphism.  
\end{defn}

Let $A$ be an admissible simplicial CDGA, for example, $(A_{PL}^*)_\bullet$ and $(A^*_{DR})_\bullet$. In what follows, for a given simplicial set $K$,
we regard $A$ as a local system over $K$ defined by  
$A_\sigma := A_{\dim \sigma}$ for $\sigma \in K$. 

We recall the {\it global section functor} $\Gamma: \mathsf{CDGA}^{K^{\text{op}}} \to \mathsf{CDGA}$ defined by 
$\Gamma(E) = \mathsf{Set}^{K^{\text{op}}}(1, UE)$ with the forgetful functor $U : \mathsf{CDGA}^{K^{\text{op}}} \to \mathsf{Set}^{K^{\text{op}}}$; see \cite[12.16 Definition]{Halperin}. 
Observe that an element $\Phi$ in $\Gamma(E)$ 
assigns to each simplex $\sigma \in K$ 
an element $\Phi_\sigma \in E_\sigma$ and the function $\Phi$ is compatible with the face and degeneracy operations. 

Let $E$ be a local system over a simplicial set $K$ and $u : L \to K$ a simplicial map. Then, we define the {\it pullback} $E^u$ of $E$ by $(E^u)_\sigma = E_{u\circ \sigma}$ for $\sigma \in L$. 
We say that a local system $E$ is {\it extendable} if the restriction map 
$
\Gamma(E^\sigma) \to \Gamma(E^{\sigma \circ i})
$
is surjective for any simplicial map 
$\sigma : \Delta[n] \to K$, where $i : \partial \Delta[n] \to \Delta[n]$ denotes the inclusion.  

\begin{rem}\label{rem:Gamma}
It follows from \cite[12.27 Theorem]{Halperin} that the functor 
$\Gamma$ preserves quasi-isomorphisms between extendable local systems. We observe that for a simplicial map $\sigma : \Delta[n] \to K$, a morphism 
$\zeta : \Gamma(E^\sigma) \to E_\sigma$ of CDGA's defined by 
$\zeta(\Phi) = \Phi_{id : \Delta[n]\to \Delta[n]}$ is an isomorphism; see \cite[Definition 1.1]{G-H-T} for the contents in this and the previous two paragraphs. 
\end{rem}

A local system $E$ is called {\it locally constant} (resp. $1$-{\it connected}) if  $\alpha^* :  E_\tau \to E_\sigma$ is an isomorphism 
for each $\alpha : \sigma \to \tau$ in $K$ (resp. if $H^0(E_\sigma)=\K$ and $H^1(E_\sigma)=0$ for every $\sigma \in K$). For a local system  $E$, we define a local system $H(E)$ with trivial differentials by $H(E)_\sigma = H(E_\sigma)$ for $\sigma \in K$. We say that a local system $E$ is of {\it finite type} if 
$H^k(E_\sigma)$ is finite dimensional for each $k$ and 
$\sigma \in K$. 

\begin{defn}(\cite[Definition 3.4]{G-H-T})\label{defn:A-algebra}
An {\it $A$-algebra} $j_E : A \to E$ is a morphism of local systems for which $E$ is extendable and the system $H(E)$ is locally constant. An $A$-{\it morphism} is a morphism $\varphi : E \to E'$ of local systems such that $\varphi \circ j_E = j_{E'}$. 
\end{defn}

Following \cite[\S 3]{G-H-T}, we introduce a {\it minimal} $A$-algebra and recall an important  result on the notion. 

\begin{defn}\label{defn:minimalA}(\cite[Definitions 2.2 and 3.5]{G-H-T})
(1) A local system $(\wedge Y, D_0)$ with values in CDGA's is a $1$-{\it connected} $A^0$ {\it minimal model} if $Y$ is an $A^0$-module and there exists a $1$-connected Sullivan minimal algebra $(\wedge Z, d)$ such that, as differential graded $(A^0)^\sigma$-algebras, 
$(\wedge Y)^\sigma \cong (A^0)^\sigma\otimes_\K (\wedge Z, d)$ for $\sigma \in K$. The CDGA $(\wedge Z, d)$ is called a {\it representative Sullivan algebra} for $(\wedge Y, D_0)$. 

(2) An $A$-algebra $(A\otimes_{A^0}\wedge Y, D = \sum_{i\geq 0}D_i)$ with 
$D_i : A^*\otimes_{A^0}\wedge Y \to A^{*+i}\otimes_{A^0}\wedge Y$ is a $1$-{\it connected} $A$ {\it minimal model} if $(\wedge Y, D_0)$ is a $1$-connected $A^0$ minimal model. 
A representative Sullivan algebra for $(\wedge Y, D_0)$ is called a {\it representative Sullivan algebra} for the $A$ minimal model. 
\end{defn}

\begin{thm}\label{thm:miminal_models} {\em (}\cite[Theorem 3.12]{G-H-T}{\em )} Let  $j : A \to E$ be an $A$-algebra for which $E$ is $1$-connected and of finite type. 
Then there are a unique $1$-connected $A$ minimal model $(A\otimes_{A^0}\!\wedge Y, D)$ 
of finite type and an $A$-quasi-isomorphism 
$m : (A\otimes_{A^0}\!\wedge Y, D) \stackrel{\simeq}{\longrightarrow} E$. 
\end{thm}

While an $A$-algebra $E$ certainly admits a minimal model, the construction is very hard in general even if the cohomology $H^*(E)$ is determined. We investigate a tractable minimal model under an appropriate condition in the next section. 

Let $A$ be an admissible simplicial CDGA and ${\mathcal A}_K$ the category of $A$-algebras over a connected simplicial set $K$.   
Let ${\mathcal S}_K$ be the category of simplicial sets over $K$. 
A morphism from $p : X \to K$ to $q : Y \to K$ is defined by a map $f : X \to Y$ in $\mathsf{Set}^{\Delta^{\text{op}}}$ with $q\circ f = p$.  
Moreover, we define ${\mathcal K}_K$ to be the full subcategory of Kan fibrations over $K$ with simply connected fibres; see \cite[the last paragraph in page 1514]{G-H-T}.  
Then, we recall the adjoint functors   
\[
\xymatrix@C20pt@R10pt{
\langle \  \rangle : {\mathcal A}_K^{\text{op}} \ar[r] & {\mathcal S}_K  \  \ \text{and} \  \ 
{\mathcal F}(\ ) : {\mathcal K}_K \ar[r] & {\mathcal A}_K^{\text{op}} 
}
\]
from \cite[Section 5, page 1515]{G-H-T}. The realization functor $\langle \  \rangle$ is defined by 
\begin{equation}\label{eq:realization}
\langle (E, j)  \rangle_n = \{ (\varphi, \sigma) \mid \sigma \in K_n, \varphi \in {\mathcal A}_{\Delta[n]}(E^\sigma, A_{\Delta[n]})  \}
\end{equation}
for an $A$-algebra $j : A \to E$, where $E^\sigma$ stands for the pullback of the local system $E$ over $K$ along the map $\sigma : \Delta[n] \to K$. 
Observe that $\varphi$ is a natural transformation.  For a Kan fibration $p : X \to K$, we define a local system 
${\mathcal F}(X, p)$ to be 
\begin{equation}\label{eq:fibrations}
{\mathcal F}(X, p)_\sigma = A(X^\sigma) := \mathsf{Set}^{\Delta^{\text{op}}}(X^\sigma, A_\bullet)
\end{equation}
for $\sigma \in K$, where $p^\sigma : X^\sigma\to \Delta[n]$ is the pullback of $p$ along $\sigma: \Delta[n] \to K$; see also Theorem \ref{thm:adjoint} below for the adjointness of the functors. Moreover, the $A$-algebra structure $j : A\to {\mathcal F}(X, p)$ is given by $j_\sigma =A(p^\sigma) : A_\sigma \to {\mathcal F}(X, p)_\sigma$. 
We may write $\F(X)$ for $\F(X, p)$ if the map $p$ is clear in the context. 

Thanks to Theorem \ref{thm:miminal_models}, we have a simply-connected 
$A$ minimal model of the form $(A\otimes_{A^0}\!\wedge Y, D) \stackrel{\simeq}{\longrightarrow} {\mathcal F}(X, p)$ 
for each object $p : X \to K$ in ${\mathcal K}_K$. 
Let ${\mathcal M}_K$ be the full subcategory of ${\mathcal A}_K$ consisting of  
$1$-connected $A$ minimal models of finite type and $\K$ the underlying field. 
Following the definition in \cite[Section 4]{G-H-T}, we introduce the homotopy relation between morphisms in ${\mathcal M}_K$. 
Let $I$ be the constant local system over $K$ defined by $I_\sigma = \wedge (t, dt)$, where $\deg t = 0$. 
Moreover, we define  $\varepsilon_0, \varepsilon_1 : I \to \K$ by $\varepsilon_i(t)= i$. 

\begin{defn} (\cite[Definition 4.1]{G-H-T})\label{defn:homotopy} (1) The {\it cylinder} for an $A$-algebra $E$ is the $A$-algebra $E\otimes_{\K} I$  with the retractions 
$\varepsilon_i = E\otimes_{\K}\varepsilon_i : E\otimes_{\K}I \to E$ for $i=0, 1$. 

(2) Let $\varphi_0, \varphi_1 : (A\otimes_{A^0}\wedge Y, D) \to E$ be morphisms of $A$-algebras with the source in ${\mathcal M}_K$. 
Then $\varphi_0$ and $\varphi_1$ are {\it homotopic} over $K$, denoted $\varphi_0 \simeq_K \varphi_1$, if there is a morphism 
$\Phi : (A\otimes_{A^0}\wedge Y, D) \to E\otimes_{\K}I$ of $A$-algebras such that $\varepsilon_i\Phi = \varphi_i$ for $i = 0, 1$. 
\end{defn}

The result \cite[Lemma 4.2]{G-H-T} shows that the homotopy relation above is an equivalence relation.

In what follows, we describe a procedure of the fibrewise localization for a diffeological space 
with an admissible simplicial CDGA $A$. That is completed in Section \ref{section:FR} below.
Let $M$ be a pointed connected diffeological space. The result \cite[Lemma 9.4.(1)]{Kihara1} yields that $S^D(M)$ is a Kan complex. 
Thus we have a Kan fibration 
\[
p : S^D(M) \to S^D(M)(1) = K(\pi_1(M), 1)
\] 
induced by the Moore-Postnikov tower 
$\{S^D(M)(n)\}_{n\geq 0}$ for $S^D(M)$; see \cite[VI, 3]{G-J} and \cite[Lecture 11]{J}.  
Then, the fibre of $p$ is regarded as the universal cover $\widetilde{S^D(M)}$ 
of $S^D(M)$. 

\begin{thm}\label{thm:A-localization} Under the setting mentioned above, 
suppose further that the $A$-cohomology $H^*(A(\widetilde{S^D(M)}))$ of the fibre $\widetilde{S^D(M)}$ of $p$ is of finite type. \\
{\em (1)} There is a $1$-connected $A$ minimal model 
$m: (A\otimes_{A^0}\wedge Y, D) \stackrel{\simeq}{\to} \F(S^D(M), p)$. \\
{\em (2)}  The minimal model $m$ in {\em (1)} gives rise to a fibrewise rationalization 
$ad(m) : S^D(M)  \to \langle (A\otimes_{A^0}\wedge Y, D) \rangle$ with the adjoint in Theorem \ref{thm:adjoint} below, 
which fits  into the commutative diagram
\[
\xymatrix@C30pt@R15pt{
\widetilde{S^D(M)}_{A} \ar[r] & \langle (A\otimes_{A^0}\wedge Y, D) \rangle \ar[r]^\pi & K(\pi_1(M), 1) \\ 
\widetilde{S^D(M)} \ar[r]  \ar[u]^{\ell}  & S^D(M)  \ar[r]_p \ar[u]_{ad(m)} & K(\pi_1(M), 1) \ar@{=}[u]
}
\]
consisting of two Kan fibrations $\pi$ and $p$, 
where $\widetilde{S^D(M)}_{A}$ denotes the  realization of the representative Sullivan algebra for $(A\otimes_{A^0}\wedge Y, D)$, $\ell$ is the usual localization and $\pi$ is the natural projection; see (\ref{eq:realization}) for the realization functor $\langle \  \rangle$. 
\end{thm}

\begin{proof} 
In view of \cite[Proposition 5.3 (2)]{G-H-T}, we see that the $A$-algebra $\F(S^D(M), p)$ is $1$-connected and of finite type. Then Theorem \ref{thm:miminal_models} yields the assertion (1). 
By virtue of \cite[Theorems 5.5 and 5.10]{G-H-T} and the adjointness in Theorem \ref{thm:adjoint} below, we have the assertion (2). 
\end{proof}


\subsection{The fibrewise localization of a diffeological space}\label{section:FR} Under the same notation as in Theorem \ref{thm:A-localization},  if 
$A$ is the simplicial polynomial de Rham complex $(A_{PL}^*)_\bullet$, 
then the realization produces the fibrewise rationalization 
\begin{equation}
\xymatrix@C30pt@R10pt{
f\ell :  |S^D(M)|_D \ar[r]^-{|ad(m)|_D} & | \langle  (A\otimes_{A^0}\wedge Y, D) \rangle |_D =:M_\text{fib$\Q$}
}
\end{equation}
of a given pointed connected diffeological space $M$. The $A$-algebra $(A\otimes_{A^0}\wedge Y, D)$  is called 
the {\it minimal \text{\em (}local system\text{\em )}} model for $M$. 
We observe that the counit  $|S^D(M)|_D \stackrel{\sim}{\to} M$ is a weak homotopy equivalence. In fact, by definition, each object in $\mathsf{Set}^{\Delta^{\text{op}}}$ is cofibrant. 
Moreover, every object in $\mathsf{Diff}$ is fibrant; see Theorem \ref{thm:Kihara_Model}. Then, the assertion follows from the fact that the pair $(| \ |_D, S^D( \ ))$ gives a Quillen equivalence; see \cite[Theorem 1.5]{Kihara}. 


\begin{rem}
The general theory of model categories enables us to deduce that the counit $|S^D(M)|_D\to M$ is a homotopy equivalence if $M$ has the smooth homotopy type of a  cofibrant object in $\mathsf{Diff}$. We remark that the class of such diffeological spaces contains the diffeomorphism group for
a compact manifold with a suitable smooth structure; see \cite[Theorem 11.1 and Corollary 11.22]{Kihara} for more details. 
\end{rem}

\begin{rem}\label{rem:weakEQ} Let $f : M \stackrel{\sim}{\to} M'$ be a weak equivalence between pointed connected diffeological spaces; see Theorem \ref{thm:Kihara_Model}. 
Then, the map $f$ induces a weak equivalence between their universal covers. Therefore, the representative Sullivan algebras of $1$-connected $(A_{PL})_\bullet$ 
minimal models for $M$ and $M'$ are 
quasi-isomorphic. It follows from \cite[Theorems 3.7 and 3.8]{G-H-T} that $1$-connected $(A_{PL})_\bullet$ minimal models for $M$ and $M'$ are isomorphic to each other and hence 
$M_\text{fib$\Q$}$ is diffeomorphic to $M'_\text{fib$\Q$}$ in $\mathsf{Diff}_*$.
\end{rem}

\begin{ex} \label{ex:T}
We recall  the irrational torus $T^2_\gamma$. Consider the two dimensional torus 
$T^2 :=\{(e^{2\pi i x} , e^{2\pi i y}) \mid (x, y) \in {\mathbb R}^2 \}$ which is a Lie group, and the subgroup 
$S_\gamma := \{ (e^{2\pi i t} , e^{2\pi i \gamma t}) \mid t \in {\mathbb R}\}$ of $T^2$, which is diffeomorphic to $\R$. 
Then, by definition, the {\it irrational torus} $T_\gamma^2$ is the quotient $T^2/S_\gamma$ with the quotient diffeology; see also \cite[Exercise 31]{IZ} for the definition. 

The result \cite[8.15]{IZ} enables us to get a diffeological bundle of the form $\R  \to T^2 \stackrel{\rho}{\to} T_\gamma^2$. 
The result \cite[Theorem 1.4]{Kihara} implies that the smooth homotopy groups of a diffeological space $M$ are naturally isomorphic to the homotopy groups of $S^D(M)$. 
By applying the smooth homotopy exact sequence 
(\cite[8.21]{IZ}) to the bundle, we see that the projection $\rho : T^2 \to T^2_\gamma$ is a weak equivalence in the sense of Theorem \ref{thm:Kihara_Model} (1). 
Therefore, the discussion in Remark \ref{rem:weakEQ} implies that $(T^2)_\text{fib$\Q$}$ is diffeomorphic to $(T^2_\gamma)_\text{fib$\Q$}$ in $\mathsf{Diff}_*$.

Since $T^2$ and $T^2_\gamma$ are indeed nilpotent diffeological spaces, we may consider the rationalizations of the spaces in the context of Theorem \ref{thm:nilpotentDiff}. 
There exists a quasi-isomorphism between the de Rham complex $A_{PL}^*(S^D(M))\otimes \R$ and  the singular de Rham complex 
$A_{DR}^*(S^D(M))$ for each diffeological space $M$; see \cite[Corollary 3.5]{K2020}.  
Then, the computation in \cite[Remark 2.9]{K2020} yields that both $H^*(A_{PL}^*(S^D(T^2_\gamma)))$ and $H^*(A_{PL}^*(S^D(T^2)))$ are isomorphic to 
the exterior algebra $\wedge (t_1, t_2)$, 
where $\deg \ \!t_i = 1$ for $i = 1$ and $2$. Thus the minimal model for $T^2_\gamma$ is isomorphic to that for $T^2$. It turns out that rationalizations 
$(T^2)_\text{$\Q$}$ and $(T^2_\gamma)_\text{$\Q$}$ are diffeomorphic to each other in $\mathsf{Diff}$.

It is worth mentioning that $T^2$ and $T^2_\gamma$ are {\it not} smooth homotopy equivalent. 
In fact, the de Rham cohomology functor $H_{\text{de Rham}}^*( \ )$ defined by Souriau--de Rham complex is a smooth homotopy invariant; see \cite[6. 88]{IZ}. 
While we have an isomorphism $H_{\text{de Rham}}^*(T^2) \cong \wedge (t_1, t_2)$, the computation in \cite[Exercise 105]{IZ} shows that $H_{\text{de Rham}}^*(T^2_\gamma) \cong \wedge (t)$, where 
$\deg \ \!t = 1$. 
\end{ex}

In the rest of this section, we address the proof of Theorem \ref{thm:summary}. In what follows, the simplicial CDGA $A$ that we use is the polynomial de Rham complex $(A_{PL}^*)_\bullet$ and the underlying field is rational,  
unless otherwise specified. 

We recall the category ${\mathcal M}$ defined in the proof of \cite[Theorem 6.2]{G-H-T}.  
Each object $(E_K, K)$ in ${\mathcal M}$ called a {\it minimal local system}
is a pair of a $K(\pi, 1)$-simplicial set $K$ in $\mathsf{Set}^{\Delta^{\text{op}}}_{\ *}$ 
and a $1$-connected $A$-minimal model $E_K$ over $K$. 
A morphism $(\varphi, u) : (E_K, K) \to (E_{K'}', K')$ in 
${\mathcal M}$ is a pair of a based simplicial map $u: K\to K'$ and a morphism $\varphi : E_K \leftarrow (E_{K'}')^u$ of $A$-algebras over $K$, where  $(E_{K'}')^u$ denotes the pullback of $E_{K'}'$ along $u$ mentioned in Section \ref{sect:local_systems}. 
The composition in ${\mathcal M}$ is defined naturally with the functoriality of pullbacks of local systems along simplicial maps. 
Observe that a morphism in ${\mathcal M}$ is a pair of a morphism in ${\mathcal M}_K^{\text{op}}$ and a morphism $K \to K'$ in 
$\mathsf{Set}^{\Delta^{\text{op}}}_{\ *}$. 

By definition, morphisms $(\varphi, u)$ and $(\psi, w)$ from $(E_K, K)$ to $(E_{K'}', K')$ are 
{\it homotopic}, denoted $(\varphi, u)\simeq (\psi, w)$, if $u=w$ and $\varphi \simeq_K \psi$ in the sense of Definition \ref{defn:homotopy}. The homotopy relation $\simeq_K$ is an equivalence relation and then so is the homotopy relation $\simeq$ in ${\mathcal M}$.
We define a functor $\langle \ \rangle : {\mathcal M} \to \mathsf{Set}^{\Delta^{\text{op}}}$ by 
$$\langle (\varphi, u) \rangle = \overline{u}\circ  \langle \varphi \rangle$$
for $(\varphi, u) : (E_K, K) \to (E_{K'}', K')$, where $\overline{u}$ is the projection 
$\langle (E_{K'}')^u\rangle = \langle E_{K'}'\rangle^u \to \langle E_{K'}'\rangle$. 
The result \cite[Proposition 5.8 (i)]{G-H-T} yields that the realization functor 
$\langle  \ \rangle :  {\mathcal M} \to \mathsf{Set}^{\Delta^{\text{op}}}$ preserves the homotopy relation $\simeq$. 
Thus, we have a sequence containing equivalences of homotopy categories and an embedding
\begin{equation}\label{eq:embedding}
\xymatrix@C35pt@R5pt{
\text{Ho}({\mathcal M}) \ar[r]^-{\langle \  \rangle} & \text{fib$\Q$-Ho}(\mathsf{Set}^{\Delta^{\text{op}}}_{\ *})
 \ar[r]^-{i}_-{\subset} &\text{Ho}(\mathsf{Set}^{\Delta^{\text{op}}}_{\ *})
\ar@<1ex>[r]^-{| \ |_D}_-{\simeq}  & \text{Ho}(\mathsf{Diff}_*).  \ar@<1.2ex>[l]^-{S^D(\ \!)}
}
\end{equation}
Here $\text{fib$\Q$-Ho}(\mathsf{Set}^{\Delta^{\text{op}}}_{\ *})$ stands for the full subcategory of pointed connected fibrewise rational Kan complexes of finite type; 
see the paragraph before Theorem \ref{thm:summary}. 
 We observe that the well-definedness of the functor $\langle \  \rangle$ in the sequence follows from \cite[Theorem 5.10 (i)]{G-H-T}. 
%
Moreover, a refinement of the simplicial argument in the proof \cite[Theorem 6.2]{G-H-T} indeed leads to the following result.

\begin{thm}\label{thm:main1}
The realization functor $\langle \  \rangle : \text{\em Ho}({\mathcal M}) \to \text{\em fib$\Q$-Ho}(\mathsf{Set}^{\Delta^{\text{op}}}_{\ *})$ gives an equivalence of categories.  
\end{thm}

\begin{proof} The faithfulness of the functor follows from the proofs of 
\cite[Theorem 5.11 and 6.2]{G-H-T}. 
In fact, let $X$ and $Y$ be  objects in $\text{fib$\Q$-Ho}(\mathsf{Set}^{\Delta^{\text{op}}}_{\ *})$ realized by $\langle \  \rangle$ with minimal local systems $(E, j_E)$ and $(E', j_{E'})$, respectively.  Suppose that $\langle f  \rangle = \langle g \rangle$ for maps $f$ and $g$ in 
$\text{Hom}_{{\mathcal M}^{\text{op}}}((E, j_E), (E', j_{E'})/\simeq_{\text{homotopic}}$; that is, $\langle f  \rangle \simeq \langle g \rangle$ in 
$\mathsf{Set}^{\Delta^{\text{op}}}_*$.  Let $u$ denote the map $\langle f  \rangle_* = \langle g \rangle_* : 
K:= K(\pi_1(X), 1)  \to K(\pi_1(Y), 1)$ induced by $\pi_1(\langle f  \rangle) = \pi_1(\langle g  \rangle)$. 
We consider a commutative diagram 
\[
\xymatrix@C15pt@R1pt{
X \ar[r]^{\varepsilon_i} \ar[dd]_p & X\times \Delta[1] \ar[rd]^-{pr} \ar[rr]^-h \ar[dd]_{p'}& & Y \ar[dd]^-{q}  \\
 & &  X \ar[dd]^(0.3){p}&   \\ 
K(\pi_1(X), 1) \ar@/_4mm/[rrd]_-{id} \ar[r]_-{\cong}^-{(\varepsilon_i)_*} & 
K(\pi_1(X\times \Delta[1]), 1) \ar[rd]_-{pr_*}\ar[rr]^(0.4){h_*}|(.54)\hole&   &  K(\pi_1(Y), 1) \\
  & &    K(\pi_1(X), 1)   \ar[ru]_-{u}&
}
\]
in which $h$ is the homotopy giving $\langle f  \rangle \simeq \langle g \rangle$. Then, since $h_*\circ (\varepsilon_0)_* = f_* = f_*\circ pr_*\circ (\varepsilon_0)_*$, 
 it follows that $h_* = f_*\circ pr_*$. This yields that $f_*\circ (p \circ pr) = f_*\circ pr_*\circ p' = h_*\circ p' = q\circ h$. Thus we have a commutative diagram 
\[
\xymatrix@C15pt@R15pt{
X\times \Delta[1]   \ar@/^4mm/[rrd]^h   \ar@/_4mm/[ddr]_{p\circ pr} \ar[rd]^-{h^u}& & \\
 & Y^u \ar[r]^{\overline{u}} \ar[d] & Y  \ar[d]^{q} \\
 &K(\pi_1(X), 1) \ar[r]_u & K(\pi_1(Y), 1), 
}
\]
where $Y^u$ is the pullback of $Y$ along $u$.
Therefore, we see that $h^u$ gives the homotopy relation $\langle f  \rangle^u \simeq_K \langle g \rangle^u : X \to Y^u$  and then $\langle f  \rangle^u =  \langle g \rangle^u$ in 
$\text{Hom}_ {{\mathcal K}_K}(X, Y^u)/\simeq$. Here  
$\varphi^u$ denotes the unique map from $X$ to $Y^u$ with $\varphi = \overline{u}\circ \varphi^u$ for a map $\varphi : X \to Y$ which satisfies the condition that 
$q\circ \varphi = u\circ p$. 
The result \cite[Theorem 5.11]{G-H-T} implies that the realization functor 
$\langle \  \rangle$ gives rise to an isomorphism 
\begin{equation}\label{eq:ISO}
\xymatrix@C20pt@R20pt{
\langle \  \rangle :  \text{Hom}_{{\mathcal M}_K}((E, j_E), ({E'}^u, j_{{E'}^u}))/\simeq_K \ar[r]^-{\cong} & 
\text{Hom}_{{\mathcal K}_K}(X, Y^u)/\simeq_K.
}
\end{equation}
Observe that $\langle (E, j_E) \rangle^u = \langle ({E}^u, j_{{E}^u}) \rangle$ for any $A$-algebra $(E, j_E)$. 
Moreover, it follows that $\langle f^u \rangle = \langle f \rangle^u$.  Thus it turns out that $f^u \simeq_K g^u$ and by definition, 
$f \simeq g$ in ${\mathcal M}$. 

We proceed to proving that the functor is full with the same notation above. 
Let $f$ be in $\text{Hom}_{\text{Ho}(\mathsf{Set}^{\Delta^{\text{op}}}_{\ *})}(X, Y^u)$. Then, since $Y^u$ is the pullback, it follows that  $f = \overline{u}\circ f^u$ for some 
$f^u : X \to Y^u$. We regard $f^u$ as a class of maps in $\text{Hom}_{{\mathcal K}_K}(X, Y^u)/\simeq_K$. Then, by applying the bijection 
(\ref{eq:ISO}), we have an $A$-morphism $\widetilde{f} : ({E'}^u, j_{{E'}^u}) \to  (E, j_E)$ with $\langle \widetilde{f} \rangle = f^u$. The definition of the functor 
$\langle \  \rangle : {\mathcal M} \to \text{Ho}(\mathsf{Set}^{\Delta^{\text{op}}}_{\ *})$ allows us to deduce that 
$\langle (\widetilde{f}, u) \rangle = \overline{u} \circ \langle \widetilde{f} \rangle = \overline{u} \circ f^u = f$. This implies that the functor $\langle \  \rangle$ is full. 

We show that the functor $\langle \  \rangle$ is essentially surjective.  Let $X$ be a fibrewise rational Kan complex. We consider the fibrewise localization 
$f\ell :  X \to \langle (E, D) \rangle$ mentioned in \cite[Theorem 5.10 (i)]{G-H-T}, which fits into the commutative diagram
\[
\xymatrix@C30pt@R15pt{
F_A  \ar[r] & \langle (E, D) \rangle \ar[r] &  K(\pi_1(X), 1) \\ 
\widetilde{X} \ar[r]  \ar[u]^-{\ell}  & X \ar[r]_-\pi \ar[u]_-{f\ell} & K(\pi_1(X), 1), \ar@{=}[u]
}
\]
where $\ell$ is the localization of the simply connected space; see also Theorem \ref{thm:A-localization}. 
Since $X$ is fibrewise rational, it follows that $\ell$ is a homotopy equivalence and hence $f\ell$ is a weak equivalence. Then the realization 
$\langle (E, D) \rangle$ is isomorphic to $X$ in  $\text{fib$\Q$-Ho}(\mathsf{Set}^{\Delta^{\text{op}}}_{\ *})$.  
Thus, we have the result. 
\end{proof}

\begin{proof}[Proof of Theorem \ref{thm:summary}]
The assertion follows from Theorem \ref{thm:main1} and the equivalence between homotopy categories induced by 
the pair of the Quillen equivalences in the diagram (\ref{eq:QE_based}). In fact, the equivalence is given by the composite $| \ |_D\circ \langle \ \rangle$ of the functors.
\end{proof}



\begin{rem}\label{rem:R-localdiffSp}  We can develop real homotopy theory in $\mathsf{Diff}$ by using the simplicial de Rham complex $(A_{DR}^*)_\bullet$ instead of 
$(A_{PL}^*)_\bullet$. Here, we call that simply-connected diffeological spaces $X$ and $Y$ are the same real homotopy type if $A_{PL}^*(S^D(X))\otimes_\Q\R$ and 
$A_{PL}^*(S^D(Y))\otimes_\Q\R$ are quasi-isomorphic. 
In fact, we define fibrewise real simplicial sets and diffeological spaces by a similar way as in the definition of fibrewise rational spaces  
replacing the underlying field $\Q$ to $\R$; see the paragraph before Theorem \ref{thm:summary}.
Since the proof of Theorem \ref{thm:summary} 
works well for $(A_{DR}^*)_\bullet$, 
the theorem remains true even if $\Q$ is replaced with $\R$.
We also have a real version of Theorem \ref{thm:nilpotentDiff}. 
\end{rem}

Under the setting of Theorem \ref{thm:A-localization}, 
for a pointed connected diffeological space $M$, we call an $A$-algebra $E$ which is quasi-isomorphic to $\F(S^D(M), p)$ a {\it local system model} for $M$. Observe that $1$-connected $A$ minimal model for $E$ determines the fibrewise rational (real) homotopy type of $M$. For a pointed connected topological space $X$, if an  $A$-algebra $E$ is a local system model for $C(X)$, then the $E$ is called a {\it local system model} for $X$; see Section \ref{sect:perspectives} for the functor 
$C : \mathsf{Top} \to \mathsf{Diff}$.


\section{Examples of localizations associated with local systems}\label{sect:LocalSystems}
We begin by recalling the result \cite[Proposition 5.3 (i)]{G-H-T} in a somewhat general setting. 
Given an admissible simplicial CDGA  $A$, such as $(A_{DR}^*)_\bullet$, let 
${\mathcal L}_K$ denote  the category of morphisms from $A$ to local systems over a connected simplicial set $K$ with values in CDGA's whose morphisms are $A$-morphisms described in Definition \ref{defn:A-algebra}. We observe that the category 
${\mathcal A}_K$ in Section \ref{sect:local_systems} is a full subcategory of ${\mathcal L}_K$. 

Recall the functor $\F( \ )$ and $\langle \  \rangle$ described in Section \ref{sect:local_systems}. In view of these definitions,  we see that 
$\langle \  \rangle$ and $\F( \ )$ give rise to functors $\langle \  \rangle : {\mathcal L}_K^{\text{op}} \to {\mathcal S}_K$ and 
$\F( \ ) : {\mathcal S}_K \to {\mathcal L}_K^{\text{op}}$, respectively. 
Moreover, we have
\begin{prop}\label{prop:DGA} {\em (}cf. \cite[Proposition 5.3 (i)]{G-H-T}{\em )} Let $K$ be a connected simplicial set. Let $p : X \to K$ be an object in the category ${\mathcal S}_K$ which is surjective. 
Then there exists a natural isomorphism $\Gamma(\F(X, p)) \cong A(X):= \mathsf{Set}^{\Delta^{\text{\em op}}}(X, A_\bullet)$ of  CDGA's.  Here $\Gamma (\ )$ denotes  the global section functor described in Section \ref{sect:local_systems}. 
\end{prop}

\begin{proof}
In the original assertion, it is assumed that the simplicial map $p$ is a Kan fibration. However, we have the result 
without the assumption. In fact, the map 
$v : A(X) \to \Gamma(\F(X, p))$ defined by $v(g)  =\left( i^*_\sigma(g)\right)_{\sigma \in K}$ for $g \in A(X)$ is an isomorphism, where $i_\sigma :  X^\sigma \to X$ is the map which fits into the pullback diagram 
\[
\xymatrix@C30pt@R15pt{
X^\sigma \ar[r]^{i_\sigma} \ar[d] &X  \ar[d]^p\\
\Delta[\dim \sigma] \ar[r]_\sigma & K
}
\]
of $p$ along $\sigma$. 
A direct calculation shows that a morphism $u :  \Gamma(\F(X, p))  \to A(X)$ defined by 
$u(f)_\tau = (f_{p(\tau)})(\text{id}_{[\dim \tau]}, \tau)$ for $f  \in  \Gamma(\F(X, p))$ is the inverse of $v$. 
\end{proof}


The proof of \cite[Theorem 5.4]{G-H-T} implies the following result. 

\begin{thm}\label{thm:adjoint}
For  an object $j_E : A \to E$ in ${\mathcal L}_K$ and a simplicial set $(X, p)$ over $K$ for which $p$ is surjective, 
there exists a natural bijection 
$
{\mathcal S}_K\left( (X, p), \langle E, j_E\rangle \right) \cong {\mathcal L}_K^{\text{op}} \left( \F(X, p), (E, j_E)\right). 
$
\end{thm}

\begin{proof} We define a map $\Phi : {\mathcal S}_K\left( (X, p) , \langle E, j_E\rangle \right) \to 
{\mathcal L}_K \left((E, j_E),  \F(X, p)\right)$ by 
\[
\left(\Phi(f)_\sigma\right)(u) ((\alpha, \tau)) := (f_\tau^1)_{\text{id}_{[\dim p(\tau)]}}(\alpha^*u),
\]
where  $\sigma : \Delta[\dim \sigma] \to K$ is an element of $K$, $u \in E_\sigma$, $(\alpha , \tau) \in X^\sigma$ and 
$f_\tau = (f_\tau^1, p(\tau))$. 
The inverse $\Psi$ of $\Phi$ is defined by 
$\psi(h)_\tau = ((\psi(h)_\tau^1, p(\tau)))$, where 
\[
\left(\Psi(h)_\tau^1\right)_\sigma(\phi) := h_{\sigma^*p(\tau)}(\phi)(\text{id}_{[\dim \sigma]}, \sigma^*\tau) 
\]
for $\tau \in X$, $\sigma \in \Delta[\dim p(\tau)]$ and $\phi \in (E^{p(\tau)})_\sigma= E_{\sigma^*p(\tau)}$. 
\end{proof}

In the rest of this section, we give two examples of local systems which come from suitable relative Sullivan algebras. 
The first one essentially explains the result \cite[Proposition 7.1 (ii)]{G-H-T}. As a consequence, 
we have the rationalization in Theorem \ref{thm:nilpotentDiff} of a given nilpotent diffeological space by using an $A$-minimal model. 

\begin{ex}\label{ex:relativeSullivanmodels} 
Let  $A$ be an admissible simplicial CDGA. 
We consider a relative Sullivan algebra $(\wedge W, \delta) \stackrel{i}{\to} (\wedge W\otimes \wedge Z, \delta) \stackrel{\rho}{\to} (\wedge Z, d)$ in which 
$\rho$ is the projection and $Z^1 = 0$. Then, we define a local system $(E, D)$ 
over the simplicial set  $\langle (\wedge W, \delta) \rangle^S$ by 
$(E_\sigma, D_\sigma) = A_n\otimes_\sigma (\wedge W\otimes \wedge Z)$ for 
$\sigma \in \langle (\wedge W, \delta) \rangle_n$. 
Here $\langle (\wedge W, \delta) \rangle^S := \mathsf{CDGA}(\wedge W, A_\bullet)$ denotes Sullivan's realization. 
By virtue of \cite[Propositions 5.5 and 7.1 (i)]{G-H-T}, we see that $(E, D)$ is an $A$ minimal model and there exists a 
commutative diagram 
\[
\xymatrix@C25pt@R15pt{
 \langle (\wedge Z, d) \rangle^S \ar[r] \ar[rd]_{\langle \rho \rangle}&  \langle (E, D )\rangle \ar[r]^\pi \ar[d]^{\xi}_{\cong} \ar[r]& 
 \langle (\wedge W, \delta) \rangle^S  \\
   &  \langle (\wedge W\otimes \wedge Z, \delta) \rangle^S \ar[ru]_{\langle i \rangle} & 
}
\]
for which the upper row is a Kan fibration, 
where $\pi$ is the natural projection of the $A$ minimal model and $\xi$ is an isomorphism of simplicial sets defined by $\xi((\varphi, \sigma))(a\otimes b) = \varphi(1\otimes a\otimes b)$. 
Suppose further that $W^i = 0$ if $i\neq 1$, it follows from 
\cite[8.12 Proposition]{B-G} that $\langle (\wedge W, \delta) \rangle^S$ is a $K(\pi, 1)$-space. 

Let $M$ be a connected nilpotent diffeological space of finite type and $(\wedge V, \delta)$ a minimal model for $A_{PL}^*(S^D(M))$. 
Then the model produces  
a relative Sullivan model of the form 
$(\wedge V^1, \delta) \to (\wedge V , \delta) \stackrel{\rho}{\to} (\wedge V^{\geq 2}, d)$. 
Applying the construction above to the sequence, we recover the rationalization 
$M \to M_\Q =| \langle (\wedge V, \delta) \rangle^S |_D \cong  | \langle (E, D) \rangle|_D$
in Theorem \ref{thm:nilpotentDiff} with the $A$-minimal model $(E, D)$, where $A = (A_{PL}^*)_\bullet$. 
\end{ex}

The second example (Example \ref{ex:Alg_models} below) deals with a fibrewise localization of $M$ over the simplicial set $K(\pi_1^D(M), 1)$ when $\pi_1^D(M)$ acts on $H^*(\widetilde{S^D(M)}; \Q)$ nilpotently.   
To explain that, we begin with a bit more general setting. 

Let $K$ be a simplicial set. 
We recall from \cite[18.1]{Halperin} the local system $R_*$ associated  with a relative Sullivan algebra (KS extension) $A(K) \stackrel{i}{\to} R \to (T, d_T)$. We observe that 
$R \cong A(K) \otimes T$ as an algebra. We assume further that $(T, d_T)$ is simply connected. 
For a simplex $\sigma : \Delta[n] \to K$, we define a CDGA 
$(R_*)_\sigma$ by 
\[
(R_*)_\sigma:= A(n)\otimes_{e_\sigma}R = A(n)\otimes_{e_\sigma, A(K)}(A(K)\otimes T),
\]
 where $e_\sigma : A(K) \to A(\Delta[n])=A(n)$ is the algebra map induced by $\sigma$ and the tensor product of CDGA's stands for the pushout of the diagram 
\[
\xymatrix@C10pt@R15pt{
A(n) & A(K) \ar[l]_{e_\sigma} \ar[r]^-i & A(K)\otimes T.
}
\] 
It follows that for $\sigma \in K$, the natural map $A_\sigma = A(n) \to A(n)\otimes_{e_\sigma, A(K)}R =(R_*)_\sigma$ induces a morphism $j : A \to R_*$ of local systems over $K$. 
By a usual argument, we have 
\begin{lem}\label{lem:SA-L} Let $K$ be a connected simplicial set. The construction above gives rise to a functor 
$(\ )_*$ from the category $\mathcal{SA}_{A(K)}$ of relative Sullivan algebras under $A(K)$ to the category $\mathcal{A}_K$ of $A$-algebras 
over $K$. 
\end{lem}

The following lemma is needed when we construct a more explicit minimal (local system) model for a given diffeological space. 

\begin{lem}\label{lem:minimality} Suppose that $(T, d_T)$ in the construction above is a simply-connected CDGA. Then, 
the map $j : A \to R_*$ is a $1$-connected $A$ minimal model. 
\end{lem}

\begin{proof} It follows from the result \cite[18.9 Proposition]{Halperin} that $j$ is an $A$-algebra over $K$; 
see Definition \ref{defn:A-algebra}. 
By the definition of $R_*$ with the tensor product mentioned above, we can write $R_* = (A\otimes_{A^0}\wedge Y, D = \sum_{i\geq 0}D_i)$ with $(\wedge Y)^\sigma = ({A^0})^\sigma \otimes (T, d_T)$ for $\sigma \in K$ and 
$D_i : A^*\otimes_{A^0}\wedge Y \to A^{*+i}\otimes_{A^0}\wedge Y$. This implies that $R_*$ is minimal in the sense of Definition \ref{defn:minimalA}. In particular, we have $D_0 = 1\otimes d_T$. 
\end{proof}


Let (*) $: F \to X \stackrel{\pi}{\to} K$ be a fibration with simply-connected fibre. 
Applying the realization functor $\langle \  \rangle : {\mathcal L}_K^{\text{op}} \to {\mathcal S}_K$ and 
the functor $( \ )_*$ mentioned in Lemma \ref{lem:SA-L} to an appropriate relative Sullivan model for $\pi$, 
we will obtain a fibrewise rationalization of $X$ in the fibration (*). To see this, let $v : \wedge V \stackrel{\simeq}{\to} A(K)$  be a Sullivan model and 
 $i : \wedge V \to \wedge V \otimes \wedge W$ a minimal Sullivan model for $\pi^*\circ v$. By definition, the map $i$ fits into the commutative diagram
\begin{equation}\label{eq:u}
\xymatrix@C30pt@R20pt{
\wedge V \ar[r]^-i \ar[d]_-{\simeq}^v&  \wedge V \otimes \wedge W \ar[r]^-\rho \ar[d]_-{\simeq}^\mu & \wedge W \ar[d]^-{\overline{\mu}}\\
A(K) \ar[r]_{\pi^*} & A(X) \ar[r] & A(F)   
}
\end{equation}
in which $\rho$ is the projection onto the quotient CDGA by the ideal generated by the image of $i$. 
It follows from \cite[Lemma 14.1 and Theorem 6.10]{FHT} that the natural map $v\otimes 1 : 
\wedge V\otimes_{\wedge V}\wedge V \otimes \wedge W \to A(K)\otimes _{\wedge V}\wedge V \otimes \wedge W$  is a quasi-isomorphism. Then, we have a commutative diagram 
\[
\xymatrix@C30pt@R20pt{
&\wedge V\otimes_{\wedge V}\wedge V \otimes \wedge W \cong \wedge V \otimes \wedge W \ar[d]_{v\otimes 1}^{\simeq} \ar[r]^-{\mu}_-\simeq & A(X) \\  
A(K)\otimes \wedge W \cong \hspace{-2cm} & A(K)\otimes _{\wedge V}\wedge V \otimes \wedge W \ar[ru]_-{\pi^*\cdot \mu=:\widetilde{\mu}} & A(K).\ar[u]_{\pi^*} \ar[l]^-{in}
}
\]
Therefore, it is readily seen that $\widetilde{\mu}$ defined to be $\pi^*\cdot \mu$ is a quasi-isomorphism and the inclusion $in : A(K) \to A(K)\otimes \wedge W$ is a Sullivan model for $\pi$.  
Let $\eta_X : X \to \langle A(X) \rangle^S$ be the unit, namely,  the adjoint to the identity on $A(X)$. We define a morphism $\theta : X \to \langle (A(K)\otimes \wedge W)_* \rangle$ of simplicial sets 
by $\theta(\sigma) = (\varphi_\sigma, \pi(\sigma))$, where 
$\sigma \in X$ and $\varphi_\sigma(u\otimes a\otimes b) := u\cdot (\eta_X(\sigma)\circ\widetilde{\mu}(a\otimes b))$. By a direct computation, we see that 
$\theta$ is a natural morphism with respect to relative Sullivan algebras under $A(K)$. 

The following proposition is essentially a rewriting of \cite[Proposition 7.1]{G-H-T} with the explicit map $\theta$. 

\begin{prop}\label{prop:FromKS-ex} Suppose that $\overline{\mu} : \wedge W \to A(F)$ in the diagram (\ref{eq:u}) is a quasi-isomorphism. Then in the commutative diagram
\[
\xymatrix@C30pt@R15pt{
\langle \wedge W \rangle^S  \ar[r] & \langle (A(K)\otimes \wedge W)_* \rangle \ar[r] & K \\ 
F \ar[r]  \ar[u]^-{\theta_{res}}  & X \ar[r]_-\pi \ar[u]_-\theta & K, \ar@{=}[u]
}
\]
the restriction $\theta_{res} : F \to \langle \wedge W \rangle^S$ is the rationalization and hence $\theta : X \to \langle (A(K)\otimes \wedge W)_* \rangle$ is 
the fibrewise rationalization. 
\end{prop}

\begin{proof} By virtue of Lemma \ref{lem:minimality}, we see that the local system $(A(K)\otimes \wedge W)_*$ is a $1$-connected $A$ minimal model. 
Then, it follows from \cite[Proposition 5.5]{G-H-T} that the sequence in the upper row is a Kan fibration. 
The rationalization of $F$ is given by $\overline{\mu}^*\circ \eta_F : F \to \langle A(F) \rangle^S \to\langle \wedge W \rangle^S$. For $w \in F$ and 
$b \in \wedge W$, we see that $\overline{\mu}^*\circ \eta_F(w)(b) = \overline{\mu}(b)(w)$. 
On the other hand, it follows that 
$\varphi_w(1\otimes 1\otimes b) = 1\cdot (\eta_X(w)\circ \widetilde{\mu})(1\otimes b) = \widetilde{\mu}^*\eta_X(w)(1\otimes b) = (j^*\overline{\mu}^*\eta_F(w))(b) = 
 \overline{\mu}(b)(w)$. Observe that $\theta(w) = (\varphi_w, \pi(w))$ by definition. This completes the proof. 
\end{proof}

\begin{cor} \label{cor:F}
The map $\F(\theta) : \F(\langle (A(K)\otimes \wedge W)_* \rangle) \to \F(X)$ induced by $\theta$ is a natural quasi-isomorphism with respect to relative Sullivan algebras under $A(K)$. 
\end{cor}

\begin{proof}
For a simplex $\sigma : \Delta[n] \to K$ and the vertex $v : \Delta[0] \to \Delta[n]$ which defines $F$, we have a commutative diagram 
\[
\xymatrix@C15pt@R18pt{
X^\sigma \ar[d]_{\theta^\sigma}& X^v \ar[l]_{\simeq} \ar[d]_{\theta^v}& F  \ar@{=}[l] \ar[d]^{\theta_{\text{res}}} \\
\langle (A(K)\otimes \wedge W)_* \rangle^\sigma & \langle (A(K)\otimes \wedge W)_* \rangle^v \ar[l]_-{\simeq}& \langle \wedge W \rangle^S. \ar[l]_-{\cong}
}
\]
Since $\theta_{\text{res}}$ is the localization, it follows that $A(\theta^\sigma) :  A(\langle (A(K)\otimes \wedge W)_* \rangle^\sigma) = 
\F(\langle (A(K)\otimes \wedge W)_* \rangle)_\sigma \to A(X^\sigma)=\F(X)_\sigma$ is a quasi-isomorphism. 
The naturality of $\theta$ yields that of $\F(\theta)$. 
\end{proof}

\begin{cor} \label{cor:F2}
The adjoint $ad(\theta) : (A(K)\otimes W)_* \to \F(X)$ is a natural quasi-isomorphism of $A$-algebras with respect to relative Sullivan algebras under $A(K)$. 
\end{cor}

\begin{proof} The naturality follows from that of $\theta$. Theorem \ref{thm:adjoint} enables us to obtain a commutative diagram 
\[
\xymatrix@C20pt@R18pt{
R_*:= (A(K)\otimes W)_* \ar[r]^-{ad(\theta)} \ar[d]_{ad(id_{R_*})}& \F(X) \\
\F(\langle (A(K)\otimes \wedge W)_* \rangle). \ar[ru]_-{\F(\theta)}
}
\]
For a simplex $\sigma : \Delta[n] \to K$ and a vertex $v : \Delta[0] \to \Delta[n]$, we have commutative squares 
\[
\xymatrix@C20pt@R18pt{
((A(K)\otimes \wedge W)_*)_\sigma \ar[r]^\simeq \ar[d]_{ad(id_{R_*})_\sigma}& ((A(K)\otimes \wedge W)_*)_v  \ar[d]_{ad(id_{R_*})_v} \ar[r]^-{\cong}& A(0)\otimes \wedge W \ar[d]_{ad} \\
A(\langle (A(K)\otimes \wedge W)_* \rangle^\sigma) \ar[r]_-\simeq &A( \langle (A(K)\otimes \wedge W)_* \rangle^v) \ar[r]_-{\cong}& A(\langle \wedge W \rangle^S),
}
\]
where $ad$ in the right-hand side denotes the unit map (the adjunction map) obtained by the adjointness of the functor $A$ and the realization functor. 
By \cite[10.1 Theorem (ii)]{B-G}, we see that $ad$ is a quasi-isomorphism and hence so is  $ad(id_{R_*})_\sigma$. This completes the proof. 
\end{proof}

\begin{ex} \label{ex:Alg_models} 
Let $A$ be the simplicial CDGA $(A_{PL}^*)_\bullet$ of polynomial forms. 
Let $M$ be a pointed connected diffeological space and $p : S^D(M) \to K(\pi_1^D(M), 1)=:K$ the fibration 
with the fibre $\widetilde{S^D(M)}$ of finite type 
such as in 
Theorem \ref{thm:A-localization}. 
By the result \cite[I, Theorem 10.10]{G-J} due to Quillen, we see that the topological realization $|\widetilde{S^D(M)}| \to |S^D(M)| \to |K|$ is a fibration with simply-connected fibre. 
Each $N$ of the three simplicial sets in the fibration $p$ is a Kan complex, so that the unit 
$N \stackrel{\simeq }{\to} S|N|$ is a homotopy equivalence; see also \cite[I, Lemma 3.4]{G-J}. Therefore, we have a diagram of the form (\ref{eq:u}) via a relative Sullivan model for the topological realization for $p$. 
Suppose further that 
\[
\text{ (P) :  $\pi_1^D(M)\cong \pi_1(S^D(M))$ acts on $H^*(\widetilde{S^D(M)}; \Q)$ nilpotently. }
\]
Then, the result \cite[20.3. Theorem]{Halperin} (see also \cite[Theorem 5.1]{FHTII}) implies that the map $\overline{\mu}$ in the diagram (\ref{eq:u}) is a quasi-isomorphism. 
By virtue of Proposition \ref{prop:FromKS-ex}, we have a fibrewise rationalization 
\[
\xymatrix@C15pt@R20pt{
f\ell :  |S^D(M) |_D \ar[r]^-{|\theta|_D}  & | \langle (A(K)\otimes \wedge W)_* \rangle |_D = M_\text{fib$\Q$}
}
\]
for which $|\theta_\text{res}| : |\widetilde{S^D(M)}|_D \to |\langle (\wedge W, \delta) \rangle^S|_D$ is the usual rationalization. 
By Lemma \ref{lem:minimality}, we see that $ (A(K)\otimes \wedge W)_*$ is a $1$-connected $A$ minimal model and then, by definition, it is the minimal local system model for $M$; 
see Section \ref{section:FR}. 
 \end{ex}

In what follows, we call a pointed connected diffeological space $M$ {\it nilpotent} if the topological realization $|S^D(M)|$ is a nilpotent space. 
As in Example \ref{ex:Alg_models}, a result in rational homotopy theory for topological spaces may be applicable for diffeological spaces via the topological realization functor. In fact, we have the following result. 




 \begin{thm}\label{thm:nilToMinimalLocalSystem}
Let $M$ be a pointed nilpotent diffeological space. 
Suppose further that $H^*(\widetilde{S^D(M)}; \Q)$ is locally finite.
Then the  $1$-connected $(A_{PL}^*)_\bullet$ minimal model $(A_{PL}^*(K)\otimes \wedge W)_*$  in Example \ref{ex:Alg_models} is a minimal local system model for $M$, where $K = K(\pi_1^D(M), 1)$. 
 \end{thm}

\begin{proof}
By virtue of \cite[Proposition 8.1]{FHTII}, we see that the diffeological space $M$ satisfies the condition (P) in Example \ref{ex:Alg_models}. In view of \cite[Theorem 7.2]{FHTII}, the result \cite[20.3. Theorem]{Halperin} is applicable to the diagram (\ref{eq:u}). 
Thus, the same argument as in Example \ref{ex:Alg_models} with Proposition \ref{prop:FromKS-ex} enables us to obtain the result. 
\end{proof}

The uniqueness of fibrewise rationalization allows us to deduce the following proposition. 
 
\begin{prop}\label{prop:nilToMinimalLocalSystem} Let $\text{\em f}\Q\text{-}\mathsf{Diff}_*$ be the category of pointed connected diffeological spaces  $M$ for each which the rational cohomology of 
the universal cover $\widetilde{S^D(M)}$ is of finite type and $\text{\em fN$\Q$-}\mathsf{Diff}_{*}$ the full subcategory of $\text{\em f}\Q\text{-}\mathsf{Diff}_*$ consisting 
of pointed nilpotent diffeological spaces. Then, one has a commutative diagram 
\begin{eqnarray}\label{eq:*-construction}
\xymatrix@C35pt@R18pt{
\text{\em fN$\Q$-}\mathsf{Diff}_{*} \ar[r]^{\text{\em sub}} \ar[d]_{( \ )_*\circ S^D( \ )}& \text{\em f}\Q\text{-}\mathsf{Diff}_{*} \ar[d]^{( \ )_{\text{\em  fib}\Q}} \\
  \text{\em Ho}({\mathcal M}_\Q) \ar[r]^-{\simeq}_-{| \ |_D\circ \langle \ \rangle}& \text{\em fib$\Q$-Ho}(\mathsf{Diff}_{*})
}
\end{eqnarray}
of categories and functors, 
where $\text{\em sub}$ denotes the embedding functor, the lower equivalence is that in Theorem \ref{thm:summary} and 
$( \ )_{\text{\em  fib}\Q}$ denotes the fibrewise rationalization; see Section \ref{section:FR}. 
 \end{prop}

\begin{proof}
By essentially the construction of a $1$-connected $A$ minimal model in \cite[Section 5]{G-H-T}, we have a functor 
$L : \text{fib$\Q$-Ho}(\mathsf{Diff}_{*}) \to  \text{Ho}({\mathcal M}_\Q)$ which makes the lower triangle commutative; see also Theorem \ref{thm:A-localization}. 
The result \cite[Theorem 4.3]{G-H-T} asserts that for a quasi-isomorphism $\varphi : F' \to F$ of $A$-algebras and a $1$-connected $A$ minimal model 
$\mathcal{E}$, the map $\varphi_\sharp : [\mathcal{E}, F'] \to [\mathcal{E}, F]$ between the homotopy sets induced by $\varphi$ is bijective. 
Therefore, Corollary \ref{cor:F2} implies that the upper triangle with $L$ is commutative. 
\end{proof}

\begin{rem}\label{rem:nil_top}
As mentioned in Section \ref{eq:Diagram}, we have $S^D\circ C = S$ for the singular simplex functor $S : \mathsf{Top}\to \mathsf{Sets}^{\Delta^{\text{op}}}$; see the diagram (\ref{eq:Diagram}). 
Therefore, for a pointed connected nilpotent CW complex $X$ whose rational cohomology is locally finite, we can apply Theorem \ref{thm:nilToMinimalLocalSystem} to the diffeological space $C(X)$. It turns out that $(A_{PL}^*(K)\otimes \wedge W)_*$ in Proposition \ref{prop:FromKS-ex} is a minimal local system model for $X$, where 
$K = K(\pi_1(X), 1)$. 
\end{rem}

\begin{ex}\label{ex:LM} We construct a minimal local system model for the diffeological free loop space of a simply connected manifold. 

We begin by recalling the smoothing theorem \cite[Theorems 1.1 and 1.7]{Kihara} due to Kihara for a particular case. 
Let $M$ and $N$ be diffeological spaces and $C^\infty(M, N)$ the space of smooth maps from $M$ to $N$ endowed with the functional diffeology. 
Let $D : \mathsf{Diff} \to \mathsf{Top}$ be the D-topology functor in the diagram (\ref{eq:Diagram}). 
Then, we see that for the standard $n$-simplex $\Delta^n$, $D(\Delta^n)$ is homeomorphic to $\Delta^n$ the $n$ simplex which is 
a subspace of $\R^{n+1}$; see Section \ref{sect:Model} for $\Delta^n$ in $\mathsf{Diff}$. 

Moreover, 
the inclusion map 
$i : D(C^\infty(M, N)) \to C^0(DM, DN)$ to the function space with the compact-open topology is continuous; see \cite[Proposition 4.2]{C-S-W}. Thus, it follows that the functor $D$ induces a natural morphism $\xi :  S^D(C^\infty(M, N)) \to S(C^0(DM, DN))$ of simplicial sets, where  $S^D( \ )$ and $S( \ )$ are  
the singular simplex functors in the diagram (\ref{eq:Diagram}). 

Let $M$ be a simply-connected manifold and $L^\infty M$ the diffeological free loop space $C^\infty(S^1, M)$.  
We construct the minimal local system model for 
$L^\infty M$ over $\Q$ applying the procedure in Example \ref{ex:Alg_models}.
For our case, the smoothing theorem implies that 
the map $\xi : S^D(L^\infty M) \to S(L^cM)$ is a weak homotopy equivalence, where $L^cM := \mathsf{Top}(D(S^1), D(M))$ is the free loop space.   We observe that $D(S^1)$ and $D(M)$ are nothing but the topological spaces $S^1$ and $M$ obtained by forgetting the smooth structures, 
respectively. 

Let $A$ be the simplicial CDGA $(A_{PL}^*)_\bullet$. 
Consider the fibration $\pi : S^D(L^\infty M) \to K(\pi_1^D(L^\infty M), 1)=:K$. Then, we can construct the diagram (\ref{eq:u}) for $\pi$ so that the vertical morphisms factor through the sequence 
\[
A(K(\pi_1(L^cM), 1)) \to A(S(L^cM)) \to  A(S(\widetilde{L^cM}))
\]
with the quasi-isomorphism induced by $\xi$. Moreover, without loss of generality, we may assume that $\wedge V$ is the minimal model of $K(\pi_1(L^cM), 1)$. Since the fundamental group of $L^cM$ is a finitely generated abelian group, it follows that $V = V^1$ with the trivial differential. By construction, the inclusion $i$ in the diagram (\ref{eq:u}) is a minimal Sullivan model for the map  $v\circ \pi^*$. Thanks to the minimality, we see that 
the CDGA $\wedge V\otimes \wedge W$ is a minimal model for $L^cM$. 
Since $L^cM$ is a nilpotent space, it follows from Theorem \ref{thm:nilToMinimalLocalSystem} that $(A(K)\otimes \wedge W)_*$ constructed in Example \ref{ex:Alg_models} is a minimal local system model for $L^\infty M$.

For instance, let $M$ be the complex projective space ${\mathbb C}P^n$. It is well known that a minimal model for $L^cM$ is of the form 
$(\wedge (x, y)\otimes \wedge(\overline{x}, \overline{y}), d)$ with $d(x) = d(\overline{x})=0$, $d(y) = x^{n+1}$ and $d(\overline{y}) = (n+1)\overline{x}x^n$, where $\deg x =2$, $\deg y = 2(n+1)-1$, $\deg \overline{x} =1$ and  $\deg \overline{y} = 2(n+1)-2$; see, for example, \cite[Section 15 (c), Example 1]{FHT}. Then, we have a minimal local system model for $L^\infty M$ over $K=K(\pi_1^D(L^\infty M), 1)$ of the form 
$R_* := (A(K)\otimes \wedge(x, y, \overline{y}))_*$
for which there exist isomorphisms 
\[
((R_*)_\sigma, d) \cong A(n)\otimes_{e_\sigma, A(K)}(A(K)\otimes \wedge(x, y, \overline{y}))\cong 
A(n)\otimes \wedge(x, y, \overline{y})
\] 
of CDGA's 
for each $\sigma \in K_n$, where $d(x) = 0$, $d(y) = x^{n+1}$ and 
$$d(\overline{y}) = (n+1)(e_\sigma \circ v)(\overline{x})\otimes x^n.$$

The strategy above works well to construct an explicit  minimal local system model for the free loop space $L^\infty M$ of 
a more general simply connected manifold $M$ if we get the Sullivan minimal model for $M$; see \cite[Section 15 (c), Example 1]{FHT} again. 
\end{ex}

\section{A local system for an adjunction space in $\mathsf{Diff}$ and $\mathsf{Top}$}\label{sect:SS}

By using the pullback of local systems over a common simplicial set, we construct a spectral sequence converging to the singular de Rham cohomology of a diffeological adjunction space. In this section, let $A$ be an admissible simplicial CDGA unless otherwise specified. 

Let 
$
\xymatrix@C15pt@R18pt{P & N \ar[l]_-f \ar[r]^-i & M}$ be morphisms between connected diffeological spaces. 
These maps produce the diffeological adjunction space 
$P\cup_{N}M$ in $\mathsf{Diff}$ together with the quotient diffeology with respect to the projection $p : P\coprod M \to P\cup_{N}M$, where $P\coprod M$ is endowed with the sum diffeology; see Example \ref{ex:finaldiff}. 
The universality of the pullback gives rise to a unique map 
\[
\Theta : A(S^D(P\cup_{N}M)) \to A(S^D(P))\times_{A(S^D(N))}A(S^D(M)).
\]
We assume that  $S^D(f) : S^D(N) \to S^D(P)$ and $S^D(i) : S^D(N) \to S^D(M)$ are in 
${\mathcal S}_K$ for a connected simplicial set $K$; that is, we have a commutative diagram  
\[
\xymatrix@C30pt@R15pt{
S^D(P) \ar[dr]_-{p_1} & S^D(N) \ar[d]_-{p_0} \ar[l]_-{S^D(f)}\ar[r]^{S^D(i)} & S^D(M). \ar[dl]^-{p_2} \\
 & K &
}
\]
Assume further that each $p_i$ is a surjective map. 
We consider a diagram 
\begin{equation}\label{eq:DGAs}
{\small 
\xymatrix@C4pt@R15pt{
 A(S^D(P))\times_{A(S^D(N))}A(S^D(M)) & &
       \Gamma(\F(S^D(P)))\times_{\Gamma(\F(S^D(N)))}\Gamma(\F(S^D(M))) \ar[ll]_-{u\times u}^-{\cong} \\
 A(S^D(P\cup_{N}M)) \ar[u]^-{\Theta} \ar[d]_-{\phi^*} && \Gamma\big(\F(S^D(P))\times_{\F(S^D(N))}\F(S^D(M))\big) \ar[u]_{\eta_{\Gamma}}^\cong\\
A(S^D(P)\cup_{S^D(N)}S^D(M)) \ar[rr]_{v}^{\cong} & & \Gamma \big(\F(S^D(P)\cup_{S^D(N)}S^D(M))\big) \ar[u]_{\Gamma(\eta_{\F})}^\cong
}
}
\end{equation}
consisting of morphisms of CDGA's, in which $u$ and $v$ are isomorphisms described in the proof of Proposition \ref{prop:DGA} and $\phi : S^D(P)\cup_{S^D(N)}S^D(M) \to S^D(P\cup_{N}M)$ is a unique map induced by $S^D(f)$ and $S^D(i)$. 
Here $\F$ denotes the functor defined by (\ref{eq:fibrations}) in Section \ref{sect:local_systems}; see also the beginning of Section \ref{sect:LocalSystems}.
The global section functor $\Gamma$ is right adjoint to the constant sheaf functor; see, for example, \cite[I, 6]{M-M}. Then, it follows that 
$\Gamma: \mathsf{Set}^{K^{\text{op}}} \to \mathsf{Set}$ preserves limits and hence $\eta_\Gamma$ induced by the functor is an isomorphism. Moreover, Theorem \ref{thm:adjoint} yields that the functor $\F$ preserves colimits. The fact enables us to conclude that the natural map $\eta_\F$ is an isomorphism in ${\mathcal L}_K$ and then it induces the isomorphism $\Gamma(\eta_\F)$ of CDGA's.

\begin{prop}\label{prop:commutative_D}
The diagram (\ref{eq:DGAs}) is commutative. 
\end{prop}
\begin{proof}
A direct calculation shows that $(v\times v)\circ \Theta = \eta_\Gamma \circ \Gamma(\eta_\F)\circ v \circ \phi^*$. 
\end{proof}

We construct a spectral sequence converging to the singular de Rham cohomology of the  diffeological adjunction space 
$P\cup_NM$.
To this end, we view the construction of the spectral sequence in a more general setting. 
Let $K$ be a connected simplicial set. 
We consider a commutative diagram 
\begin{equation}\label{eq:triangles}
\xymatrix@C30pt@R15pt{
X_1 \ar[dr]_-{p_1} & X_0 \ar[d]_-{p_0} \ar[l]_-i \ar[r]^j & X_2\ar[dl]^-{p_2} \\
 & K &
}
\end{equation}
in which each $p_i$ is surjective. 
We write $(X_1, X_0, X_2)_K$ for the diagram (\ref{eq:triangles}). A {\it morphism} $$(\psi_1, \psi_0, \psi_2)_{u} : (X_1, X_0, X_2)_K \to (X_1', X_0', X_2')_{K'}$$ between two triples is defined by maps $\psi_i : X_i \to X_i'$  and $u : K \to K'$ which are compatible with every composable maps.   

Our main theorem concerning the spectral sequence described in Introduction is as follows. 

\begin{thm}\label{thm:SS} In the commutative diagram (\ref{eq:triangles}), suppose that maps $p_0$, $p_1$ and $p_2$ are Kan fibrations over a pointed connected simplicial set $K$ and the map $i$ is injective.\\
\text{\em (1)} 
Then, there exists a first quadrant spectral sequence $\{E_r^{*,*}, d_r \}$ converging to the cohomology 
$H^*(A(X_1\cup_{X_0}X_2))$ as an algebra such that 
$E_2^{*,*} \cong H^*(K, {\mathcal H}_{\mathcal Q}^*)$
as a bigraded algebra, 
where ${\mathcal H}_{\mathcal Q}^*$ is the system of local coefficients associated with the local system 
$${\mathcal Q}:=\F(X_1)\times_{\F(X_0)}\F(X_2)$$
over $K$; see \cite[Chapter 14]{Halperin} for the cohomology with local coefficients. Moreover, for any $\sigma \in K$, one has an isomorphism 
$({\mathcal H}_{\mathcal Q}^*)_\sigma \cong H^*(A(F_1)\times_{A(F_0)}A(F_2) )$ with $F_i$ the fibre of $p_i$ for $i = 0, 1$ and $2$. \\
\text{\em (2)} The spectral sequence is natural with respect to triples $(X_1, X_0, X_1)$ of Kan fibrations in the sense that for any morphism, \[\Psi:=(\psi_1, \psi_0, \psi_2)_{id_K} : (X_1, X_0, X_2)_K \to (X_1', X_0', X_2')_{K}
\]
between triples with a common base simplicial set $K$, there exists a morphism of spectral sequence $\{\Psi_r\} : \{{}^\backprime E_r^{*,*}, {}^\backprime d_r\} \to \{E_r^{*,*}, d_r\}$. Here $\{{}^\backprime E_r^{*,*}, {}^\backprime d_r\}$  and $\{E_r^{*,*}, d_r\}$ denote the spectral sequences in (1) associated with the triples $(X_1', X_0', X_2')_{K}$ and  $(X_1, X_0, X_2)_K$, respectively. 
\end{thm}

\begin{rem} \label{rem:naturality}
The naturality of the spectral sequence in Theorem \ref{thm:SS} is generalized to cases of morphisms between two triples whose base simplicial set is different from each other.  In fact, for a triple 
${\mathcal T}:=(X_1', X_0', X_2')_K'$ and a simplicial map $u : K \to K'$, 
using the pullback $u^*p_i' \to K$ of each fibration $p_i'$ along $u$, we have a new triple $u^*{\mathcal T}'=(u^*p_1', u^*p_0', u^*p_2')_K$. Let  $\Psi:=(\psi_1, \psi_0, \psi_2) : {\mathcal T} \to {\mathcal T}'=(X_1', X_0', X_2')_{K'}$ be a morphism of triples. Then, we have a map ${\mathcal T} \to u^*{\mathcal T}'$ between triples with a common base simplicial set $K$ by means of the universality of the pullbacks. Therefore, Theorem \ref{thm:SS} (2) is applicable to the map ${\mathcal T} \to u^*{\mathcal T}'$. The consideration is applied in Example \ref{ex:compare} below.
\end{rem}

\begin{rem} (1) Let $F=\{F_\sigma\}_{\sigma \in K}$ be a local system with trivial differentials in which 
$\alpha^* :  F_\tau \to F_\sigma$ is an isomorphism for each 
$\alpha : \sigma \to \tau$ in $K$. Then, we see that the cohomology $H^*(K, F)$ is isomorphic to the cohomology of $K$ with 
the local coefficients $F_{\text{res}}=\{F_{\iota_\sigma^*\sigma} \}_{\sigma \in K}$ in the sense of Whitehead 
\cite[Chapter VI]{W}, where $\iota_\sigma : \Delta[0] \to \Delta[\dim \sigma]$ denotes the canonical inclusion for each $\sigma \in K$; 
see also \cite[Chapter 16]{Halperin} for the isomorphism. 


(2) If $X_1=X_0=X_2$ and the maps between them are the identity map, then the spectral sequence in Theorem \ref{thm:SS} is nothing but a simplicial version of 
the Leray--Serre spectral sequence with local coefficients converging to $H^*(A(X_1))$; see also \cite[Theorem 5.4]{K2020} for the spectral sequence of a fibration in $\mathsf{Diff}$ in the sense of Christensen and Wu \cite{C-W}.  
\end{rem}


Before proving Theorem \ref{thm:SS}, we give a toy example to which the spectral sequence is applicable.

\begin{ex}\label{ex:triangles}
Let 
$
\xymatrix@C15pt@R15pt{ M  & \ar[l]_-i  N \ar[r]^f & P }
$
be smooth maps between connected diffeological spaces. These maps give  rise to maps 
$
\xymatrix@C15pt@R15pt{S^D(M)  & S^D(N) \ar[l]_-{i_*}  \ar[r]^-{f_*}  & S^D(P)}
$ in $\mathsf{Set}^{\Delta^{\text{op}}}$. 
By considering the Moore--Postnikov tower for each singular simplex, we have a commutative diagram 
\[
\xymatrix@C15pt@R15pt{
S^D(M) \ar[d]_-{p_1} &S^D(N) \ar[d]_-{p_0} \ar[l]_-{i_*} \ar[r]^{f_*} & S^D(P)\ar[d]^-{p_2} \\
K(\pi_1(M), 1) & K(\pi_1(N), 1) \ar[l]^{i_K} \ar[r]_{f_K} & K(\pi_1(P), 1)
}
\]
in which $p_0$, $p_1$ and $p_2$ are Kan fibrations. The maps $i_*$ and $f_*$ factor through the bullbacks 
$S^D(M)^{i_K} \stackrel{p_1'}{\to} K(\pi_1(N), 1)$ and $S^D(P)^{f_K} \stackrel{p_2'}{\to} K(\pi_1(N), 1)$, respectively.
In particular, we have the decomposition $i_* = k \circ j : S^D(N) \to S^D(M)^{i_K} \to S^D(M)$. 
If $i$ is injective, then so is $j$. Thus, we have the triple $(p_1', p_0, p_2')$ of Kan fibrations over $ K(\pi_1(N), 1)$
to which we can apply Theorem \ref{thm:SS}.  
We observe that $S^D(M)^{i_K}$ is homotopy equivalent to $S^D(M)$ if $i$ induces the isomorphism on the fundamental groups. 
\end{ex}

\begin{proof}[Proof of Theorem \ref{thm:SS}] 
We recall the diagram (\ref{eq:triangles}). Then,  
we have a sequence of isomorphisms replacing $S^D(P)$, $S^D(N)$ and $S^D(M)$ in (\ref{eq:DGAs}) with the more general simplicial sets $X_1$, $X_0$ and $X_2$, respectively. 
As a consequence, the sequence allows us to deduce that 
$A(X_1\cup_{X_0} X_2) \cong A(X_1)\times_{A(X_0)}A(X_2)$.

We consider the local system $\mathcal{P}:= \F(X_1\cup_{X_0} X_2)$. Since the simplicial CDGA $A$, is extendable, it follows from the definition (\ref{eq:fibrations}) that $\mathcal{P}$ is extendable as a local system. 
We  prove that ${\mathcal P}$ is locally constant. As mentioned in the paragraph before Proposition \ref{prop:commutative_D}, Theorem \ref{thm:adjoint} tells us that the natural map 
\[
\eta_{\F} : \F(X_1\cup_{X_0} X_2) \stackrel{\cong}{\to} \F(X_1) \times_{\F(X_0)}\F(X_2) =: \mathcal{Q}
\] 
is an isomorphism in $\mathcal{L}_K$; see Section \ref{sect:LocalSystems}. 
Let $\alpha : \alpha^*\sigma = \tau \to \sigma$ be a morphism in $K$.  
We have a commutative diagram
\[
\xymatrix@C30pt@R15pt{
\F(X_1)_\sigma \ar@{->>}[r]^{i^*} \ar[d]_-{\alpha^*}^{\simeq}& \F(X_0)_\sigma \ar[d]_-{\alpha^*}^{\simeq}
 & \F(X_2)_\sigma \ar[l]  \ar[d]_-{\alpha^*}^{\simeq}\\ 
\F(X_1)_\tau \ar@{->>}[r]_{i^*}   & \F(X_0)_\tau & \F(X_2)_\tau. \ar[l]
}
\]
By assumption, the map $i$ is injective. Then, since $A$ is extendable, it follows that $i^*$ is an epimorphism; see  (\ref{eq:fibrations}) for the definition of the functor $\F$.  
By virtue of \cite[Lemma 13.3]{FHT}, we see  that the map 
$$(\alpha^*\times_{\alpha^*}\alpha^*) : {\mathcal Q}_\sigma =\F(X_1)_\sigma\times_{\F(X_0)_\sigma}F(X_2)_\sigma 
\to {\mathcal Q}_\tau$$ induced by $\alpha^*$ is a quasi-isomorphism. Then $H({\mathcal Q})$ and $H({\mathcal P})$ are locally constant. 
By applying \cite[Theorems 12.47 and 14.18]{Halperin} to the local system ${\mathcal P}$, we have a spectral sequence converging to 
$H^*(\Gamma(\mathcal{P}))$ with $E_2^{*,*} \cong H^*(K, {\mathcal H}_{{\mathcal P}}^*)$. 
We see that 
\begin{eqnarray*}
\Gamma(\mathcal{P}) \cong \Gamma({\mathcal Q})&\maprightud{\eta_\Gamma}{\cong} & \Gamma(\F(X_1)) \times_{\Gamma(\F(X_0))}\Gamma(\F(X_2))  \\
 &\cong & A(X_1)\times_{A(X_0)} A(X_2). 
\end{eqnarray*}
The second isomorphism follows from Proposition \ref{prop:DGA}; see (\ref{eq:DGAs}) for the isomorphism $\eta_\Gamma$. 
As a consequence, the spectral sequence converges to $H^*(X_1\cup_{X_0} X_2)$. 
%
As for the $E_2$-term, it follows that for any $\sigma \in K$, 
\begin{eqnarray*}
({\mathcal H}_{{\mathcal P}}^*)_\sigma &= & H^*(\F(X_1)_\sigma\times_{\F(X_0)_\sigma}\F(X_2)_\sigma) \\ 
& = & H^*(A(X_1^\sigma) \times_{A(X_0^\sigma)}A(X_2^\sigma)) \\
&\cong & H^*(A(X_1^x) \times_{A(X_0^x)}A(X_2^x))  \\ 
&= &  H^*(A(F_1 ) \times_{A(F_0)}A(F_2)),   
\end{eqnarray*}
where $x$ is a vertex of $\sigma$. By definition, we have the first equality; see \cite[12.43]{Halperin}. 
By assumption, the map $i : X_0 \to X_1$ is injective. Then, the result \cite[Lemma 13.4]{FHT} gives the isomorphism above. 
We have the assertion (1). 

The construction in \cite[Theorem 12.47]{Halperin} implies the naturality of the spectral sequence. This yields the assertion (2).  
\end{proof}


\subsection{An application of Theorem \ref{thm:SS} to a stratifold}

While Theorem \ref{thm:SS} is for a triple of  simplicial sets, the spectral sequence is also applicable to objects in $\mathsf{Diff}$. 
To see this, we consider appropriate diffeological adjunction spaces. 

In what follows, we may write $A(M)$ for $A(S^D(M))$ for a diffeological space $M$. Under the same condition as in Example \ref{ex:triangles}, 
the universality of pullback induces a map $\Theta :  A(P\cup_NM) \to  A(P)\times_{A(N)}A(M)$. 
Assume further that $N$ is a diffeological subspace of $M$ with the inclusion $i : N \to M$.
In general, we call a pair $(M, N)$ of a diffeological space $M$ and its diffeological subspace $N$ a {\it relative space}. 
Then, the canonical map $\widetilde{\Theta} : M \to P \cup_NM$ gives a map  
$\widetilde{\Theta} : (M, N) \to (P\cup_NM, P)$ of relative spaces; see the paragraph before Lemma \ref{lem:Theta}.  
Thus,  these maps fit into the commutative diagram 
\begin{equation}\label{eq:Theta}
\xymatrix@C20pt@R15pt{
0 \ar[r] & A(P\cup_NM, P) \ar[d]_{\widetilde{\Theta}^*}\ar[r]& 
                  A(P\cup_NM) \ar[r] \ar[d]^{\Theta}& A(P) \ar[r] \ar@{=}[d] &0 \\
0 \ar[r] & A(M, N) \ar[r]& A(P)\times_{A(N)}A(M) \ar[r] & A(P) \ar[r] &0 
}
\end{equation}
in which two row sequences are exact and $\Theta$ is the morphism of CDGA's in the deagram (\ref{eq:DGAs}); see the proof of \cite[15.18]{Halperin}. 
By virtue of Proposition \ref{prop:commutative_D}, we have 
\begin{prop}\label{prop:well_attached} The following are equivalent. \text{\em (i)} The map $\phi^*$ in the diagram (\ref{eq:DGAs}) is a quasi-isomorphism.
\text{\em (ii)} $\widetilde{\Theta}^*$ is a quasi-isomorphism, \text{\em (iii)} $\Theta$ is a quasi-isomorphism.
\end{prop}

With the argument above, we give results concerning the singular de Rham cohomology of diffeological adjunction spaces. 

\begin{ex} \label{ex:an_example2} We use the simplicial singular de Rham complex $(A_{DR}^*)_\bullet$ as the admissible simplicial CDGA 
in order to consider the Souriau-de Rham cohomology of a stratifold. 

Let $S$ be 
an $n$-dimensional stratifold 
and 
$W$ a $s$-dimensional manifold with compact boundary $\partial W$ endowed with a collar 
$c : \partial W \times [0,  \varepsilon) \to W$, where $s > n$. Let $f: \partial W \to S$ be a smooth map. 
Then, we construct a parametrized stratifold of the form $S\cup_f W$. 
We refer the reader to \cite{Kreck} and Appendix \ref{app:appA} for the definition of a stratifold. 
The functor $k$ from the category of stratifolds to $\mathsf{Diff}$ in \cite{A-K}, which is also recalled in Appendix \ref{app:appA}, gives a diffeological space 
$k(S\cup_f W)$. By Theorem \ref{thm:p-stfd}, we have a smooth homotopy equivalence  $l :  k(S)\cup_{k(f)}k(W) \to k(S\cup_f W)$ in $\mathsf{Diff}$. 
Moreover, the excision axiom for the singular homology of a diffeological space (\cite[Proposition 3.1]{Kihara}) and the de Rham theorem (\cite[Theorem 2.4]{K2020}) allow us to deduce that the canonical map $$\widetilde{\Theta} : (k(W), k(\partial W)) \to (k(S)\cup_{k(f)}k(W), k(S))$$ gives rise to an isomorphism on the relative de Rham cohomology; 
see Lemma \ref{lem:Theta}. 
%
Therefore, Proposition \ref{prop:well_attached} implies that the map 
\[
\phi^* : A_{DR}(S^D(k(S)\cup_{k(f)}k(W))) \to A_{DR}(S^D(k(S))\cup_{S^D(k(f))}S^D(k(W)))
\] 
is a quasi-isomorphism. 

Suppose that the inclusion $\partial W \to W$ and the map $f$ induce isomorphisms on the fundamental groups, respectively. 
Then, in view of the isomorphism $H(\phi^*)\circ H(l)$, Theorem \ref{thm:SS} and Example \ref{ex:triangles} enable us to construct a spectral sequence converging to $H^*(A_{DR}(k(S\cup_f W)))$ as an algebra. 

%

As mentioned in Remark \ref{rem:tautological_map}, 
the de Rham theorem for the Souriau--de Rham complex does not hold in general. 
However, for a parametrized stratifold $S\cup_f W$, we have a sequence
\[
\xymatrix@C20pt@R15pt{
\Omega^*(k(S\cup_f W)) \ar[r]^-{\alpha}_-{\simeq} & A_{DR}^*(S^D(k(S\cup_f W)) \ar[r]^-{l^*}_-{\simeq}&  A_{DR}^*(S^D(k(S)\cup_{k(f)}k(W)))  
}
\]
of quasi-isomorphisms; see Proposition \ref{prop:factor_map} for the {\it factor map} $\alpha$.  
Thus, the spectral sequence constructed above converges to the Souriau--de Rham cohomology $H^*(\Omega^*(k(S\cup_f W)))$.
\end{ex}

\begin{ex} \label{ex:an_example} 
Let $X$ be a connected smooth CW complex in the sense of Iwase \cite{I} and $X_n$ the $n$-skeleton of $X$. 
Let $i : N := X_{k} \to X_{k+1}$ be the inclusion. Then a variant of Whitney's approximation theorem \cite[Theorem 3.2]{I} allows us to deduce that $i$ induces an isomorphism $i_* : \pi_1^D(N) \stackrel{\cong}{\to} \pi_1^D(X_{k+1})$ if $k > 1$. Here the smooth fundamental group of a diffeological space $X$ is regarded as that in \cite[Theorem 3.2 (9)]{C-W}. Thus, an element of the fundamental group is represented by a smooth map from $S^1$ to $X$. 

Let $P$ be a connected diffeological space and $N \to P$ 
a smooth map in $\mathsf{Diff}$ which induces an isomorphism on the fundamental groups.  
Consider the diffeological adjunction space $P\cup_NX_{k+1}$. 
Each open subset of a smooth CW complex for the  D-topology coincides with the open subset of the {\it underlying topology}; 
see \cite[Proposition 3.2]{I} and its previous sentence for the fact and the underlying topology. 
Therefore, by the same argument as in 
Example \ref{ex:an_example2} with Theorem \ref{thm:SS}, Example \ref{ex:triangles} and the excision axiom, 
we have a spectral sequence 
converging to $H^*(A(S^D(P\cup_NX_{k+1})))$ with 
$E_2^{*,*} \cong H^*(K, {\mathcal H}_{\mathcal Q}^*),$ 
where ${\mathcal H}_{\mathcal Q}^*$ is the system of local coefficients associated with the local system 
${\mathcal Q}:=\F(S^D(P))\times_{\F(S^D(N))}\F(S^D(X_{k+1}))$
over $K=K(\pi_1(N), 1)$. 
\end{ex}


In the following example, we compare two spectral sequences obtained by Theorem \ref{thm:SS}. In particular, the computation of the singular de Rham cohomology of the unreduced suspension of a manifold described in Proposition  \ref{prop:suspensions} is used in order to find non trivial elements in the cohomology of a more general diffeological adjunction space. 

\begin{ex}\label{ex:compare}
Let $K$ and $S$ be manifolds without boundaries and $M$ a closed manifold. 
We assume that $M$ and $K$ are connected, and $S$ has at least two connected components. 
Let $j_1 : (M\coprod M) \times K \to S$ be a smooth map. We define a map 
$j :  (M\coprod M)\times K \to S\times K$ by $j(m , k) = (j_1(m, k), k)$. 
Consider a commutative diagram of triples   
\[
\xymatrix@C20pt@R12pt{
(M\times I) & (M\coprod M) \ar[l]_{in}\ar[r]^f & (*\coprod *) \\
(M\times I)\times K \ar[u]^{\pi_1}\ar[rd]_{p_1}& (M\coprod M)\times K \ar[d]_{p_0}\ar[u]^{\pi_0} \ar[l]_{in\times 1}\ar[r]^-j & S\times K, \ar[ld]^{p_1} \ar[u]_{\pi_2}\\
 & K &
}
\]
where $\pi_1$ and $\pi_0$ are projections in the first factor, $p_i$ is the projection in the second factor 
for $i =0, 1$ and $2$. Here, we regard the manifold $M\times I$ with boundary as the diffeological space given by a relevant stratifold via the functor $k$ in \cite{A-K}; see Theorem \ref{thm:p-stfd}. 
By definition, the adjunction space $(M\times I)\cup_{(M\coprod M)}(*\coprod *)$ is the unreduced suspension $\Sigma M$ of $M$; see Appendix \ref{app:appB}. 

Let $Z_{(M, S)}$ be a diffeological adjunction space of the form 
$$(M\times I)\times K\cup_{(M\coprod M)\times K}(S\times K).$$ Then, the maps $\pi_1$, $\pi_0$ and $\pi_2$ give rise to a smooth map $\widetilde{\pi} : Z_{(M, S)} \to \Sigma M$.  The first row in the diagram above is regarded as a triple over the point $*$. 
Thanks to Theorem \ref{thm:SS} (2) and Remark \ref{rem:naturality}, we see that 
the trivial map $\rho: K \to *$ induces a morphism $\{\Psi_r\} : \{{}^\backprime E_r^{*,*}, {}^\backprime d_r\} \to \{E_r^{*,*}, d_r\}$ of spectral sequences constructed with the simplicial singular de Rham complex $A:=(A_{DR}^*)_\bullet$.

Moreover, the argument in Example \ref{ex:an_example2} (i) enables us to deduce that the spectral sequences $\{{}^\backprime E_r^{*,*}, {}^\backprime d_r\}$ and $\{E_r^{*,*}, d_r\}$ converge to the singular de Rham cohomology algebras $ H^*(A(K))\otimes H^*(A(\Sigma M))$ and $H^*(A(Z_{(M, S)}))$,  
respectively. Observe that the adjunction space of the pullback of the first row along $\rho$ is nothing but the product 
$K\times \Sigma M$. 

Since the fibration which gives rise to ${}^\backprime E_r^{*,*}$ is trivial, the system of local coefficients in the $E_2$-term is simple. Moreover, the spectral sequence converges to $H^*(A(K))\otimes H^*(A(\Sigma M))$. 
Thus, we see that  the spectral sequence collapses at the $E_2$-term. 
This implies that for $1\otimes \xi \in {}^\backprime E_2^{0, *}$, the element $\Psi_2(1\otimes \xi)$ is a permanent cycle in the spectral sequence $\{E_r^{*,*}, d_r\}$; see Propositions \ref{prop:suspensions} and \ref{prop:formsOnSigma M} for the explicit form of $H^*(A(\Sigma M))$.  
We observe that the system of local coefficients in the $E_2$-term of the spectral sequence $\{E_r^{*,*}, d_r\}$ is not necessarily trivial because $j$ is not the product of maps in general.

We conclude this example with the following assertion. 
\begin{assertion}\label{assertion:nonzero_cycles} Under the notation above, 
suppose that $H^i(A(S))= 0$ for $0< i< l$. Then 
$\Psi_2(1\otimes \xi) \neq 0$ provided $\xi \neq 0$ in $H^*(A(\Sigma M))$ and $\deg \xi \leq l$.
\end{assertion}

Thus, under the same assumption as in Assertion \ref{assertion:nonzero_cycles}, we may find many nontrivial permanent cycles in $\{E_r^{*,*}, d_r\}$. As a consequence, we see that $\widetilde{\pi}^*(\xi)\neq 0$ if $\deg \xi \leq l$ 
for the map $\widetilde{\pi}^* : H^*(A(\Sigma M)) \to H^*(A(Z_{(M, S)}))$. Moreover, since 
$\xi \wedge \xi' = 0$ in $H^*(A(\Sigma M))$ if $\deg \xi$ and $\deg \xi'$ are greater than or equal to $1$, it follows that 
$\widetilde{\pi}^*(\xi)\wedge \widetilde{\pi}^*(\xi')= 0$ in $H^*(A(Z_{(M, S)}))$.


\begin{proof}[Proof of Assertion \ref{assertion:nonzero_cycles}]
The naturality of 
the spectral sequence (Theorem \ref{thm:SS} (2)) and 
the cohomology with local coefficients allow us to obtain a commutative diagram 
\[
\xymatrix@C20pt@R12pt{
{}^\backprime E_r^{0,*} \ar[r]^{\Psi_2} \ar[d]_-{i'}^-\cong& E_r^{0,*} \ar[d]^{i} \\
H^*(A(M\times I)\times_{A(M\coprod M)}A(*\coprod *)) \ar[r]_-{t^*}&
H^*(A(M\times I)\times_{A(M\coprod M)}A(S)),
}
\]
where $i$ and $i'$ denote the edge homomorphisms of the spectral sequences and $t$
is the map between fibres induced by the map 
$j_1 ( \ , *) : M\coprod M \to S$. 

Observe that the edge homomorphisms are injective; see, for example, the argument before \cite[(16.16)]{Halperin}. As mentioned above, all differential in $\{{}^\backprime E_r^{*,*}, {}^\backprime d_r\}$ are trivial. 
Then, it follows from \cite[(16.16)]{Halperin} that $i'$ is an isomorphism. 

In what follows, we use the same notation for an element in cohomology as that for its representative. 
Let $\wedge V_\alpha$ be a minimal model for each connected component $S_\alpha$ ($\alpha \in B$) of $S$. By \cite[Lemma 13.3]{FHT}, we see that the homomorphism $t^*$ is induced by a sequence of chain maps 
\[
\xymatrix@C15pt@R10pt{
        & A(M\times I)\times_{A(M\coprod M)}\coprod_{\alpha \in B} A(S_\alpha) \\
F:= A(M\times I)\times_{A(M\coprod M)}A(*\coprod *) \ar[ru]^{\widetilde{t}} &
A(M\times I)\times_{A(M\coprod M)}\prod_{\alpha \in B}(\wedge V_\alpha) \ar[u]_{\simeq}
}
\]
including a quasi-isomorphism, where 
$\widetilde{t}((\eta, m)) = (\eta, A(\pi_2 \circ i_1)(m))$ with  $i_1 : S \to S \times *$ the inclusion. The computation in Proposition \ref{prop:suspensions} shows that a cycle $\xi$ in $F$ with degree greater than or equal to $1$ is of the form 
$(\eta ,0)$. Indeed, the map $\kappa$ in (\ref{eq:xi}) is a quasi-isomorphism. Therefore, we may assume that $\xi = (\eta, 0)$ for some $\eta$ in $A(M\times I)$. 

We see that $t^*((\eta, 0)) = (\eta, 0) \in 
(A(M\times I)\times_{A(M\coprod M)} \prod_{\alpha \in B}(\wedge V_\alpha))^{\deg \eta}$. Since $V_\alpha^i = 0$ for $0< i < l$ by assumption, it follows that 
$$(A(M\times I)\times_{A(M\coprod M)}\prod_{\alpha \in B}(\wedge V_\alpha))^i = (A(M\times I)\times_{A(M\coprod M)} 0)^i$$ for $0< i < l$. 
Thus, for degree reasons, we see that $t^*((\eta, 0)) \neq 0$ if $\xi \neq 0$ in $H^*(A(\Sigma M))$ and $\deg \xi = \deg \eta \leq l$. In fact, if $t^*((\eta, 0)) = 0$, then there exists an element $c$ in $A(M\times I)$ such that $d(c) = \eta$. This implies that $\xi = (\eta, 0) = 0$ in $H^*(A(\Sigma M))\cong H^*(A((M\times I)\times_{A(M\coprod M)}A(*\coprod *)))$, which is a contradiction. 
\end{proof}
\end{ex}

\subsection{A local system of an adjunction space in $\mathsf{Top}$}\label{sect:Top}
In this section, we work in the category $\mathsf{Top}$ and give applications of local systems to adjunction spaces. 

Consider a commutative diagram 
\begin{equation}\label{eq:Top_triangles}
\xymatrix@C30pt@R15pt{
X \ar[dr]_-{p_1} & C \ar[d]_-{p_0} \ar[l]_-i \ar[r]^j & Y\ar[dl]^-{p_2} \\
 & B &
}
\end{equation}
with a connected space $B$ 
in which $p_i$ is a fibration with fibre $F_i$ for $i = 0, 1$ and $2$. 
We denote by $\texttt{hocolim} \ \! (i, j)$ or  $\texttt{hocolim} \ \! (\xymatrix@C10pt@R10pt{\! X& C \ar[l]_i \ar[r]^j & Y})$ the homotopy colimit (homotopy pushout) of the maps $i$ and $j$. 
Let $A$ denote the  simplicial CDGA $(A_{PL}^*)_\bullet$ and $S(X)$ the singular simplicial set of a space $X$.
We may write $A(X)$ for $A(S(X))$.  

\begin{thm}\label{thm:NewAssertions} {\em (1)} Suppose that the map $i$ in the diagram {\em{(\ref{eq:Top_triangles})}} is a closed cofibration. Then there exists a natural quasi-isomorphism 
$\F(S(X\cup_CY)) \stackrel{\simeq}{\to}  \F(S(X)) \times_{\F(S(C))}\F(S(Y))$
of $A$-algebras. \\
{\em (2)} 
Suppose that the map $i$ in the diagram {\em{(\ref{eq:Top_triangles})}} is injective. 
Then there exists a zig-zag of natural quasi-isomorphisms of $A$-algebras connecting  
$\F(S(\texttt{\em hocolim} \ \! (i, j)))$ with $\F(S(X)) \times_{\F(S(C))}\F(S(Y))$
\end{thm}

In order to prove the theorem above, we recall the {\it mapping cylinder of a fibre map} due to McAuley \cite{McAuley}.  
In the diagram (\ref{eq:Top_triangles}), we can replace $i$ with a cofibration without changing the weak homotopy type of the fibre $p_1$.  
To see this, consider a commutative diagram 
\begin{equation}\label{eq:triangle1}
\xymatrix@C30pt@R15pt{
X \ar@/_3.5mm/[drr]_-{p_1} & (C\times I)\cup_{\widetilde{f}}X \ar@{.>}[l]_(0.6)u^(0.6)\simeq \ar@{.>}[dr]_{\overline{p_1}}& C \ar[d]^-{p_0} \ar@/_4.5mm/[ll]_-f  \ar@{.>}[l]^-i\\ 
  & & B 
}
\end{equation}
in $\mathsf{Top}$ of solid arrows in which $p_i$ is a fibration for $i =0$ and $1$.  Then, the decomposition $u\circ i$ of $f$ via the mapping cylinder $\text{Cyl}(f):= (C\times I)\cup_{\widetilde{f}}X$ fits in the commutative diagram with solid and dotted arrows. Here $\widetilde{f} : C\times \{0\} \to X$ is defined by 
$\widetilde{f}(a, 0) = f(a)$ and 
$\overline{p_1}$ is defined by $\overline{p_1}(a, t)= p_0(a)$ for $(x, t) \in A\times I$ 
and $\overline{p_1}(x) = p_1(x)$ for $x \in X$. 
By the result \cite[Corollary 2]{McAuley}, we see that $\overline{p_1}$ is a fibration and then the fibre of  $\overline{p_1}$ is weak homotopy equivalent to 
the fibre of $p_1$. The argument above is used in the proof of Theorem \ref{thm:NewAssertions} (2).

\begin{proof}[Proof of Theorem \ref{thm:NewAssertions} ] (1) 
For a simplicial map $\sigma : \Delta[n] \to S(B)$, let $\widetilde{\sigma} \in S_n(B)$ denote the map $\Delta^n \to B$ corresponding to 
$\sigma$. We consider a commutative diagram 
\begin{equation}
\xymatrix@C30pt@R15pt{
S(X)^\sigma \ar[d] \ar[r]^\zeta & S(X^{\widetilde{\sigma}})  \ar[d]  \ar[r] & S(X) \ar[d]^-{S(p)} \\
 \Delta[n] \ar[r]_t & S(\Delta^n) \ar[r]_{S(\widetilde{\sigma})}& S(B)
}
\end{equation}
consisting of pullback diagrams. Observe that $\sigma = S(\widetilde{\sigma})\circ t$. Since $\Delta[n]$ and $S(\Delta^n)$ are contractible, it follows that 
$\zeta$ is a weak homotopy equivalence. 

By assumption, the map $i$ is a closed cofibration. Then the result \cite[Proposition]{Kieboom} enables us to deduce 
that the natural projection $\overline{p} : X\cup_CY \to B$ obtained by $p_i$ is a fibration.  Moreover, it follows from \cite[Theorem]{Kieboom} that the natural map 
$C^\tau \to  X^\tau$ induced by $i$ between pullbacks along arbitrary map $\tau : B' \to B$ is a closed cofibration.  
Mather's Cube Lemma \cite{Mather} allows us to obtain a homotopy fibration of the form 
$F_1\cup_{F_0}F_2 \to  X\cup_CY \stackrel{\overline{p}}\to B$; see also \cite[Proposition 2.3]{C-S}. By the same argument, we have a homotopy fibration 
$F_1\cup_{F_0}F_2 \to  X^{\widetilde{\sigma}}\cup_{C^{\widetilde{\sigma}}}Y^{\widetilde{\sigma}} \stackrel{\overline{p}}\to \Delta^n$. Therefore, the canonical map $\Psi : X^{\widetilde{\sigma}}\cup_{C^{\widetilde{\sigma}}}Y^{\widetilde{\sigma}} \to (X\cup_CY)^{\widetilde{\sigma}}$ is a weak homotopy equivalence.  

Recall the unique map $\phi : S(X^{\widetilde{\sigma}})\cup_{S(C^{\widetilde{\sigma}})}S(Y^{\widetilde{\sigma}}) \to 
S(X^{\widetilde{\sigma}}\cup_{C^{\widetilde{\sigma}}}Y^{\widetilde{\sigma}})$ described in the diagram (\ref{eq:DGAs}). 
As mentioned in Section \ref{sect:perspectives}, the singular  simplex functor $S : \mathsf{Top} \to \mathsf{Set}^{\Delta^{\text{op}}}$ coincides with the composite $S^D\circ C$. 
Recall the diagram (\ref{eq:Theta}) in the setting of this theorem. 
Since the map $C^{\widetilde{\sigma}} \to  X^{\widetilde{\sigma}}$ is a cofibration, it follows that $\widetilde{\Theta}^*$ in (\ref{eq:Theta}) is a quasi-isomorphism. Then, by Proposition \ref{prop:well_attached}, we see that the map $\phi$ above is also a quasi-isomorphism. Consider a commutative diagram 
\begin{equation}
\xymatrix@C35pt@R10pt{
S(X\cup_CY)^\sigma  \ar[r]^\zeta_\simeq & S((X\cup_CY)^{\widetilde{\sigma}}) \\
 & S(X^{\widetilde{\sigma}}\cup_{C^{\widetilde{\sigma}}}Y^{\widetilde{\sigma}}) \ar[u]_{S(\Psi)}^\simeq \\
S(X)^\sigma\cup_{S(C)^\sigma}S(Y)^\sigma \ar[r]_(0.48){\zeta_3 := \zeta\cup_\zeta\zeta} \ar[uu]^\gamma & S(X^{\widetilde{\sigma}})\cup_{S(C^{\widetilde{\sigma}})}S(Y^{\widetilde{\sigma}}), \ar[u]_\phi^\simeq
  }
\end{equation}
where $\gamma$ is  the natural map defined by the universality of the colimit. Moreover, we consider a commutative diagram 
\[
\xymatrix@C35pt@R18pt{
A(S(X^{\widetilde{\sigma}})\cup_{S(C^{\widetilde{\sigma}})}S(Y^{\widetilde{\sigma}})) \ar[d]_{\zeta_3} \ar[r]^-{\eta_A}_-\cong&
      A(S(X^{\widetilde{\sigma}})\times_{ A(S(C^{\widetilde{\sigma}})}   A(S(Y^{\widetilde{\sigma}})) \ar[d]^{A(\zeta)\times_{A(\zeta)} A(\zeta)}\\
A(S(X)^\sigma\cup_{S(C)^\sigma}S(Y)^\sigma) \ar[r]_-{\eta_A}^-\cong &  A(S(X)^{\sigma})\times_{ A(S(C)^{\sigma})}  A(S(Y)^{\sigma})
}
\]
in which $\eta_A$ denotes the map defined by the universality of the limit. Since $A$ is the left adjoint to the realization functor, it follows that each $\eta_A$ is an isomorphism. We see that the natural maps $S(C^{\widetilde{\sigma}}) \to S(X^{\widetilde{\sigma}})$ and 
$S(C)^{\sigma} \to S(X)^{\sigma}$ defined by $i$ are injective.  It follows from \cite[Lemma 13.3]{FHT} that $A(\zeta)\times_{A(\zeta)} A(\zeta)$ is a quasi-isomorphism. As a consequence, the composite $\eta_A\circ A(\gamma) : A(S(X\cup_CY)^\sigma ) \to A(S(X)^{\sigma})\times_{ A(S(C)^{\sigma})}  A(S(Y)^{\sigma})$ for each $\sigma \in S(B)$ gives rise to the quasi-isomorphism required.  

(2) The homotopy colimit $\texttt{hocolim} \ \! (i, j)$ is regarded as 
the double mapping cylinder, namely, 
the pushout of two mapping cylinders  which are in the first sequence of the commutative diagram 
\[
\xymatrix@C30pt@R10pt{
\text{Cyl}(i) \ar[dr]_-{\overline{p_1}} & C \ar[d]_-{p_0} \ar[l]_-{\widetilde{i}} \ar[r]^{\widetilde{j}} & \text{Cyl}(j). \ar[dl]^-{\overline{p_2}} \\
 & B &
}
\]
Here $\overline{p_i}$ is the fibration induced by $p_i$ for $i =1$ and $2$; see the argument after Theorem \ref{thm:NewAssertions}. 
Since the map $i$ is injective, the restriction of $S(i)$ is also injective. The extendability of $A$ allows us to deduce that 
$A(u)\times_{A(id_C)} A(v) : A(S(X)^\sigma)\times_{A(S(C)^\sigma)}A(S(Y)^\sigma) 
\stackrel{\simeq}{\to}A(S(\text{Cyl}(i))^\sigma)\times_{A(S(C)^\sigma)}A(S(\text{Cyl}(j))^\sigma)$ is a quasi-isomorphism, where 
$u : \text{Cyl}(i) \stackrel{\simeq}{\to} X$ and $v : \text{Cyl}(j) \stackrel{\simeq}{\to} Y$ are the homotopy equivalences.
Hence, the result follows from (1). 
\end{proof}

\begin{ex}\label{ex:TopLoopSp} Let $M$ be the complex projective space $\mathbb{C}P^n$. Let $L^cM$ and $\Omega^c M$ be the space of continuous free loops on $M$ and the subspace of based loops, respectively. 
We construct a tractable local system model for the homotopy colimit 
$\texttt{hocolim} \ \! (\xymatrix@C10pt@R10pt{\!L^cM &  \Omega^c M \ar[r]^i \ar[l]_i & L^cM}\!)$ by using the minimal local system model described in Example \ref{ex:LM}.  

With the same notation as in Example \ref{ex:LM}, we have a commutative diagram
\[
\xymatrix@C30pt@R14pt{
 & \wedge (x, y) \ar[dl]_{\simeq} \ar@{>->}[r]& \wedge (x, y, \overline{x}, \overline{y}) \ar[dl]_{\simeq}^{w_1} \ar[r]^{\widetilde{i}} & \wedge ( \overline{x}, \overline{y})\ar[dl]_(0.65){\simeq}^(0.6){w_2} \\
A(M) \ar[r]_-{ev^*} & A(L^cM) \ar[r]_-{i^*} & A(\Omega^c M) & \wedge ( \overline{x}) \ar[ld]_-{\simeq}^-{w_3} \ar@{.>}[u]_{\widetilde{p_0}} \ar@{.>}[ul]^(0.3){\widetilde{p_1}}|(0.5)\hole \\
 & &  A(K(\pi_1(L^cM), 1))  \ar[u]_{p_0^*}\ar[ul]^{p_1^*}  &
}
\]
of solid arrows in which the upper squares consist of a relative Sullivan model for the evaluation map $ev :L^cM \to M$ at zero and a model $\widetilde{i}$ 
for the inclusion 
$i : \Omega^c M \to L^cM$. Observe that $K(\pi_1(L^cM), 1)=K(\pi_1(\Omega^c M), 1)$.

The Lifting lemma \cite[Proposition 12.9]{FHT} gives morphisms $\widetilde{p_i}$ which fit in the commutative square diagrams up to homotopy.  In the lower homotopy commutative square consisting of $w_1$, $w_3$, $p_1^*$ and $\widetilde{p_1}$, we see that $p_1^*w_3 (\overline{x}) = w_1\widetilde{p_1} (\overline{x})$ in the cohomology $H^*(L^cM)$. Then there exists an element $z \in A^1(L^cM)$ such that $p_1^*w_3 (\overline{x}) - w_1\widetilde{p_1} (\overline{x}) = d(z)$. We define $w_1' : \wedge (x, y, \overline{x}, \overline{y}) \to A(L^cM)$ by $w_1'(\overline{x}) = w_1(\overline{x}) + d(z)$. It follows that  the lower square is commutative and $w_1 \simeq w_1'$. Moreover, by the same argument as above, we can define $w_2' :  \wedge (\overline{x}, \overline{y}) \to A(\Omega^cM)$ so that  the upper right-hand square is commutative and $w_2 \simeq w_2'$. Therefore, the right triangular prism is commutative with $w_i'$ instead of $w_i$ for $i = 1$ and $2$. 

In view of Corollary \ref{cor:F2}, we have a commutative diagram 
\[
\xymatrix@C30pt@R14pt{
(A(K)\otimes \wedge (x, y, \overline{y}))_* \ar[r]^-{ad(\theta)}_-\simeq \ar[d]_{(\widetilde{i})_*}& \F(S(L^cM))\ar[d]^{\F(i)} \\
(A(K)\otimes \wedge (\overline{y}))_* \ar[r]_-{ad(\theta)}^-\simeq & \F(S(\Omega^c M)), \
}
\]
where $K = K(\pi_1(L^cM), 1)=K(\pi_1(\Omega^c M), 1)$. By Theorem \ref{thm:NewAssertions} (2), we have a zig-zag of quasi-isomorphisms  
of $A$-algebras connecting $\texttt{hocolim} \ \! (\xymatrix@C10pt@R10pt{\!L^cM &  \Omega^c M \ar[r]^i \ar[l]_i & L^cM}\!\!)$ and 
$$
(A(K)\otimes \wedge (x, y, \overline{y}))_*\times_{(A(K)\otimes \wedge (\overline{y}))_*}(A(K)\otimes \wedge (x, y, \overline{y}))_*,
$$
which is a local system model for the homotopy colimit. 
\end{ex}

By virtue of Theorem \ref{thm:NewAssertions}, we establish a spectral sequence for an adjunction space over a space in $\mathsf{Top}$. 

\begin{thm}\label{thm:Top_SS} Suppose that the map $i$ in the diagram {\em (\ref{eq:Top_triangles})} is injective. Then, 
there exists 
a first quadrant spectral sequence $\{E_r^{*,*}, d_r \}$ converging to the cohomology 
$H^*(A(\texttt{\em hocolim}(i, j))) = H^*(\texttt{\em hocolim}(i, j); \Q)$ as an algebra such that 
$E_2^{*,*} \cong H^*(B, {\mathcal H}_{\mathcal Q}^*)$
as a bigraded algebra. Here, ${\mathcal H}_{\mathcal Q}^*$ is a system of local coefficients 
over $S(B)$ which satisfies the condition that, for any $\sigma \in S(B)$, 
$$({\mathcal H}_{\mathcal Q}^*)_\sigma \cong H^*(A(F_1)\times_{A(F_0)}A(F_2))$$ with $F_i$ the fibre of $p_i$ for $i = 0, 1$ and $2$. 
Moreover, the spectral sequence is natural with respect to triples of fibrations over $B$ in the same sense as in Theorem \ref{thm:SS} (2). 
\end{thm}

\begin{proof}
We apply the proof of Theorem \ref{thm:SS} to a local system of the form $\mathcal{P} = \F(S(X)\cup_{S(C)}S(Y)) \cong \F(S(X)) \times_{\F(S(C))}\F(S(Y))= \mathcal{Q}$. Then, Theorem \ref{thm:NewAssertions} (2) 
enable us to obtain the spectral sequence converging to $H^*(\texttt{hocolim}(i, j); \Q)$. 
We see that $S(X)^{\widetilde{b}} = S(p_1^{-1}(b))$ for $b\in B$, where $\widetilde{b} : \Delta[0] \to S(X)$ denotes the constant map at $b$. 
Thus, the formula for $({\mathcal H}_{\mathcal Q}^*)_\sigma$ follows from that in Theorem \ref{thm:SS}. 
\end{proof}

\begin{rem}\label{rem:simple_LC}
In the Theorem \ref{thm:Top_SS}, assume that the base space $B$ is simply connected. Then, since the system ${\mathcal H}_{\mathcal Q}^*$ of local coefficients is simple, it follows that $E_2^{*, *} \cong H^*(B; \Q)\otimes H^*(A(F_1)\times_{A(F_0)}A(F_2) )$ as a bigraded algebra. 
\end{rem}

\begin{rem}
The restriction $i_| : F_0 \to F_1$ is also injective. Then, by using \cite[Lemma 13.3]{FHT} and the same argument as in the proof above with (\ref{eq:Theta}), we have a sequence of quasi-isomorphisms  between $A(F_1)\times_{A(F_0)}A(F_2)$ and $A(\texttt{hocolim}(i_|, j_|)))$. Then, the spectral sequence above is regarded as the Leray--Serre spectral sequence for a homotopy fibration of the form 
$\texttt{hocolim}(i_|, j_|)) \to \texttt{hocolim}(i, j)) \to B$; see \cite[Proposition 2.3]{C-S} for such a fibration. 
\end{rem}

\medskip
\noindent
{\it Acknowledgements.} The author is grateful to Hiroshi Kihara for several helpful conversations about the model structure on the 
category of diffeological spaces. The author would like to thank Yukiho Tomeba for discussing the de Rham cohomology of diffeological adjunction spaces. Author's grateful thanks are extended to Jean-Claude Thomas for his explanation of the results in \cite{F-T}, which enables the author 
to deduce Theorem \ref{thm:nilToMinimalLocalSystem}. The author would also like to thank the referee for very careful reading and valuable comments on a version of this manuscript without which he could not have obtained the refined proof of Theorem \ref{thm:SS}, Proposition \ref{prop:formsOnSigma M} and the results in Section \ref{sect:Top}. 
This work was partially supported by JSPS KAKENHI Grant Numbers JP19H05495 and JP21H00982

\appendix

\section{A stratifold as a diffeological space}\label{app:appA}

Stratifolds \cite{Kreck} are important examples of diffeological spaces. 
In this section, after reviewing the notion and a functor which assigns to a stratifold a diffeological space, we investigate a parametrized stratifold from the view point of diffeology. As a consequence, we have Theorem \ref{thm:p-stfd} which completes the argument 
in Example \ref{ex:an_example2}. 
In order to define a general stratifold,  
we recall a differential space in the sense of Sikorski \cite{Sik}. 
\begin{defn} \label{defn:differential_space}
A {\it  differential space} is a pair $(S, \C)$ consisting of a topological space $S$ and an $\mathbb{R}$-subalgebra $\C$
of the $\mathbb{R}$-algebra $C^0(S)$ of continuous real-valued functions on $S$, which is assumed to be {\it locally detectable} and 
$C^\infty$-{\it closed}.  

\medskip
Local detectability means that $f \in \C$ if and only if 
for any $x \in S$, there exist an open neighborhood $U$ of $x$ and an element $g \in \C$ such that 
$f|_U = g|_U$.  

\medskip
$C^\infty$-closedness means that for each $n\geq 1$, each $n$-tuple $(f_1, ..., f_n)$ of maps in $\C$ and each smooth map 
$g : \mathbb{R}^n \to \mathbb{R}$, the composite  
$h : S \to \mathbb{R}$ defined by $h(x) = g(f_1(x), ...., f_n(x))$ belongs to $\C$.
\end{defn}

Let $(S, \C)$ be a differential space and $x$ an element in $S$. The vector space consisting of derivations on the $\mathbb{R}$-algebra 
$\C_x$ of the germs at $x$ is denoted by $T_xS$, which is called the {\it tangent space} 
of the differential space  at $x$; see \cite[Chapter 1, section 3]{Kreck}. 

\begin{defn} \label{defn:stratifold} 
An $n$-{\it dimensional stratifold} is a differential space $(S, \C)$ such that the following four conditions hold: 

\begin{enumerate}
\item
$S$ is a locally compact Hausdorff space with countable basis; 
\item for each $x\in S$, the dimension of the tangent space $T_xS$ is less than or equal to $n$ and 
the {\it skeleta} $sk_t(S):= \{x \in S \mid \text{dim } \!T_xS\leq t\}$ are closed in $S$;
\item
for each $x \in S$ and open neighborhood $U$ of $x$ in $S$, 
there exists a {\it bump function} at $x$ subordinate to $U$; that is, a non-negative function $\phi \in \C$ such that 
$\phi(x)\neq 0$ and such that the support $\text{supp }\! \phi :=\overline{\{p \in S \mid f(p) \neq 0\}}$ is contained in $U$; 
\item
the {\it strata} $S^t := sk_t(S) - sk_{t-1}(S)$ are $t$-dimensional smooth manifolds such that each restriction along 
$i : S^t \hookrightarrow S$ induces an isomorphism of stalks 
$
i^* : \C_x \stackrel{\cong}{\to} C^\infty(S^t)_x
$
for each $x \in S^t$.
\end{enumerate}
\end{defn}

Observe that for a manifold $M$ without boundary, we can regard $M$ as the stratifold $j(M) :=(M, \C_M)$ defined by $\C_M=\{\phi : M \to \R \mid \phi : \text{smooth} \}$; see \cite{Kreck}. 
Let $(S_1, \C_1)$ and $(S_2, \C_2)$ be stratifolds and $h : S_1 \to S_2$ a continuous map. We call the map $h$, denoted 
$h  : (S_1, \C_1) \to (S_2, \C_2)$, a morphism of stratifolds if $\phi\circ h \in \C_1$ for every $\phi \in \C_2$. 
Thus, we have the category $\mathsf{Stfd}$ of stratifolds. 

We recall from \cite{A-K} that a functor  
$k :  \mathsf{Stfd} \to \mathsf{Diff}$
is defined by 
$k(S, \C) = (S, \D_\C)$ and $k(f) = f$ for a morphism $f : S \to S'$ of stratifolds, where 
\[
\D_\C:=\Set{u : U \to S | 
\begin{array}{l}
\text{$U :$ open in ${\R^q}, q \geq 0$,} \\
\text{$\phi\circ u \in C^\infty(U)$ for any  $\phi \in \C$} 
\end{array} }. 
\]
Observe that a plot in $\D_\C$ is a set map.  We see that 
the functor $k$ is faithful, but not full in general. It is worth mentioning that the fully faithful embedding $m : \mathsf{Mfd} \to 
\mathsf{Diff}$ in \cite{IZ} from the category $\mathsf{Mfd}$ of manifolds is indeed an embedding $j$ followed by $k$; that is, we have a sequence of functors
\[
\xymatrix@C40pt@R12pt{
m : \mathsf{Mfd} \ar[r]_-{\text{\tiny fully faithful}}^-j  & \mathsf{Stfd} \ar[r]^-k &\mathsf{Diff}. 
}
\]
Here the functor $j$ is defined by the assignment $M \mapsto j(M)$ mentioned above. 
We refer the reader to \cite[Section 5]{A-K} for the details of the functors. 

A parametrized stratifold is constructed from a manifold attaching another manifold with compact boundary. More precisely, let $(S, \C)$ be 
an $n$-dimensional stratifold. Let 
$W$ be an $s$-dimensional manifold with compact boundary $\partial W$ endowed with a collar 
$c : \partial W \times [0,  \varepsilon) \to W$. Suppose that $s > n$. Let $f : \partial W \to S$ be a morphism of stratifolds. 
We consider the adjunction topological space $S':=S\cup_f W$. Define the subalgebra of $C^0(S')$ by 
\[
\C_{\text{ad}}^{(S, W)}\!=\!\Set{ g : S' \to \R  | \!
\begin{array}{l}
g_{| S} \in \C,  g_{|W\backslash \partial W} \ \text{is smooth} \  \text{and for some positive real}  \\ \text{number $\delta < \varepsilon$}, 
\text{$gc(w, t)=gf(w)$ for $w \in \partial W$ 
and $t < \delta$}  
\end{array} }. 
\]
Then the pair $(S', \C_{\text{ad}}^{(S, W)})$ is a stratifold; see  \cite[Example 9]{Kreck} for more details. 
A stratifold constructed by attaching inductively manifolds with such a way is called 
a {\it parametrized stratifold} ($p$-stratifold for short). 


Let $(S, \C)$ be a stratifold and $W$ a manifold with compact boundary 
$\partial W$ mentioned above. 

\begin{thm}\label{thm:p-stfd}
The natural map $l :k(S, \C)\cup_{k(f)} k(W, \C_W) \to k(S\cup_fW, \C')$ is a smooth homotopy equivalence in $\mathsf{Diff}$, where $(W, \C_W)$ denotes a stratifold of the form $(\partial W\cup_{id}W, \C_{\text{\em ad}}^{(\partial W, W)})$. 
\end{thm}

We recall results in \cite{K2020} and \cite{Kreck} which are used repeatedly when proving the theorem. 
\begin{lem}\label{lem:A} \text{\em (\cite[Lemma A.7]{K2020})}
For a stratifold $(S, \C)$,  an open subset of the underlying topological space $S$ is a $D$-open subset of the diffeological space $k(S, \C)$; see Section \ref{sect:diffspaces} for the D-topology. 
\end{lem}

\begin{lem}\label{lem:B} \text{\em (\cite[Proposition 2.4]{Kreck})}
Let $(S, \C)$ be a stratifold and $A$ a closed  subset of $S$. For a smooth map $g : U \to {\mathbb R}$ from an open neighborhood $U$ of $A$, there exists a smooth map $\widetilde{g} : S \to {\mathbb R}$ such that $\widetilde{g}|_A = g|_A$. 
\end{lem}

We consider a smooth function $r : {\mathbb R} \to  {\mathbb R}$ which satisfies the condition that 
$r(t) = 0$ for 
$t \leq \frac{\varepsilon}{2}$, $0 \leq r(t) \leq t$ for $\frac{\varepsilon}{2} \leq t \leq \frac{3\varepsilon}{4}$ and 
$r(t) = t$ for $\frac{3\varepsilon}{4} \leq t$. 
Then the map $1\times r : \partial W \times [0, \varepsilon) \to \partial W \times [0, \varepsilon)$
is defined. Extending $1\times r$ to $\widetilde{1\times r} : W \to W$ and using 
$\widetilde{1\times r}$, we further define a map 
$\overline{r} : k(S\cup_f W) \to k(S)\cup_{k(f)}k(W)$ with $\overline{r} 
|_S = 1_S$.

\begin{lem}\label{lem:smoothness}
The map $\overline{r}$ is smooth in $\mathsf{Diff}$.
\end{lem}

\begin{proof} In order to prove the result, by definition, it suffices to show that $\overline{r}$ locally factors through the projection 
$\pi: k(S)\coprod k(W) \to k(S)\cup_{k(f)}k(W)$, where $k(S)\coprod k(W)$ denotes the coproduct with the sum diffeology. 
In fact, the diffeological quotient space $k(S)\cup_{k(f)}k(W)$ is endowed with the quotient diffeology of the coproduct; see 
Example \ref{ex:finaldiff}. 

In what follows, we may write $\pi(S)$ for $\pi(k(S))$. 
Let $p : U_p \to  k(S\cup_f W)$ be a plot; that is, $p$ is in $\D_{\C_{\text{ad}}^{(S, W)}}$. 
Since the set $\pi(S)$ is closed, 
it follows from Lemma \ref{lem:A} that $V':=U_p\backslash p^{-1}(\pi(S))$ is open in $U_p$. 
Moreover, the composite $\overline{r}\circ p$ has a lift $p'':=(\widetilde{1\times r})\circ p' : V' \to W \to W$ with $\pi \circ p' = p$ for some $p' : V' \to W$ and 
$\pi \circ (\widetilde{1\times r})\circ p' = \overline{r}\circ p$.  We show that $p''$ is smooth. To this end, it suffices to prove that 
$\phi\circ p''$ is a $C^\infty$-function for each $\phi : W= \partial W\cup_{id}W \to {\mathbb R}$ in $\C_W$. 


Consider the composite $\phi\circ (\widetilde{1\times r}) : W \to W \to {\mathbb R}$. Observe that the restriction of $\widetilde{1\times r}$ to 
$\partial W \times [0, \varepsilon)$ is nothing but the map $1\times r$. Then, since the map $\phi$ is in $\C_W$,
the composite $\phi\circ \widetilde{1\times r}|_{\partial W \times [0, \varepsilon)}$ is extended to a smooth map 
$\phi' :  \partial W \times (-\varepsilon, \varepsilon) \stackrel{1\times r}{\to} \partial W \times (-\varepsilon, \varepsilon) \stackrel{\phi}{\to} {\mathbb R}$.
For any element $x\in V'$, we see that $p(x) \in W\backslash \partial W$. Thus, it follows from Lemma \ref{lem:B} that there exist a smooth map 
$\widetilde{\phi'} : S\cup_f W \to {\mathbb R}$ and an open neighborhood $V'(p(x))$ of $p(x)$ such that 
$\widetilde{\phi'}|_{V'(p(x))} = \phi'|_{V'(p(x))}$. Let $U(x)$ be the open subset $p^{-1}(V'(p(x)))$ of the domain of $p$; see Lemma \ref{lem:A}. Then, we have 
\[
\phi\circ p''|_{U(x)} = \phi\circ ((\widetilde{1\times r}))\circ p'|_{U(x)} = (\phi' \circ p')|_{U(x)} = \widetilde{\phi'}\circ p|_{U(x)}. 
\]
This implies that $p'' : V'(p(x)) \to k(W)$ is smooth. 

We show that $\overline{r}\circ p$ on $p^{-1}(\pi(S\coprod (\partial W\times [0, \varepsilon') ))$ factors through $k(S)$ with a suitable smooth map, where $ \varepsilon' = \frac{\varepsilon}{2}$. We write $V''$ for 
the open subset $\pi(S\coprod (\partial W\times [0, \varepsilon') )$. Define $\widetilde{p} : p^{-1}(V'') \to 
S\coprod (\partial W\times [0, \varepsilon'))$ by $\widetilde{p}(x) = p(x)$ if  $x \in p^{-1}(\pi(S))$ and 
$\widetilde{p}(x) = (w, t)$ for some $(w, t) \in \partial W\times [0, \varepsilon')$ with $p(x) = \pi(w, t)$ if $x \in p^{-1}(V'')\backslash p^{-1}(\pi(S))$.  

\begin{lem}\label{lem:smooth}
The composite $(1\coprod (f\circ r)) \circ \widetilde{p} :  p^{-1}(V'') \to k(S)$ is a plot of $k(S)$; that is, the map is smooth. 
\end{lem}

It follows that the composite $\pi\circ (1\coprod (f\circ r)) \circ \widetilde{p} : p^{-1}(V'') \to k(S)\cup_{k(f)}k(W)$ coincides 
with  the map $p \circ \overline{r}$.  Thus, Lemma \ref{lem:smooth} implies that $\overline{r}$ is smooth. 
\end{proof}

\begin{proof}[Proof of Lemma  \ref{lem:smooth}] By definition, it suffices to prove that, for any $\phi \in \C$, the composite $\phi \circ (1\coprod (f\circ r)) \circ \widetilde{p}$ is a smooth function on $p^{-1}(V'')$. By using the composite $f\circ r$, we extend the map $\phi$ to a smooth map $\widetilde{\phi} : S\cup_f (\partial W\times [0, \varepsilon')) \to {\mathbb R}$.
Since $S$ is closed in $S\cup_fW$ and $S\cup_f (\partial W\times [0, \varepsilon')$ is open in $S\cup_fW$, it follows from Lemma \ref{lem:B} that there exists a smooth extension $\Psi : S\cup_fW \to {\mathbb R}$ in $\C_{\text{ad}}^{(S, W)}$  
of $\widetilde{\phi}$ such that $\Psi|_S = \widetilde{\phi}|_S = \phi$. Moreover, it is readily seen that 
$\phi \circ (1\coprod (f\circ r)) \circ \widetilde{p} = \Psi\circ p$. The map $p$ is a plot of $k(S\cup_fW)$ and then  composite $\Psi\circ p$ is smooth. We have the result. 
\end{proof}

\begin{proof}[Proof of Theorem \ref{thm:p-stfd}]
We have to prove that I) $l \circ\overline{r} \simeq 1_{k(S\cup_fW)}$ and II) 
$\overline{r} \circ l \simeq 1_{k(S)\cup_{k(f)}k(W)}$. To this end, we construct smooth homotopies between them. Let $\lambda : {\mathbb R} \to {\mathbb R}$ be a smooth cut-off function with $\lambda(s)= 0$ for $s\leq 0$,  $\lambda(s)= 1$ for $s \geq 1$ and $0 < \lambda(s) < 1$ for $0< s < 1$. 
Define a  continuous homotopy $H : {\mathbb R} \times {\mathbb R} \to {\mathbb R}$ 
between $1_{\mathbb R}$ and $r$ mentioned above 
by $H(t, s) = (1-\lambda(s))r(t) + \lambda(s)u(t)$, where $u(t) = t$ for $t\geq 0$ and $u(t) = 0$ for $t \leq 0$. 
Then, the map $H$ gives a continuous homotopy $H' : W \times {\mathbb R} \to W$. Observe that ${H'}^{-1}(W\backslash \partial W)$ is an open subset of 
$W\backslash \partial W \times {\mathbb R}$. 
By the construction of the homotopy, we have 
\begin{lem}\label{lem:H}
The restriction $H' : {H'}^{-1}(W\backslash \partial W) \to W\backslash \partial W$ is a smooth map between the manifolds. 
\end{lem}

The homotopy $H'$ produces a map 
$\widetilde{H}  : (S\coprod W) \times {\mathbb R} \to S\coprod W$ with $\widetilde{H}(y, s) =y$ for $y \in S$, $s \in {\mathbb R}$, 
$\widetilde{H}(c(x, t), s) = c(x, t)$ for $t \geq \frac{3\varepsilon}{4}$, $s \in {\mathbb R}$ and $\widetilde{H}(c(x, t), s) = c(x, H'(t, s))$.  
Then, the homotopy $\widetilde{H}$ gives rise to a map $\overline{H}  : k(S\cup_f W) \times {\mathbb R} \to k(S\cup_f W)$.  

We show that the map $\overline{H}$ is smooth in $\mathsf{Diff}$. For any plot $p$ of $k(S\cup_fW) \times {\mathbb R}$, it will be proved that 
$\phi \circ \overline{H} \circ p : U \to k(S\cup_fW) \times {\mathbb R} \to  k(S\cup_fW) \to {\mathbb R}$ is a smooth function on $U$ 
for each $\phi \in \C_{\text{ad}}^{(S, W)}$. We represent $p$ as $(p', q)$ with plots $p'$ and $q$ of $k(S\cup_fW)$ and  ${\mathbb R}$, respectively. 

Observe that there exists $\delta < \varepsilon$ such that $\phi c(w, t) = \phi f(w)$ for $w \in \partial W$ and for $ s< \delta$. 
Let $x$ be an element in $p^{-1}(\overline{H}^{-1}(S\cup_f\partial W \times [0, \delta)))$. Then, 
we see that $\phi \circ \overline{H} \circ p = \phi\circ p'$ on an appropriate open neighborhood 
$U_{x, \phi}$ of $x$. 
Suppose that $x$ is in $p^{-1}(\overline{H}^{-1}(W\backslash \partial W))$. It follows that 
\[
\phi \circ \overline{H} \circ p = \phi_{W\backslash \partial W}\circ {H'}\circ (p', q) : 
U_x' \to  {H'}^{-1}(W\backslash \partial W) \to W\backslash \partial W \to {\mathbb R}
\]
on some neighborhood $U'_x$ of $x$ in $U$. 
Lemma \ref{lem:H} yields that the right-hand side map is smooth. We have the assertion I). 

In order to prove the assertion II), we use the same construction and notations as above. 
We define a map $(\overline{H})' : k(S)\cup_{k(f)} k(W)\times {\mathbb R} \to k(S)\cup_{k(f)} k(W)$ by 
$(\overline{H})' = \overline{H}$ as a map between the underlying sets. Observe that the diffeology of $k(S)\cup_{k(f)} k(W)$ is different from that of $k(S\cup_fW)$. 
It suffices to show that the map 
$(\overline{H})'$ is smooth in $\mathsf{Diff}$. For any plot $p=(p', q)$ on 
$k(S)\cup_{k(f)} k(W)\times {\mathbb R}$, by definition, we see that $p' : U_p \to k(S)\cup_{k(f)} k(W)$ locally factors through $k(S)$ or $k(W)$ as a plot.

If there exists a plot $p'' : U' \to k(S)$ for some $U' \subset U_p$ such that $\pi \circ p'' = p'$. Then, it follows that $(\overline{H})'\circ p = \pi \circ pr\circ (p'', q)$, where $pr : k(S) \times {\mathbb R} \to k(S)$ is the projection in the first factor. The projection is smooth. It turns out that $(\overline{H})'\circ p$ is smooth. 

We consider the case where $p'$ factors through $k(W)$ on an open neighborhood $U_x$ of an element $x \in U_p$. 
We have to prove that for $\phi \in 
\C_{W}$, the composite $\phi \circ \widetilde{H} \circ (p', q) :  U_r \to k(W) \times {\mathbb R} \to k(W) \to {\mathbb R}$ is a smooth function. 
For $z$ in $U_x$, suppose that $p'(z) \in c(\partial W\times [0, \delta_\phi))$, where $\delta_\phi < \varepsilon$ with $\phi c(w, t) = \phi f(w)$ for 
$w \in \partial W$ and for $ s< \delta_\phi$. 
Then, we see that $\phi \circ \widetilde{H} \circ (p', q) = \phi \circ p'$ on some open neighborhood of $z$, which is determined by $\phi$, in $U_x$. 
If $p'(z)$ is in ${H'}^{-1}(W\backslash \partial W)$, then it follows from Lemma \ref{lem:H} that the composite $\phi \circ \widetilde{H} \circ (p', q) = \phi \circ H' \circ (p', q): U_x \to {H'}^{-1}(W\backslash \partial W) \to  W\backslash \partial W\to {\mathbb R}$  is smooth 
for some open neighborhood of $z$.  This completes the proof. 
\end{proof}

We conclude this section by describing a homological property of a diffeological adjunction space. 

Let $M$ be a diffeological space. 
The homology $H_*(\mathbb{Z}S^D(M))$ of the simplicial set $S^D(M)$ is called the singular homology of $M$, denoted $H_*(M; \mathbb{Z})$. The singular homology $H_*(M, N; \mathbb{Z})$ of a relative space $(M, N)$ is defined by 
$H_*(\mathbb{Z}S^D(M)/\mathbb{Z}S^D(N))$.  

Since $W$ is a manifold and $\partial W$ is a submanifol of $W$, it follows that $k(\partial W)$ is a diffeological subspace of $k(W)$. Moreover, it is readily seen that 
for an inclusion $A \to X$ and a smooth map $f : A \to Y$, the space $X$ is a diffeological subspace of the adjunction space $X\cup_fY$. 

\begin{lem}\label{lem:Theta} Under the same notation as in Example \ref{ex:an_example2}, 
the smooth map $\widetilde{\Theta} : (k(W), k(\partial W)) \to (k(S)\cup_{k(f)} k(W), k(S))$ described before Proposition \ref{prop:well_attached} induces an isomorphism on the singular 
homology and thus on the singular de Rham cohomology. 
\end{lem}

\begin{proof} In order to prove the result, 
we use the excision theorem for the singular homology twice. We define the composite $\overline{r}\circ l :  (X, A) :=(k(S)\cup_{k(f)} k(W), k(S)\cup_{k(f)}k(\partial W \times [0, \e))) \to (k(S)\cup_{k(f)} k(W), k(S))$ of pairs of diffeological spaces with maps $l$ and $\overline{r}$ in Theorem \ref{thm:p-stfd} and Lemma \ref{lem:smoothness}. 
Then, the composite 
is a homotopy equivalence with the inclusion $i_1$ as the inverse. In fact, we have 
$(\overline{r}\circ l)\circ i_1 = id$ and $i_1\circ (\overline{r}\circ l) \simeq id$. The second homotopy is induced by the homotopy 
$\overline{H}$ in the proof of Theorem \ref{thm:p-stfd}. Since $l$ is smooth, it follows from Lemma \ref{lem:A} that 
$U:= k(S)\cup_{k(f)}k(\partial W \times [0, \delta))$ is an open subset of $k(S)\cup_{k(f)} k(W)$, where $\delta = \frac{\e}{2}$. 
The excision axiom for the singular homology of a diffeological space (\cite[Proposition 3.1]{Kihara}) enables us to deduce that the inclusion 
\[j : (X- U, A- U)=(W', k(\partial W \times [\delta, \e))) \to (X, A)
\] is an isomorphism on the singular homology.  

Define an increasing smooth function  $r_1 :  [0, \e) \to [0, \e)$ which satisfies the condition that $r_1(t) = \delta$ for 
$0 \leq t \leq \delta$ and $r_1(t) = t$ for $\delta' < t <\e$  for some $\delta'$ with  
$\delta < \delta' <  \e$. We may have a smooth map $\widetilde{1\times r_1} : W \to W'$ extending the map $1\times r_1 : \partial W\times [0, \e) \to 
\partial W\times [\delta, \e)$.  With the smooth map, we obtain a commutative square
\[
\xymatrix@C25pt@R20pt{
(W', k(\partial W\times [\delta, \e)))& (W' - V', k(\partial W\times [\delta, \e)) - V') \ar[l]^-{\text{excision}} \\
(W,  k(\partial W\times [0, \e))) \ar[u]^-{\widetilde{1\times r_1}} & (W - V, k(\partial W\times [0, \e)) - V),  \ar[l]^-{\text{excision}}  \ar[u]_{(\widetilde{1\times r_1)}|}
}
\]
where $V = \partial W \times [0, \delta')$ and $V' = \partial W \times [\delta, \delta')$. 
It is readily seen that the restriction $(\widetilde{1\times r_1})|$ is the identity map between the manifolds. Therefore, we see that  the map 
$\widetilde{1\times r_1}$ induces an isomorphism on the relative singular homology.  Moreover, it follows from the proof of Theorem \ref{thm:p-stfd} 
that the restriction 
$(\overline{r}\circ l)| : (W,  \partial W\times [0, \e)) \to (W,  \partial W)$ is a smooth homotopy equivalence with the inclusion  $i_2$ as a homotopy inverse. Furthermore, we have 
\[
\widetilde{\Theta} = (\overline{r} \circ l) \circ j \circ (\widetilde{1\times r_1}) \circ i_2. 
\]
It turns out that the map $\widetilde{\Theta}$ induces an isomorphism on the relative singular homology. The result \cite[Theorem 2.4]{K2020} implies that 
the $\widetilde{\Theta}$ gives rise to an isomorphism on the singular de Rham cohomology; see Remark \ref{rem:S^D}.
\end{proof}

\section{An algebraic model for the suspension of a manifold}\label{app:appB}

In this section, a commutative model for the unreduced suspension of a manifold is given. We begin by defining the interesting object in $\mathsf{Diff}$. 

Let $M$ be a closed manifold and $j_t : M \to M\times I$ the inclusion defined by $j_t(x) = (x, t)$ for $t= 0$ and $1$. 
Let $f : M \coprod M \to * \coprod *$ be the trivial map. 
Then, we define the {\it (unreduced) suspension} $\Sigma M$ of $M$ in $\mathsf{Diff}$ by 
the image of the functor $k : \mathsf{Stfd} \to \mathsf{Diff}$ of the adjunction space for smooth maps 
$
\xymatrix@C20pt@R20pt{
M\times I & M \coprod M \ar[l]_-{(j_0, j_1)} \ar[r]^-f & \text{$*$} \coprod \text{$*$} 
}
$
in $\mathsf{Stfd}$, 
which is a p-{\it stratifold}; see Section \ref{app:appA}, 
namely,  
\[
\Sigma M := k((M\times I)\cup_{M \coprod M} (*\coprod *)). 
\]
Theorem \ref{thm:p-stfd} yields  that $\Sigma M$ is smooth homotopy equivalent to 
$k(M\times I)\cup_{M\coprod M}(*\coprod *)$. 

We try to describe the Souriau--de Rham complex $\Omega^*(\Sigma M)$ in terms of the de Rham complex $\Omega^*(M)$ of the manifold $M$; 
see Proposition \ref{prop:formsOnSigma M} below. 
We first observe that the factor map $\alpha$ in Section \ref{sect:deRham} and the map $\Theta$ in the diagram (\ref{eq:Theta}) give rise to a quasi-isomorphism 
\[
\Theta \circ \alpha : \Omega^*(\Sigma M) \stackrel{\simeq}{\to} A(k(M\times I))\times_{A(M)\times A(M)}(A(*)\times A(*))=:A_{\Sigma M}^I, 
\]
where $A(\ )$ stands for the singular de Rham complex functor defined by 
$A(X):=\mathsf{Sets}^{\Delta^{op}} (S^D(X)_{\text{aff}}, (A_{DR})_\bullet)$ 
for a diffeological space $X$ and $S^D_n(X)_{\text{aff}}=\mathsf{Diff}({\mathbb A}^n, X)$; see Remark \ref{rem:S^D}.
The result follows from Propositions \ref{prop:well_attached},  \ref{prop:factor_map} and Lemma  \ref{lem:Theta}. We obtain a commutative model for 
the suspension $\Sigma M$. 

\begin{prop}\label{prop:suspensions} \text{\em (cf. \cite[Proposition 13.9]{FHT})} 
Let $M$ be a connected closed manifold. Define a commutative differential graded subalgebra 
${\mathcal C}_{\Sigma M}$ of $\Omega^*(M)\otimes \Omega^*(\R)$ by 
${\mathcal C}_{\Sigma M}:=(\Omega'(M)\otimes dt)\oplus \R$, where the cochain complex 
$\Omega'(M)$ is defined by $\Omega^{>0}(M) = \text{\em Im}\ \! d^0\oplus \Omega'(M)$. 
Then, there exists a quasi-isomorphism 
$\varphi : {\mathcal C}_{\Sigma M} \stackrel{\simeq}{\to} A_{\Sigma M}^I$. 
As a consequence, one has an isomorphism 
\[
\Upsilon : H^*(\Omega^*(\Sigma M)) \stackrel{\cong}{\to} (\widetilde{H^*}(M)\otimes dt)\oplus \R
\]
of graded algebras, where $\R$ denotes the subalgebra of $H^*(\Omega^*(\Sigma M))$ generated by the unit $1$. 
\end{prop}

\begin{proof}
We identify the de Rham complex $\Omega^*(M)$ and $\Omega^*(\R)$ with the usual de Rham complexes of the manifolds via the tautological maps in 
Remark \ref{rem:tautological_map}. Let $I=[0,1]$ be the diffeological subspace of the manifold $\R$. 
We define a map $\psi : \wedge(t, dt) \to A(I)$ by the composite 
$
\xymatrix@C20pt@R20pt{
 \wedge(t, dt) \ar[r]^-{\iota} &\Omega^*(\R) \ar[r]^-{j^*} & \Omega^*(I) \ar[r]^-{\alpha} &A(I)
 }
$
with the canonical inclusion $\iota$,  the inclusion  $j : I \to \mathbb{R}$ and the factor map $\alpha$ mentioned in Section \ref{sect:deRham}. 
 Then, we have a commutative diagram 
\[
\xymatrix@C28pt@R20pt{
A(k(M \times I)) \ar[r]^-{in^*} \ar[rd]^{(j^*_0, j^*_1)} & A(M\times \{0\} \coprod M\times \{1\}) & A(* \coprod *) \ar[l]_-{f^*}\\
A(M)\otimes \wedge(t, dt) \ar[u]^{p^* \cdot \psi} \ar[r]_{\nu:=(id\cdot \varepsilon_0, id\cdot \varepsilon_1) }& A(M) \times A(M) \ar[u]_{\cong}^\eta & \R \times \R, \ar[u]_{\simeq}  \ar[l]^-c
}
\]
where $p$ is the projection $k(M\times I) \to M$, $c$ is the canonical map, $in$ is the inclusion $M\times \{0\} \coprod M\times \{1\}\to M\times I$
and $\eta$ denotes the natural isomorphism. We observe that the vertical arrows are quasi-isomorphism 
and the map 
$\nu=(id\cdot \varepsilon_0, id\cdot \varepsilon_1)$ is an epimorphism. The map $in$ induces an injective map 
$S^D(M\times \{0\} \coprod M\times \{1\})_{\text{aff}} \to S^D(M \times I)_{\text{aff}}$. Moreover, the simplicial CDGA $(A_{DR}^*)_\bullet$ is extendable; see \cite[The comment before Corollary 3.5]{K2020}. 
Thus,  it follows from \cite[Proposition 10.4]{FHT} that $in^*$ is an epimorphism. Therefore, by virtue of \cite[Lemma 13.3]{FHT}, 
we see that the vertical maps yield a quasi-isomorphism 
\[
\widetilde{\varphi} : \widetilde{C}_{\Sigma M}:= (A(M)\otimes \wedge(t, dt))\times_{A(M)\times A(M)}(\R\times \R)\stackrel{\simeq}{\longrightarrow} A_{\Sigma M}^I. 
\]
We consider the homology long exact sequence associated with the short exact sequence
\[
\xymatrix@C15pt@R20pt{
0 \ar[r]  & \widetilde{C}_{\Sigma M} \ar[r] & (A(M)\otimes \wedge(t, dt))\times (\R\times \R) \ar[r]^-{\nu-c}  & A(M)\times A(M) \ar[r] & 0.
}
\]
It is immediate that 
$H^0(\nu-c) : H^0((A(M)\otimes \wedge(t, dt))\times (\R\times \R)) \to H^0( A(M)\times A(M))$ is an epimorphism. 
Moreover, we have $H^i(\nu-c)(w\otimes 1, 0)= (w, w)$ for $i \geq 1$. We write $H^i(\nu-c)H^i(A(M))\oplus H^i(A(M))$ for $H^i(A(M)\times A(M))$, where $i \geq 2$. 
Thus, it follows that 
$H^1(\widetilde{C}_{\Sigma M})= 0$ and that the connecting homomorphism $\delta$ in the long exact sequence induces an isomorphism $\delta : 0\oplus H^i(A(M)) \stackrel{\cong}{\to} H^{i+1}(\widetilde{C}_{\Sigma M})$ for $i\geq 1$.  The definition of $\delta$ enables us to deduce that $\delta(0, w)= (w\otimes dt, 0)$ on cohomology. It turns out that 
$\kappa : (\Omega'(M)\otimes dt )^* \to \widetilde{C}_{\Sigma M}$ defined by 
\begin{eqnarray}\label{eq:xi}
\kappa(w\otimes dt) = (\alpha(w)\otimes dt, 0)
\end{eqnarray}
is a quasi-isomorphism for $* \geq 2$. Define $\varphi$ by the composite $\widetilde{\varphi}\circ\kappa$. 
The isomorphism $\Upsilon$ in the assertion is given by $\Upsilon = H(\varphi)^{-1}\circ H(\Theta)\circ H(\alpha)$. 
We have the result. 
\end{proof}

We recall a smooth cut-off function $\rho : {\mathbb R} \to {\mathbb R}$ in the proof of Theorem \ref{thm:p-stfd}. 
Observe that the codomain is included in $I$. 
We see that a map  $1\times \lambda : M\times \R \to k(M\times I)$ is smooth and a smooth homotopy equivalence with the inverse $1\times (\lambda|_I)$. 

\begin{prop}\label{prop:formsOnSigma M}
There exists an isomorphism 
\[\Upsilon' : H^*(\Omega^*(\Sigma M)) \stackrel{\cong}{\to} 
H^*(\Omega^*(M\times \mathbb{R})\times_{\Omega^*(M)\times \Omega^*(M)}
\Omega^*(\ast)\times \Omega^*(\ast)).
\]
Here the pullback in the right-hand side is given by the inclusion $(j_0, j_1) : M \coprod M \to M\times \R$ defined by 
$j_t(x) = (x, t)$ for $t= 0$ and $1$. 
Moreover, for each element $\omega$ in $H^*(\Omega^*(\Sigma M))$ with $\deg \omega > 0$, there exists a cocycle 
$\beta \in \Omega^* (M)$ such that $\Upsilon'(\omega) = (\beta \cdot \lambda^*(dt), 0)$. 
\end{prop}

\begin{proof} Let $\Omega_{\Sigma M}^I$ and $\Omega_{\Sigma M}^\R$ denote the CDGA's 
$\Omega^*(k(M\times I))\times_{\Omega^*(M)\times \Omega^*(M)}\Omega^*(\ast)\times \Omega^*(\ast))$ and
$\Omega^*(M\times \R)\times_{\Omega^*(M)\times \Omega^*(M)}\Omega^*(\ast)\times \Omega^*(\ast))$, respectively. Let 
$A_{\Sigma M}^\R$ be the CDGA $A(M\times \R)\times_{A(M)\times A(M)}A(\ast)\times A(\ast))$. 
We consider a diagram
\begin{eqnarray}\label{eq:Omega_A}
\xymatrix@C15pt@R15pt{
H^*(\Omega^*(\Sigma M)) \ar[r]_-{\cong}^-\Upsilon \ar[d]_{H(\Theta)} & (\widetilde{H^*}(M)\otimes dt)\oplus \R \ar[d]^{H(\varphi)}_{\cong} \\
H(\Omega_{\Sigma M}^I) \ar[r]_{H(\alpha\times \alpha)} \ar[d]_{((1\times \lambda)\times 1)^*}& H(A_{\Sigma M}^I) \ar[d]^{((1\times \lambda)\times 1)^*}\\
H(\Omega_{\Sigma M}^\R) \ar[r]_{H(\alpha\times \alpha)=:\alpha'}& H(A_{\Sigma M}^\R),
 }
\end{eqnarray}
where $\Theta$ is the morphism of CDGA's obtained by replacing $A( \ )$ with the Souriau-de Rham complex $\Omega^*(\ )$ in the diagram (\ref{eq:Theta}).
The naturality of the factor map $\alpha$ with respect to smooth maps allows us to deduce that the lower square is commutative. 
The universality of the pullbak and the naturality of $\alpha$ gives the commutativity of the upper square. 
Since $1\times \lambda$ is a homotopy equivalence, it follows from \cite[Lemma 13.3]{FHT} 
that the right-hand side $((1\times \lambda)\times 1)^*$ is an isomorphism.

Consider a commutative triangle
\[
\xymatrix@C15pt@R20pt{
\Omega^*(M\times \R) \ar[r]^-{(j_0^*, j_1^*)}& \Omega^*(M)\times \Omega^*(M), \\
\Omega^*(M)\otimes \Omega^*(\R) \ar[u]^{p_1^*\cdot p_2^*} \ar[ru]_-{(id\cdot \iota_0^*, id\cdot \iota_1^*)} &
}
\]
where $p_i$ is the projection in the $i$th factor and $\iota_\tau : \{\tau\} \to \R$ is the inclusion for $\tau = 0$ and $1$. Then it is readily seen that 
$(id\cdot \iota_0^*, id\cdot \iota_1^*)$ is an epimorphism and hence so is $(j_0^*, j_1^*)$. Thus, by using Proposition \ref{prop:factor_map} and  \cite[Lemma 13.3]{FHT} again, we see that $\alpha'=H(\alpha\times \alpha)$ is an isomorphism.  
The composite $((1\times \lambda)\times 1)^* \circ H(\Theta)$ is the required map $\Upsilon$. 

For an element $\omega \in H^*(\Omega^*(\Sigma M))$ with $\deg \omega > 0$, we can write $\Upsilon(\omega) = \beta \otimes dt$ for some cocycle $\beta \in \Omega^*(M)$. It follows from a direct computation that $((1\times \lambda)\times 1)^*\circ H(\varphi)(\beta \otimes dt) =(\alpha(\beta \cdot \lambda^*(dt)), 0)$.  Moreover, we see that $\alpha'((\beta \cdot \lambda^*(dt), 0)) = (\alpha(\beta \cdot \lambda^*(dt)), 0)$. In view of the fact that $\alpha'$ is an isomorphism, 
the commutativity of the diagram (\ref{eq:Omega_A}) implies the latter assertion. 
\end{proof}



\end{document}